\documentclass[11pt,final]{article}

\usepackage{thmtools}
\usepackage[margin=1in]{geometry}

\usepackage{lmodern}
\usepackage[T1]{fontenc}

\usepackage[english]{babel} %Utilisation du Fran?ais comme langue principale, de l'anglais comme langue secondaire
\usepackage[latin1]{inputenc}

\usepackage{mathrsfs} % pour pouvoir utiliser \mathscr{F}
\usepackage{amsfonts}
\usepackage{marvosym} % Quelques symboles utiles, comme l'euro (\EUR)
\usepackage{amsmath,amssymb,amsthm}% Formules math\`ematiques

\usepackage{bm}

\usepackage[usenames, dvipsnames, svgnames]{xcolor}
\definecolor{vert}{RGB}{15,120,5}
\definecolor{gris}{RGB}{128,128,128}
\definecolor{bleu}{RGB}{0,50,150}
\definecolor{rouge}{RGB}{149,24,24}
\usepackage[hypertexnames=false, colorlinks = true,linkcolor=bleu,citecolor=vert]{hyperref}

\theoremstyle{plain}
\newtheorem{thm}{Theorem}[subsection]
\newtheorem{prop}[thm]{Proposition}

\newtheorem{lem}[thm]{Lemma}
\newtheorem{cor}[thm]{Corollary} 
\newtheorem{conject}[thm]{Conjecture}
\theoremstyle{definition}
\newtheorem{defi}[thm]{Definition}
\newtheorem{constr}[thm]{Construction}

\newtheorem{introthm}{Theorem}

\newtheorem{introprop}[introthm]{Proposition}

\theoremstyle{remark}

\newtheorem{rem}[thm]{Remark}

\newtheorem{recoll}[thm]{Recollection}

\numberwithin{equation}{thm}

\usepackage{zref-clever}
\zcsetup{noabbrev}
\zcsetup{cap}
\zcsetup{hyperref=true}
\zcsetup{nameinlink=true}
\newcommand{\cref}[1]{\zcref{#1}}

\newcommand{\zcrefname}[3]{%
  \zcRefTypeSetup{#1}{
    Name-sg = #2 ,
    name-sg = #2 ,
    Name-pl = #3 ,
    name-pl = #3 ,
  }
}

\zcrefname{lem}{Lemma}{Lemmas}
\zcrefname{defi}{Definition}{Definitions}
\zcrefname{cor}{Corollary}{Corollaries}
\zcrefname{introcor}{Corollary}{Corollaries}
\zcrefname{introthm}{Theorem}{Theorems}
\zcrefname{introprop}{Proposition}{Propositions}

\zcrefname{prop}{Proposition}{Propositions}
\zcrefname{conj}{Conjecture}{Conjectures}
\zcrefname{conject}{Conjecture}{Conjectures}
\zcrefname{thm}{Theorem}{Theorems}
\zcrefname{rmk}{Remark}{Remarks}
\zcrefname{section}{Section}{Sections}
\zcrefname{constr}{Construction}{Constructions}
\zcrefname{ex}{Example}{Examples}
\zcrefname{exs}{Examples}{Exampless}
\zcrefname{recoll}{Recollection}{Recollections}
\zcrefname{nota}{Notation}{Notations}

\theoremstyle{definition}

\theoremstyle{plain}

\usepackage{bbm}

\usepackage{epsfig}

\usepackage{tikz-cd} %Pour faire des diagrammes plus facilement

\usepackage{stmaryrd}
\usepackage{manfnt}
\usepackage{multicol}%Utilisation du multicolonage
\usepackage{array}%Utilisation avanc\`ee des tableaux
\usepackage{multirow}%Utilisation du multicolonnage dans les tableaux

\usepackage{fancyhdr} % En-t?tes personnalis\`es ; voir plus loin
\usepackage{fancybox} % Pour redimensionner ou tourner des blocs

\newcommand{\titre}{}
\newcommand{\auteur}{}

\usepackage{calligra}
\title{\titre}
\author{\auteur}

\usepackage{tikz}
\usetikzlibrary{calc, intersections}

\usepackage{adjustbox}

\numberwithin{equation}{subsubsection}

\newcommand{\C}{\mathbb{C}}
\newcommand{\R}{\mathbb{R}}
\newcommand{\Q}{\mathbb{Q}}
\newcommand{\Z}{\mathbb{Z}}
\newcommand{\N}{\mathbb{N}}

\newcommand{\Spec}{\operatorname{Spec}}
\newcommand{\A}{\mathbb{A}}

\newcommand{\HH}{\mathrm{H}}
\newcommand{\D}{\mathrm{D}}

\newcommand{\lcal}{\mathcal{L}}
\newcommand{\acal}{\mathcal{A}}

\newcommand{\bcal}{\mathcal{B}}
\newcommand{\ecal}{\mathcal{E}}
\newcommand{\fcal}{\mathcal{F}}
\newcommand{\scal}{\mathcal{S}}

\newcommand{\ccal}{\mathcal{C}}

\newcommand{\dcal}{\mathcal{D}}
\newcommand{\mcal}{\mathcal{M}}
\newcommand{\ncal}{\mathcal{N}}

\newcommand{\mc}{\mathcal}
\newcommand{\St}{\mathrm{St}}

\newcommand{\proet}{\mathrm{pro\acute{e}t}}
\newcommand{\an}{\mathrm{an}}

\newcommand{\bscr}{\mathscr{B}}

\newcommand{\dconset}{\D_\mathrm{cons}}
\newcommand{\dconspro}{\D_\mathrm{cons}}
\newcommand{\dconsan}{\D_\mathrm{cons}}
\newcommand{\dgman}{\D_\mathrm{gm}}

\newcommand{\Hom}{\mathrm{Hom}}

\newcommand{\Gm}{{\mathbb{G}_m}}

\newcommand{\F}{\mathbb{F}}

\newcommand{\hcal}{\mathcal{H}}

\newcommand{\Mp}{{\mathcal{M}_{\mathrm{perv}}}}

\DeclareMathOperator{\sHom}{\mathscr{H}\text{\kern -3pt {\calligra\large om}}\,}

\newcommand{\map}{\mathrm{map}}

\newcommand{\CAlg}{\mathrm{CAlg}}
\newcommand{\Perf}{\mathrm{Perf}}

\newcommand{\HHp}{{{}^{p}\HH}}

\newcommand{\DM}{\mathrm{DM}^\et}
\newcommand{\DN}{\mathrm{DN}}
\newcommand{\et}{\mathrm{\acute{e}t}}
\newcommand{\Mod}{\mathrm{Mod}}
\newcommand{\un}{\mathbbm{1}}
\newcommand{\Sh}{\mathrm{Sh}}

\newcommand{\Sch}{\mathrm{Sch}}
\newcommand{\op}{\mathrm{op}}

\newcommand{\nscr}{\mathscr{N}}

\newcommand{\catinfty}{\mathrm{Cat}_\infty}

\newcommand{\DMgm}{\mathrm{DM}_\mathrm{gm}^\et}
\newcommand{\DMc}{\mathrm{DM}_c^\et}
\newcommand{\DNgm}{\mathrm{DN}_\mathrm{gm}}

\newcommand{\DNc}{\mathrm{DN}_c}

\DeclareFontFamily{U}{BOONDOX-calo}{\skewchar\font=45 }
\DeclareFontShape{U}{BOONDOX-calo}{m}{n}{
  <-> s*[1.05] BOONDOX-r-calo}{}
\DeclareFontShape{U}{BOONDOX-calo}{b}{n}{
  <-> s*[1.05] BOONDOX-b-calo}{}
\DeclareMathAlphabet{\mathcalboondox}{U}{BOONDOX-calo}{m}{n}
\SetMathAlphabet{\mathcalboondox}{bold}{U}{BOONDOX-calo}{b}{n}
\DeclareMathAlphabet{\mathbcalboondox}{U}{BOONDOX-calo}{b}{n}

\newcommand{\colim}{\operatorname{\mathrm{colim}}}
\newcommand{\In}{\operatorname{\mathrm{Ind}}}

\newcommand{\heart}{\heartsuit}

\newcommand{\rar}{\to}
\newcommand{\PrL}{\mathrm{Pr}^\mathrm{L}}

\newcommand{\Sm}{\mathrm{Sm}}

\newcommand{\id}{\mathrm{Id}}
\usepackage{appendix}

\usepackage{comment}

\usepackage[colorinlistoftodos,obeyFinal]{todonotes}
\usepackage{xpatch}
\usepackage[explicit]{titlesec}

\def\thissectiontitle{}
\def\thissectionnumber{}
\def\thissubsectiontitle{}
\def\thissubsectionnumber{}

\newtoggle{todoSection}
\newtoggle{todoSubsection}

\titleformat{\section}
  {\normalfont\Large\bfseries\sffamily}
  {\thesection}
  {0.5em}
  {\gdef\thissectiontitle{#1}\gdef\thissectionnumber{\thesection}#1}

\titleformat{\subsection}
  {\normalfont\large\bfseries\sffamily}
  {\thesubsection}
  {0.5em}
  {\gdef\thissubsectiontitle{#1}\gdef\thissubsectionnumber{\thesubsection}#1}

  \pretocmd{\section}{\global\toggletrue{todoSection}}{}{}
  \pretocmd{\subsection}{\global\toggletrue{todoSubsection}}{}{}

\AtBeginDocument{%
  \xpretocmd{\todo}{%
    \iftoggle{todoSubsection}{
     \addtocontents{tdo}{\protect\contentsline{subsection}%
        {\protect\numberline{\thissubsectionnumber}{\thissubsectiontitle}}{}{} }
      \global\togglefalse{todoSubsection}
        }{}
    }{}{}%
  }
\AtBeginDocument{%
  \xpretocmd{\todo}{%
    \iftoggle{todoSection}{
     \addtocontents{tdo}{\protect\contentsline{section}%
        {\protect\numberline{\thissectionnumber}{\thissectiontitle}}{}{} }
      \global\togglefalse{todoSection}
        }{}
    }{}{}%
  }
\begin{document}

\font\myfont=cmr12 at 14.5pt
\title{{\myfont Nori motives (and mixed Hodge modules) with integral coefficients}}

\date{}
\author{Rapha\"el Ruimy, Swann Tubach}

\maketitle
\begin{abstract} 
  We construct abelian categories of integral Nori motivic sheaves over a scheme of characteristic zero. The first step is to study the presentable derived category of Nori motives over a field. Next we construct an algebra in \'etale motives such that modules over it afford a t-structure that restricts to constructible objects. This category of integral Nori motives has the six operations and arc-descent. We finish by providing analogous constructions and results for mixed Hodge modules on schemes over the reals.
\end{abstract}
% \listoftodos
% \todomedium{Il faut faire quelque chose a propos de "At the
% same time, the main novelties of the paper is somewhat unclear, and should be spelt out
% more explicitly."}
% \todolater{mettre à jour les ref}
\tableofcontents
\section*{Introduction}
A cohomology theory of algebraic varieties can be represented by a motivic ring spectrum, or motivic algebra, and vice versa. Finding the universal cohomology of algebraic varieties is therefore the same as finding a universal algebra in the category of \'etale motives. The most natural candidate is the unit object, which represents \'etale motivic cohomology. In characteristic zero, all known cohomology theories compare to the Betti motivic algebra, that is the algebra corresponding to the cohomology theory given by singular cohomology of the complex points. This algebra has the convenient property that modules over it afford a t-structure. In this paper, we define the Nori algebra which is universal among motivic algebras that compare to the Betti algebra and for which modules over them have a reasonable t-structure. 

The above description is derived in nature: the category of \'etale motives is constructed out of a derived category, and we what we call a cohomology theory is a functor
\[\mathrm{R}\Gamma\colon\mathrm{Sm}_k^\op\to\D(\acal)\] sending a smooth variety over a field $k$ to a complex of objects (in some abelian category $\acal$); this contains more information than all the $\HH^i$ of this complex. There is however another construction of a universal cohomology theory, that has a more abelian nature. This construction is due to Nori. We recall it before going into more details about our construction. 
\newline

\textbf{Motives over a field.}
We fix for now a subfield $k$ of $\C$. Nori's fundamental insight is that, as every known cohomology theory of $k$-varieties can be compared to singular cohomology, one should start from the abelian groups \[\HH^i_\mathrm{B}(X,Z):=\HH^i_\mathrm{sing}(X^\an,Z^\an)\] given by the singular cohomology of the analytification of any algebraic variety $X$ over $k$, relative to the analytification of a closed subvariety $Z\subseteq X$. The category $\mcal(k,\Z)$ of Nori motives over $k$ is then an abelian category containing an object $\HH^i(X,Z)$ which, in a sense that will be made precise, is the abelian group $\HH^i_\mathrm{B}(X,Z)$ endowed with the finest structure that is preserved by algebraic maps; this mysterious structure is in fact the action of an affine group scheme over $\Z$: the motivic Galois group $\mathrm{G}_\mathrm{mot}(k)$. 
 There is a quiver (or oriented graph) $\mathrm{Pairs}_k$, constructed out of the morphisms between the $\HH^i(X,Z)$ giving functoriality in $(X,Y)$,  boundaries for the long exact sequence of cohomology and expressing the fact that the Lefschetz object $\Z(-1):=\HH^2(\mathbb{P}^1_k,\emptyset)$ is invertible (see \cite[Definition 9.3.2]{MR3618276}). For Nori, a cohomology theory is a representation 
\[\HH^*\colon\mathrm{Pairs}_k \to \acal\] of the quiver of pairs, so that our $\HH^i_\mathrm{B}$ gives a cohomology theory $\HH^*_\mathrm{B}$. The abelian category of Nori motives is the universal representation of the quiver $\mathrm{Pairs}_k$ that factors the Betti representation in a faithful way: 
for every diagram 
\[\begin{tikzcd}
	{\mathrm{Pairs}_k} & {\acal} & {\mathrm{Ab}^\mathrm{ft}} \\
	& {\mcal(k,\Z)}
	\arrow["{\HH^*_\acal}", from=1-1, to=1-2]
	\arrow[from=1-1, to=1-3, bend left=30,"\HH^*_\mathrm{B}"{description}]
	\arrow["{\HH^*_\mathrm{univ}}"', from=1-1, to=2-2]
	\arrow[from=1-2, to=1-3,"\mathrm{f}_\acal"]
	\arrow[from=2-2, to=1-2,dashed,"R_\acal"]
	\arrow["{R_\mathrm{B}}"', from=2-2, to=1-3]
\end{tikzcd}\]
of full arrows, with $\mathrm{f}_\acal$ faithful and exact, there exists an essentially unique faithful exact functor $R_\acal$ fitting in the dashed arrow. In such a diagram, one can see $\acal$ as additional structure on singular cohomology of pairs, so that it is clear that $\mcal(k,\Z)$ provides the best possible structure.

This abelian approach is linked to the above derived approach. The $\infty$-category of geometric \'etale motives $\DMgm(k,\Z)$ with integral coefficients is made in such way that any enriched mixed Weil cohomology theory in the sense of Cisinski and D\'eglise's \cite{mixedweil} $\mathrm{R}\Gamma\colon\Sm_k^\op\to\D(\acal)$ factors through the canonical functor $h\colon\Sm_k^\op\to\DMgm(k,\Z)$ (which sends an smooth variety to its cohomological motive). The following result provides the link between Voevodsky and Nori motives.
\begin{introthm}[Nori, Harrer, Choudhury, Gallauer \cite{harrerComparisonCategoriesMotives2016,MR3649230}] There exists a Nori realisation functor 
\[\rho_\mathrm{N}\colon\DMgm(k,\Z)\to\D^b(\mcal(k,\Z))\] such that the composition with the Betti realisation of Nori motives $\D^b(\mcal(k,\Z))\xrightarrow{R_\mathrm{B}}\D^b(\mathrm{Ab}^\mathrm{fr})$ gives the derived Betti cohomology theory $C^*_\mathrm{sing}\colon \Sm_k^\op \to\D^b(\mathrm{Ab}^\mathrm{ft})$.
\end{introthm}
  This shows that any cohomology theory \emph{\`a la} Nori induces a derived cohomology theory. Conversely there exists a representation $\mathrm{Pairs}_k\to\DMgm(k,\Z)$, therefore any derived cohomology theory induces a representation of the quiver of pairs. Giving a formal meaning to these considerations led to a new universal property of Nori motives (\cite[Corollary 10.1.7]{MR3618276}):
for every diagram 
\[\begin{tikzcd}
	{\DMgm(k,\Z)} & {\acal} & {\mathrm{Ab}^\mathrm{ft}} \\
	& {\mcal(k,\Z)}
	\arrow["{\HH^*_\acal}", from=1-1, to=1-2]
	\arrow[from=1-1, to=1-3, bend left=30,"\HH^0\circ\rho_\mathrm{B}"{description}]
	\arrow["{\HH^0\circ\rho_\mathrm{N}}"', from=1-1, to=2-2]
	\arrow[from=1-2, to=1-3,"\mathrm{f}_\acal"]
	\arrow[from=2-2, to=1-2,dashed,"R_\acal"]
	\arrow["{R_\mathrm{B}}"', from=2-2, to=1-3]
\end{tikzcd}\]
of full arrows with $\HH^*_\acal$ a cohomological functor and $\mathrm{f}_\acal$ a faithful exact functor, there exists a unique faithful exact functor $R_\acal$ filling the dashed arrow.
\newline

\textbf{Rational Nori motivic sheaves.} This new point of view opened a way to construct relative versions of Nori motives. Indeed, why stop at abelian groups? Sheaves of abelian groups are so nice. Using the method of Nori with a diagram of pairs, Arapura (\cite{MR2995668}) and Ivorra (\cite{MR3723805}) constructed abelian categories of Nori motivic sheaves that are the universal recipient of a cohomology theory of relative pairs of $k$-varieties modelled on constructible sheaves for Arapura, and on perverse sheaves for Ivorra. Although we will not go into details, let us say that these categories have simple universal properties but lack a critical structure that one expects when dealing with sheaves: a six functors formalism. Indeed, for a theory of sheaves $\acal(X)$ to be efficient, notably for some reductions, it is desirable to dispose of the six operations: those  are adjunctions 
\[ f^* \colon 
\D^b(\acal(X))\leftrightarrows \D^b(\acal(Y))\colon f_*,\]
\[ f_! \colon 
\D^b(\acal(Y))\leftrightarrows \D^b(\acal(X))\colon f^!,\] for $f\colon Y\to X$ a morphism of varieties, and adjunctions 
\[ -\otimes F \colon 
\D^b(\acal(X))\leftrightarrows \D^b(\acal(X))\colon \underline{\mathrm{Hom}}(F,-)\] for $F$ belonging to $\D^b(\acal(X))$. The abelian categories constructed by Arapura and Ivorra do not have all of the six operations, because it is very hard to exhibit some morphisms of the quivers of pairs that would induce the above operations. This is where using the new universal property of Nori motives is critical: working with rational coefficients, if instead of taking a diagram of pairs, one takes the category $\DMgm(X,\Q)$ of relative \'etale motives and uses the perverse Betti realisation, one can construct an abelian category $\Mp(X,\Q)$ of perverse motivic sheaves which has the new universal property. This is what Ivorra and S. Morel do in their paper \cite{ivorraFourOperationsPerverse2022}. As promised, using the full strength of the functoriality of $\DMgm(X,\Q)$ (that affords the six operations), this allows for the construction of the six operations:

\begin{introthm}[Ivorra, Morel, Terenzi \cite{ivorraFourOperationsPerverse2022,terenziTensorStructurePerverse2024a}]
  Over quasi-projective $k$-varieties, the derived category $\D^b(\Mp(-,\Q))$ of perverse motives is endowed with the six operations. 
\end{introthm}

At that point, the two sheaf theories $\DMgm(-,\Q)$ and $\D^b(\Mp(-,\Q))$ did not talk to each other because of the way they were constructed: the first one is very derived, and its universal property is highly $\infty$-categorical, whereas the second one is based on abelian constructions, and had poor $\infty$-categorical properties. In \cite{SwannRealisation}, the second author of the present paper proved that the constructions of Ivorra-Morel and Terenzi had canonical $\infty$-categorical lifts, so that the universal property of the functor $\DMgm(-,\Q)$ proven by Drew and Gallauer in \cite{MR4560376} would provide a realisation functor 
\[\rho_\mathrm{N}\colon \DMgm(-,\Q)\to\D^b(\Mp(-,\Q))\] that commutes with the six operations. This realisation functor induces an adjunction 
\[ \rho_\mathrm{N} \colon 
\DM(-,\Q)\leftrightarrows \In\D^b(\Mp(-,\Q))\colon \rho_*^\mathrm{N}\] between the associated functors of presentable $\infty$-categories. As $\rho_\mathrm{N}$ is monoidal, its right adjoint $\rho_*^\mathrm{N}$ is lax-monoidal and thus creates a commutative algebra object $\nscr_{(-),\Q}:=\rho_*^\mathrm{N}\Q_{(-)}$ in the presentable $\infty$-category of \'etale motives $\DM(-,\Q)$.
Hence, for every variety $X$ the functor $\rho_\mathrm{N}$ factors through a refined realisation 
\[\DM(X,\Q)\xrightarrow{-\otimes\nscr_{X,\Q}}\mathrm{Mod}_{\nscr_{X,\Q}}(\DM(X,\Q))\xrightarrow{\widetilde{\rho_\mathrm{N}}}\In\D^b(\Mp(X,\Q)).\] 
\begin{introthm}[\cite{SwannRealisation}]
Let $\DN(-,\Q):=\mathrm{Mod}_{\nscr_{(-),\Q}}(\DM(-,\Q))$. The induced functor \[\widetilde{\rho_\mathrm{N}}\colon \DN(-,\Q)\to\In\D^b(\Mp(-,\Q))\] is an equivalence. In particular, for every $k$-variety of structural map $p_X\colon X\to\Spec(k)$ the canonical map $p_X^*\nscr_{k,\Q}\to\nscr_{X,\Q}$ is an equivalence.

Furthermore, if $K$ is a field extension of $k$ included in $\C$, the map $(\nscr_{k,\Q})|_K \to \nscr_{K,\Q}$ is an equivalence. 
\end{introthm}
The whole sheaf theory $\In\D^b(\Mp(-,\Q))$ can therefore be recovered from the algebra $\nscr_{k,\Q}\in \CAlg(\DM(k,\Q))$.
The last part of the theorem allows to extend the functor $\DN(-,\Q)$ to any $\Q$-scheme $p_X\colon X\to \Spec(\Q)$ by setting $\nscr_{X,\Q}=p_X^*\nscr_{\Spec(\Q),\Q}$. Hence, everything can be recovered from the algebra $\nscr_{\Spec(\Q),\Q}$ and the functor $\DN(-,\Q)$ can be enhanced into a six functors formalism thanks to the results of \cite{MR4466640}.
\newline 

\textbf{Main results of this paper.}
The goal of this paper is to extend the above picture to integral coefficients. We first provide a new construction of integral Nori motivic sheaves. The idea is to first define a presentable stable $\infty$-category by taking modules over some algebra (the \textit{Nori algebra}) and then prove that this category is endowed with an ordinary t-structure and a perverse t-structure which provides abelian categories of Nori motivic sheaves. The advantage of our construction over that of \cite{MR2995668} and \cite{MR3723805} is that our categories are endowed with a full six functors formalism which still remains open for Arapura and Ivorra's construction. Furthermore, they are canonically endowed with a realisation functor from \'etale motives. The existence of the six functors implies one of the main results of \cite{MR2178703} (see \cref{Leray}). We are also able to prove remarkable descent properties (namely arc-descent) on our category of Nori motives. Finally, we can compare our stable version of Nori motives with the derived categories of both of its hearts; for the ordinary heart we obtain an equivalence and we formulate a conjecture on the perverse heart which implies the equivalence between its bounded derived category and geometric Nori motives as defined below.
\newline 

\textbf{The Nori algebra.}
The construction of Ivorra and Morel could only work with rational coefficients: their construction notably relies on the fact that the Verdier duality functor is perverse t-exact and can be expressed as a derived functor on the category of perverse sheaves using \cite{MR0923133}. This is not at all the case with integral coefficients as this was pointed out in \cite[Section 3.3]{MR0751966}. They also use in a crucial manner the motivic unipotent nearby cycle functor constructed by Ayoub, and this functor is behaves badly with integral coefficients \cite[Remarque 3.4.14]{AyoubPartII}. However, the following idea to extend their work is very natural and is the starting point of the present paper: note first that in any $\Z$-linear stable presentable $\infty$-category $\ccal$ and any object $A$ in $\ccal$, there is a Hasse fracture square 
\[\begin{tikzcd}
	A & {\lim_nA\otimes_\Z\Z/n\Z} \\
	{A\otimes_\Z\Q} & {(\lim_nA\otimes_\Z\Z/n\Z)\otimes_\Z\Q}
	\arrow[from=1-1, to=1-2]
	\arrow[from=1-1, to=2-1]
	\arrow["\lrcorner"{anchor=center, pos=0.125}, draw=none, from=1-1, to=2-2]
	\arrow[from=1-2, to=2-2]
	\arrow[from=2-1, to=2-2]
\end{tikzcd}\] where the limits are taken over the poset $\N^*$ of nonzero integers ordered by divisibility (so that if $A=\Z\in\D(\Z)$, the top right object is just the profinite completion $\hat{\Z}$ of the integers).
If one manages to construct a map 
\[\nscr_{k,\Q}\to(\lim_n\Z/n\Z)\otimes_\Z \Q\] in $\DM(k,\Q)$, then one could use a Hasse fracture square to \emph{define} an algebra $\nscr_k:= \nscr_{k,\Z}$ in $\DM(k,\Z)$ such that for every $k$-variety $p_X\colon X\to\Spec(k)$ the $\infty$-category $\DN(X,\Z)$ of modules over $p_X^*\nscr_k$ in $\DM(X,\Z)$ is entitled the name of integral Nori motives over $X$.

To do so, one solution is to study integral Nori motives over a field. As explained above, we have an abelian category $\mcal(k,\Z)$ of Nori motives over $k$. Thus there are two naive possibilities for a presentable $\infty$-category of Nori motives over $k$: one could take the unbounded derived category $\D(\In\mcal(k,\Z))$ of the indisation of Nori motives, or take the indisation $\In\D^b(\mcal(k,\Z))$ of the bounded derived category of Nori motives. These two categories do not give the correct presentable category as explained in \cref{WhyDnComplicated}.

Nonetheless, we first consider the $\infty$-category $\D(\In \mcal(k,\Z))$, as it is easier to get a grasp of. In fact as mentionned above, the work of Choudhury-Gallauer (\cref{realFieldDInd}) gives us a symmetric monoidal functor 
$$\rho'_\mathrm{N}\colon\DM(k,\Z)\to\D(\In(\mcal(k,\Z))).$$ Thus we may take a right adjoint $F$ of this functor, and \emph{define} $\nscr_k$ as the value of $F$ at the unit object of $\D(\In(\mcal(k,\Z)))$. Using rigidity for Nori motives over a field (as proven by Nori, \cref{rigab}), and the fact that $\D^b(\mcal(k,\Z))$ is a full subcategory if $\D(\In(\mcal(k,\Z)))$, we manage to prove (\cref{algRat} and \cref{rigAlgebre}) that the algebra $\nscr_k$ does fit in a Hasse fracture square
\[\begin{tikzcd}
	\nscr_k & {\lim_n\Z/n\Z} \\
	{\nscr_{k,\Q}} & {(\lim_n\Z/n\Z)\otimes_\Z\Q}
	\arrow[from=1-1, to=1-2]
	\arrow[from=1-1, to=2-1]
	\arrow["\lrcorner"{anchor=center, pos=0.125}, draw=none, from=1-1, to=2-2]
	\arrow[from=1-2, to=2-2]
	\arrow[from=2-1, to=2-2]
\end{tikzcd}.\] 
As a consequence, we have:

\begin{introthm}[\cref{geomfields} and \cref{indfieldsalg}]\label{geomfieldsIntro}
  Let $k$ be a subfield of the complex numbers. Consider the factorisation of $\rho'_\mathrm{N}$ through the modules category 
  \[\overline{\rho'_\mathrm{N}}\colon\Mod_{\nscr_k}(\DM(k,\Z))\to\D(\In\mcal(k,\Z)).\]
  This functor is fully faithful when restricted to the thick subcategory $\DNgm(k,\Z)$ generated by the $M_k(X)(i)\otimes \nscr_k$ for $X$ smooth over $k$ and induces an equivalence between the latter and $\D^b(\mcal(k,\Z))$.

  Moreover, if $f\colon \Spec(K)\to\Spec(k)$ is a morphism of fields, then the canonical comparisaon map \[f^*\nscr_k\to\nscr_K\] is an equivalence.
\end{introthm}

\textbf{Integral Nori motivic sheaves.}
For any qcqs $\Q$-scheme with a structural map $p_X\colon X\to \Spec(\Q)$, we now set $\nscr_X:=p_X^*\nscr_{\Spec(\Q)}$. If $\Lambda$ is a commutative ring, we set $\DN(X,\Lambda)=\Mod_{\nscr_X \otimes_\Z \Lambda}(\DM(X,\Z))$; we have a six functors formalism on $\DN(-,\Lambda)$ by \cref{SixFunBig}. We have therefore produced a presentable stable $\infty$-category of Nori motives. Furthermore, we have a Nori realisation functor \[\rho_\mathrm{N}\colon \DM(X,\Lambda)\to \DN(X,\Lambda)\] which is simply the functor $-\otimes \nscr_X$. It is compatible with the six functors in the following sense:

\begin{introthm}[\cref{compatibilite sif}]
  Let $\Lambda$ be a commutative ring. Over qcqs $\Q$-schemes of finite dimension, the Nori realisation $\rho_\mathrm{N}\colon \DM(-,\Lambda)\to \DN(-,\Lambda)$ is compatible with the operations of type $f^*$, $f_*$, $f_!$, $f^!$, $\otimes$ and $\underline{\Hom}(M,-)$ for $M$ of geometric origin.
\end{introthm}

Then next steps will aim to produce a small abelian category from $\DN(X,\Lambda)$. 
To that end, we study the subcategory $\DNgm(X,\Lambda)$ made of Nori motives of geometric origin (\cref{DNgm}). In \cite{DMpdf}, we proved a categorical criterion for being geometric which is \emph{constructibility} as well as continuity and descent properties of the category $\DMgm(-,\Lambda)$ of geometric \'etale motives. The same properties hold for Nori motives, more precisely:

\begin{introthm}[\cref{megathmcons}]
  Let $\Lambda$ be a commutative ring. The functor $\DNgm(-,\Lambda)$ is finitary. Furthermore, 
  \begin{enumerate}    
    \item If $X$ is a qcqs $\Lambda$-finite (see \cref{LambdaFin}) $\Q$-scheme, then $\DNgm(X,\Lambda)=\DN(X,\Lambda)^\omega$.
    \item If the ring $\Lambda$ is regular, a Nori motive $M$ over a qcqs $\Q$-scheme of finite dimension is of geometric origin if and only if it is \emph{constructible}, that is for any open affine $U\subseteq X$, there is a finite stratification $U_i\subseteq U$  made of constructible locally closed subschemes such that each $M|_{U_i}$ is dualisable. 
  \end{enumerate}
\end{introthm}

In addition, the six functors from $\DN(-,\Lambda)$ induce a six-functors formalism on $\DNgm(-,\Lambda)$ over quasi-excellent $\Q$-schemes of finite dimension (\cref{sifgm}) and we also have a Verdier duality functor under some mild resolution of singularities assumptions (see \cref{Verdier}).
\newline

\textbf{t-structures on geometric motives.} 
The next step is to define a t-structure on $\DNgm(X,\Z)$ in order to define abelian categories of Nori motivic sheaves. By analogy with constructible sheaves, two t-structures should exist: the ordinary t-structure and the perverse t-structure. The first we construct is the ordinary one. Assume that $X$ is of finite-type over a subfield $k$ of $\C$ and recall that following Saito (\cite[4.6]{MR1047415}), the second author defined in \cite{SwannRealisation} a t-structure on $\DNgm(X,\Q)(=\D^b(\Mp(X,\Q)))$ such that the Betti realisation functor \[\DNgm(X,\Q)\to \dconsan(X^\an,\Q)\] to the category of constructible analytic complexes is t-exact when the latter is endowed with its ordinary t-structure. When $k$ is algebraically closed, this can be used to define a t-structure integrally. Indeed, we prove that a geometric motive with coefficients $\Z$ is exactly a geometric motive with coefficients $\Q$ whose Betti realisation has an underlying $\Z$-lattice, and those have a t-structure (\cref{PBBetti}).

In general, we use a gluing procedure due to Vaish (\cite{MR4109490}) to derive a t-structure on a scheme $X$ from the knowledge of a t-structure on fields. This is \cref{tord} and this gives an ordinary t-structure for integral Nori motives over any qcqs $\Q$-scheme of finite dimension. We denote by $\mcal_\mathrm{ord}(X,\Z)$ the heart of this t-structure.

In \cite{MR1940678}, Nori showed that $\dconsan(X^\an,\Q)$ is the derived category of constructible sheaves. Following his ideas and using the adaptation of his argument to the setting of rational Nori motives in \cite{SwannRealisation}, we show that $\DNgm$ is the derived category of ordinary Nori motives.

\begin{introthm}[{\cref{DbNat}}]
  Let $X$ be a qcqs $\Q$-scheme of finite dimension. Then, the canonical functor 
  \[\D^b(\mcal_\mathrm{ord}(X,\Z))\to \DNgm(X,\Z)\] is an equivalence.
\end{introthm}

We also give two other applications of the ordinary t-structure: firstly, the Leray spectral sequence with integral coefficients is motivic \cref{Leray} and secondly, we show that $\DNgm(-,\Z)$ satisfies arc-hyperdescent over qcqs $\Q$-schemes of finite dimension \cref{ArcDesc}. We therefore expect it to be the case for $\DMgm(-,\Z)$ but we do not know how to prove it.

From the ordinary t-structure, we can construct a perverse t-structure using the usual gluing method of \cite{MR0751966}.
\begin{introprop}[\cref{ExistPerv}]Let $X$ be an excellent $\Q$-scheme of finite dimension endowed with a dimension function. Then, there is a t-structure on $\DNgm(X,\Z)$ such that the $\ell$-adic realisation functor \[R_\ell\colon \DNgm(X,\Z)\to \dconspro(X_\proet,\Z_\ell)\] is t-exact when the right-hand side is endowed with its perverse t-structure. 
\end{introprop}
We let $\Mp(X,\Z)$ be the category of perverse Nori motives which is defined as the heart of the above t-structure. As in the case of perverse sheaves, the Verdier duality functor exchanges $\Mp(X,\Z)$ with another abelian category $\mcal_\mathrm{perv}^+(X,\Z)$ (see \cref{perv+}). We can also define a motivic intermediate extension functor that is sent to its $\ell$-adic or analytic counterpart through the appropriate realisation functors (\cref{j!*}). This gives a motivic intersection complex $\mathrm{IC}_{X,\Z}$ in $\mcal_\mathrm{perv}(X,\Z)$. We prove that the classical pairings for intersection homology can be lifted to Nori motives. We expect that Wildeshaus' intersection complex of Shimura varieties in \'etale motives \cite{MR2953419,MR4050012} realises to the complex we defined though the functor from \'etale motives to Nori motives. %(from the unicity properties of the intermediate extension and of the object constructed by Wildeshaus this should be easy). 
Finally, any t-exactness property of the six functors which is true after realisation carries over to Nori motives, in particular, we can prove the hard Lefschetz theorem for Nori motives. 

We then study whether $\DNgm(X,\Z)$ is the derived category of perverse Nori motives. For this the classical argument of Beilinson \cite{MR0923133} fails because unipotent nearby cycles are badly behaved with integral coefficients. Under the assumption that the abelian category $\Mp(X,\Z)$ has \emph{enough torsion-free objects} in the sense of \cref{TorsOb}, it is easy to prove that the canonical functor 
\[\D^b(\Mp(X,\Z))\to\DNgm(X,\Z)\] is an equivalence by reducing to torsion coefficients (this is Beilinson's result) and to rational coefficients. Unfortunately, we were note able to prove that there are enough torsion-free objects, but we propose a strategy (\cref{DbPerv}) that could work to achieve this goal. This is linked to a comparison with the approach of Ivorra quoted above, which has enough torsion-free objects because the generators of the category are very explicit.

The perverse t-structure on Nori motives can also be used to prove new cases of the Artin vanishing theorem for Artin motives proved by the first named author of the present paper in \cite{ruimyAbelianCategoriesArtin2023}. Indeed, we prove in \cref{0mot} that the functor from \'etale motives to Nori motives restricts to an equivalence on Artin motives, and we use that %the perverse t-structure restricts to smooth Artin motives over a regular basis 
to show that if $f\colon X\to S$ is a quasi-finite affine map of characteristic zero schemes, with regular source, the functor 
\[f_!\colon \mathrm{DM}_{\Q-c}^{\et,sm0}(X,\Z)\to \mathrm{DM}^{\et,0}(S,\Z)\] is t-exact (see \cref{ArtinVanishing}). Here the left-hand side consists of $\Q$-constructible smooth Artin motives and the right-hand side consist of Artin motives; %(that is, smooth Artin motives that are constructible smooth Artin motives when rationalised), 
both sides are endowed with their perverse homotopy t-structure (these notions were defined by the first author in \cite[Definition 3.2.1.1]{ruimyAbelianCategoriesArtin2023}).
\newline 

\textbf{Integral Mixed Hodge modules.}
The story we told also works for mixed Hodge modules, with simpler arguments. , \cite{SwannRealisation} and \cite{SwannMHM} provide an $\infty$-categorical enhancement of mixed Hodge modules, and a description of the objects of geometric origin as modules in $\DM$ over an algebra, as introduced by Drew in \cite{drewMotivicHodgeModules2018}. Thus we can define a category of mixed Hodge modules with integral coefficients as the category of mixed Hodge modules with rational coefficients having an underlying lattice. We can also define these notions for schemes over the reals (by taking Galois homotopy fixed points), and prove that objects of geometric origin are also modules over an algebra in \'etale motives (\cref{theoMHMmod}). Over $\Spec(\C)$, this category contains the bounded derived category of mixed Hodge structures with integral coefficients (\cref{MHMpt}). We can also prove that the stable $\infty$-category we define is indeed the bounded derived category of ordinary mixed modules with integral coefficients (\cref{Dbmhm}). Finally, we also prove in \cref{VMHS} that (admissible) variations of mixed Hodge structures with integral coefficients can be seen as integral mixed Hodge modules. 
\newline

\textbf{Conjectural picture.}
One of the goal of this paper was to clarify how the existence of a motivic t-structure with rational coefficients would imply the existence of a motivic t-structure with integral coefficients. 
Recall that following \cite{MR2953406}, by motivic t-structure on $\DMgm(k,\Q)$ we mean a t-structure compatible with the monoidal structure and such that the $\ell$-adic realisation functor
\begin{equation}
  \label{motivic t-structure}
  \rho_\ell\colon\DMgm(k,\Q)\to\D^b(\mathrm{Rep}_{\Q_\ell}(\mathrm{Gal}(\overline{k}/k)))\tag{1}
\end{equation}
is t-exact.
In \cite[Theorem 6.22]{SwannRealisation} the second author proved that if a motivic t-structure existed over schemes for rational motives, then \'etale motives and the indisation of the bounded derived category of perverse Nori motives would be equivalent. Using rigidity, we show in \cref{ifstdconjthengood} that under the same assumption, \'etale motives with integral coefficients are also equivalent to Nori motives, so that they also afford a t-structure. 
Moreover, we show that the motives killed by the Betti realisation are the only obstruction for $\DN(k,\Q)$ to be the derived category of its heart (\cref{DIndQ}).
This confirms a suggestion of Ayoub in \cite[Remark 2.13]{MR3751289}, and implies that the following analogue of \cite[Conjecture 2.8]{MR3751289} holds for Nori motives:
\begin{introprop}[\cref{conjAyoubDN}]
  Let $k$ be a subfield of $\C$. Let $M\in\DNgm(k,\Q)$ be a geometric Nori motive and let $\fcal\in\DN(k,\Q)$ be killed by the Betti realisation. Then the group 
  \[\mathrm{Hom}_{\DN(k,\Q)}(\fcal,M)=0\] vanishes.
\end{introprop}

\section*{Notations and conventions} 
We freely use the language of higher categories and higher algebra of \cite{MR2522659,lurieHigherAlgebra2022,lurieSpectralAlgebraicGeometry}. We will denote by $\catinfty$ the $\infty$-category of small $\infty$-categories, $\PrL$ (resp. $\PrL_\St$) the $\infty$-category of presentable (resp. stable presentable) $\infty$-category, $\catinfty^\mathrm{perf}$ the $\infty$-category of stable and idempotent complete small $\infty$-categories, and $\Pr^{\mathrm{L},\omega}_\St$ the $\infty$-category of compactly generated stable $\infty$-categories. Those four $\infty$-categories are endowed with the Lurie tensor product, and the indisation functor $\In$ is an equivalence between $\catinfty^\mathrm{perf}$ and $\Pr^{\mathrm{L},\omega}_\St$, of inverse the functor sending a compactly generated $\infty$-category $\ccal$ to its full subcategory $\ccal^\omega$ of compact objects.
 
If $\ccal$ is a symmetric monoidal $\infty$-category and $A\in\mathrm{CAlg}(\ccal)$ is a commutative algebra object in $\ccal$, we denote by $\Mod_A(\ccal)$ the $\infty$-category of $A$-modules in $\ccal$. If $\ccal = \mathrm{Sp}$ is the $\infty$-category of spectra we omit $\mathrm{Sp}$ from the notation: $\mathrm{Mod}_A$. If $\ccal$ is a stable $\infty$-category, and $a,b\in\ccal$, 
we denote by $\map_\ccal(a,b)\in\mathrm{Sp}$ the mapping spectrum between $a$ and $b$. If $\Lambda$ is an ordinary commutative ring, the abelian category of $\Lambda$-modules is denoted by $\Lambda$-$\mathrm{mod}$.

All schemes we consider are of characteristic zero, so that they have a map to $\Spec(\Q)$. If $X$ is a scheme, we denote by $\mathrm{Sch}_X$ the category of $X$-schemes, by $\mathrm{Sch}_X^\mathrm{qcqs}$ the category of quasi-compact quasi-separated (which we will write qcqs in the rest of the paper) $X$-schemes, by $\Sch^\mathrm{ft}_X$ the category of $X$-schemes of finite type and by $\Sm_X$ the category of smooth separated $X$-schemes. 

We will use $\infty$-categories of \'etale motives: if $X$ is a scheme and $\Lambda$ is a commutative ring, we denote by $\DM(X,\Lambda)$ the $\infty$-category of \'etale motives with coefficients in $\Lambda$: to construct it, one first considers $\A^1$-invariant \'etale hypersheaves on the category $\mathrm{Sm}_X$ with values in $\mathrm{Mod}_\Lambda$, and then one Tate-stabilise (see \cite{MR3205601,MR3477640,rob}). The thick subcategory generated by the Tate twisted image of the Yoneda functor \[M_X\colon\Sm_X\to\DM(X,\Lambda)\] is $\DMgm(X,\Lambda)$ the $\infty$-category of motives of geometric origin (that is, it is the smallest category of $\DM(X,\Lambda)$ closed under finite limits and colimits and by retracts that contains the $M_X(Y)(n)$ for $Y\in\mathrm{Sm}_X$ and $n\in\Z$).

The following table lists the main players in this article, indicating their first appearance for reference.
\begin{center}
\begin{tabular}{ccc}
Nori motives over a field & $\mcal(k,\Z)$ & \cref{sect:Norimot} \\
Algebra representing Nori motivic cohomology & $\mathscr{N}_k$ & \cref{defi:NoriAlgebra} \\
Integral Nori motives & $\DN(X,\Lambda)$ &  \cref{defi:DNbig} \\
Nori motives of geometric origin & $\DNgm(X,\Lambda)$  & \cref{DNgm}\\
$\ell$-adic realisation of Nori motives & $R_\ell$ & \cref{LAdicSmall}\\
Betti realisation of Nori motives & $R_\mathrm{B}$ & \cref{Betti real 6FF}\\
Abelian category of ordinary Nori motives & $\mcal_\mathrm{ord}(X,\Z)$ & \cref{defi:Mord}\\
Abelian category of perverse Nori motives & $\Mp(X,\Z)$ & \cref{defi:Mperv}\\
Universal category of perverse Nori motives & $\mcal_\mathrm{univ}(X,\Z)$ & \cref{constru:MunivPerv} \\
$\infty$-category of integral mixed Hodge modules & $\D_\mathrm{H}(X,\Z)$ & \cref{defDH}
\end{tabular}
\end{center}
%\todoeasy{J ai ajoute ca pour le point (1)}
\subsection*{Acknowledgements}

We thank very warmly Fr\'ed\'eric D\'eglise and Sophie Morel for their constant support. We thank Martin Gallauer for a discussion about \cref{thm:TrueCG}. We would also like to thank Joseph Ayoub, Robin Carlier, Denis-Charles Cisinski, Rune Haugseng, Maxime Ramzi and Luca Terenzi for answering some of our naive questions on $\infty$-categories, Nori motives and \'etale motives. We also thank the anonymous referee for their comments. Finally, we would like to thank Fabrizio Andreatta, Luca Barbieri-Viale, Federico Binda, Niels Feld, Jan Nagel, Riccardo Pengo, Juan Diego Rojas and Alberto Vezzani for their interest in this work and some useful conversations.

We were funded by the ANR HQDIAG. RR was funded by the \emph{Prin 2022: The arithmetic of motives and L-functions} and by the Universit\'e Grenoble Alpes. ST acknowledges support by the \'E.N.S de Lyon for his PhD fundings and the European Research Council (ERC) under
the European Union's Horizon 2020 research and innovation programme ``EMOTIVE'' (grant agreement
no. 101170066).

\section{Preliminaries.}
\subsection{Computing right adjoints}
Recall that an $\infty$-category $\ccal$ is called \emph{rigid} if it is a  stably symmetric monoidal $\infty$-category in which every object is dualisable.
\begin{prop}\label{BlackMagic}
  Let $\ecal_c$ be a rigid stable $\infty$-category and denote by $\ecal$ its indisation. Let $\mcal,\ncal$ be $\ecal$-modules in $\PrL$ and let $\acal$ be an $\ecal$-algebra in $\PrL$ which is also the indisation of a rigid $\infty$-category. Assume given an adjunction 
  \[F\colon \mathcal{M} \leftrightarrows \mathcal{N}\colon G\] in which the right adjoint is co-continuous and the left adjoint is $\ecal$-linear. Then the right adjoint of 
  \[F\otimes_\ecal \mathrm{Id}_\acal \colon \mcal\otimes_\ecal \acal\to\ncal\otimes_\ecal\acal\] is given by 
  $G\otimes_\ecal \mathrm{Id}_\acal$. 
\end{prop}
\begin{proof}
  Denote by $\mathrm{Pr}^\mathrm{L}(\ecal)$ the $\infty$-category $\mathrm{Mod}_\ecal(\mathrm{Pr}^\mathrm{L})$ of modules in $\mathrm{Pr}^\mathrm{L}$ over $\ecal$. By \cite[Section 4.4]{MR3607274} this $\infty$-category has a natural $(\infty,2)$-categorical enhancement $\mathbf{Pr}^\mathrm{L}(\ecal)$ such that the map $\ecal\to\acal$ of Ind-rigid categories induces an $2$-functor 
  \[-\otimes_\ecal \acal\colon \mathbf{Pr}^\mathrm{L}(\ecal) \to \mathbf{Pr}^\mathrm{L}(\acal).\]
  Now, the adjunction 
  \[f^*\colon \mcal \leftrightarrows \ncal \colon f_*\] is a priori only an adjunction internal to $\catinfty$. However, as $f^*$ is $\ecal$-linear, the functor $f_*$ has an $\ecal$-lax-linear structure for free (see \cite[Examples 7.3.2.8 and 7.3.2.9]{lurieHigherAlgebra2022}). This lax structure is in fact strong by \cite[Proposition 4.9 (3)]{MR3607274}, and as $f_*$ is also a left adjoint, it is a map in $\mathrm{Pr}^\mathrm{L}(\ecal)$, so that by \cite[Proposition 5.2]{MR4695687} the adjunction $(f^*,f_*)$ is an internal adjunction in $\mathbf{Pr}^\mathrm{L}(\ecal)$. As internal adjunctions are preserved by $2$-functors (because an internal adjunction in $\mathbf{Pr}^\mathrm{L}(\ecal)$ is just a $2$-functor $\mathfrak{adj}\to \mathbf{Pr}^\mathrm{L}(\ecal)$ from the `walking adjunction category', see \cite[Section 4]{MR4367222}), the image $(f^*\otimes_\ecal \mathrm{Id}_\acal,f_*\otimes_\ecal\mathrm{Id}_\acal)$ of $(f_*,f^*)$ by $-\otimes_\ecal\acal$ is again an internal adjunction. Thus, the right adjoint of $f^*\otimes_\ecal \mathrm{Id}_\acal$ is indeed $f_*\otimes_\ecal \mathrm{Id}_\acal$.

\end{proof}

\subsection{Lemmas on t-structures and bounded derived categories}
For t-structures we will always use cohomological conventions as in \cite[Section 1.3]{MR0751966}.
\begin{lem}[Lemma 3.2.18 of \cite{MR4061978}]
  \label{LimTStructures}
  Let $(\mc{C}_i)_i$ be a projective system of stable (resp. stable presentable) $\infty$-categories which admits a limit $\ccal$ in $\catinfty$ (resp. $\PrL_\mathrm{St}$). 
  Assume that for any index $i$, the stable $\infty$-category $C_i$ is endowed with a t-structure and such that all transition functors are t-exact. 
  Then,  \[\ccal^{\leqslant 0}:=\{M \mid \forall i, p^i(M)\in \mc{C}_i^{\leqslant 0}\}\] 
  \[\ccal^{\geqslant 0}:=\{M \mid \forall i, p^i(M)\in \mc{C}_i^{\geqslant 0}\}\]
 defines a t-structure on $\mc{C}$ such that each projection $p^i\colon \ccal\to\ccal_i$ is t-exact. 
\end{lem}

\begin{lem}
  \label{coLimTStructures}
  Let $(\mc{C}_i)_i$ be a filtered inductive system of stable $\infty$-categories which admits a colimit $\ccal$ in $\catinfty$. 
  Assume that for any index $i$, the stable $\infty$-category $\ccal_i$ is endowed with a t-structure and that all transition functors are t-exact. 
  Then,  \[\ccal^{\leqslant 0}:=\colim  \ccal_i^{\leqslant 0}\] 
  \[\ccal^{\geqslant 0}:=\colim  \ccal_i^{\geqslant 0}\]
 defines a t-structure on $\mc{C}$ such that each $f_i\colon \ccal_i\to\ccal$ is t-exact. 
\end{lem}
\begin{proof}
  Filtered colimits in $\catinfty$ are computed in the category of simplicial sets. Thus, we can compute the mapping spaces as a colimit, and then as taking homotopy groups commutes with filtered colimits we obtain the orthogonality. The stability under shift is obvious, and the existence of triangles comes from the existence of triangle at some level of the colimit.
\end{proof}

If $\acal$ is an abelian category, we denote by $\D^b(\acal)$ its bounded derived $\infty$-category as constructed in \cite[Section 7.4]{bunkeControlledObjectsLeftexact2019}.
\begin{lem}
  \label{colimDb}
  Let $\acal = \colim_i \acal_i$ be a filtered colimit of small abelian categories with exact transitions. 
  Then the canonical functor 
  \[\colim_i \D^b(\acal_i) \to\D^b(\acal)\] is an equivalence.
\end{lem}
\begin{proof}
  Denote by \[F\colon \colim_i \D^b(\acal_i) \to\D^b(\acal)\] the canonical functor induced by the functors $\D^b(\acal_i)\to\D^b(\acal)$. By \cref{coLimTStructures} the $\infty$-category $\colim_i\D^b(\acal_i)$ has a t-structure, of heart $\colim_i\acal_i = \acal$. Thus, by the universal property of $\D^b(\acal)$ proved in \cite[Corollary 7.4.12]{bunkeControlledObjectsLeftexact2019}, there exists a canonical functor 
  \[G\colon \D^b(\acal)\to\colim_i\D^b(\acal_i)\] determined by its restriction to $\acal$. The restriction of $F\circ G$ to $\acal$ is the identity of $\acal$, thus $F\circ G \simeq \mathrm{Id}_{\D^b(\acal)}$. Conversely, the functor $G\circ F$ is determined by its composition with the natural functors $\iota_j\colon \D^b(\acal_j)\to\colim_i\D^b(\acal_i)$ which are themselves determined by their restriction to $\acal_j$. The restriction to $\acal_j$ of $G\circ F\circ \iota_j$ lands in \[\acal_j\subset \colim_i\acal_i\subset\colim_i\D^b(\acal_i)\] and is just the canonical functor $\acal_j\to\acal$, thus we have that $G\circ F\simeq \mathrm{Id}_{\colim_i\D^b(\acal_i)}$.
\end{proof}
\begin{defi}\label{TorsOb}
  Let $\acal$ be a small abelian category and let $n$ be an integer. An object $A$ of $\acal$ is 
\begin{enumerate}
  \item torsion-free if for all nonzero integers $m\in\Z\setminus\{0\}$ the map $A\xrightarrow{\times m} A$ is a monomorphism.
  \item of $n$-torsion if the map $A\xrightarrow{\times n}A$ is the zero map.
\end{enumerate}
We denote by $\acal[n]$ the full subcategory of $\acal$ that consists of $n$-torsion objects. 

Finally, we say that $\acal$ has \emph{enough torsion-free objects} if for any object $A$ in $\acal$, there exists an epimorphism $B\twoheadrightarrow A$ with $B$ a torsion-free object. 
\end{defi}
For $A$ a commutative ring, we denote by $\mathrm{Perf}_A$ the $\infty$-category of perfect complexes of $A$-modules. It consists of dualisable objects of $\Mod_A$.
\begin{prop}
   \label{torsderive}
  Let $\acal$ be a small abelian category and let $n$ be a positive integer. Assume that $\acal$ has enough torsion-free objects. 
  Then the left derived functor \[-\otimes_\Z^\mathrm{L} \Z/n\Z\colon \D^b(\acal)\to\D^b(\acal[n])\] of the functor 
  $-\otimes_\Z \Z/n\Z \colon A\mapsto A/nA$ exists, factors through the canonical functor 
  \[\D^b(\acal)\to\D^b(\acal)\otimes_{\mathrm{Perf}_\Z}\mathrm{Perf}_{\Z/n\Z}\] and induces an fully faithful functor 
  \[\D^b(\acal)\otimes_{\mathrm{Perf}_\Z}\mathrm{Perf}_{\Z/n\Z}\to \D^b(\acal[n]).\]
\end{prop}
\begin{proof}
  The proof goes with several steps.

  \underline{Step 1:} The functor \[-\otimes_\Z \Z/n\Z\colon \acal\to\acal[n]\] is the left adjoint of the inclusion functor \[\iota\colon\acal[n]\to\acal,\] and we can derive this adjunction to obtain an adjunction \[-\otimes_\Z^\mathrm{L} \Z/n\Z \colon 
        \D^b(\acal)\leftrightarrows \D^b(\acal[n])\colon \iota.\]
    Indeed, the adjunction induces an adjunction of functors between the categories of bounded complexes $\mathrm{Ch}^b(\acal)$ and $\mathrm{Ch}^b(\acal[n])$ and we can then use Cisinski's formalism of derived functors: we endow $\mathrm{Ch}^b(\acal[n])$ with the trivial structure of category with equivalences and fibrations (so that equivalences are quasi-isomorphisms and all maps are fibrations), and $\mathrm{Ch}^b(\acal)$ with the structure of a category with equivalences and cofibrations such that equivalences are quasi-isomorphisms and cofibrations are monomorphism with term wise torsion-free cokernels (this is indeed a structure of category with weak equivalences and cofibrations because there are enough torsion-free objects and that they are acyclic for $-\otimes_\Z\Z/n$). By \cite[Theorem 7.5.30]{MR3931682}, using that $\D^b(\acal)\simeq\mathrm{Ch}^b(\acal)[W_{qiso}^{-1}]$, we obtain the existence of the adjunction on the derived categories.

    \underline{Step 2:} The functor $-\otimes^\mathrm{L}_\Z\Z/n\Z$ factors trough $\D^b(\acal)\otimes_{\mathrm{Perf}_\Z}\mathrm{Perf}_{\Z/n\Z}$. This is because of the universal property of the tensor product of idempotent complete categories: it is computed as the compact objects of the tensor product between indisation, and the latter has a universal property by definition of Lurie's tensor product of presentable $\infty$-categories. As we have a $\mathrm{Perf}_\Z$-linear functor $\D^b(\acal)\to\D^b(\acal[n])$ and $ \D^b(\acal[n])$ is $\mathrm{Perf}_{\Z/n\Z}$-linear, this step is clear.

    \underline{Step 3:} The functor \[\alpha\colon \D^b(\acal)\otimes_{\mathrm{Perf}_\Z}\mathrm{Perf}_{\Z/n\Z}\to\D^b(\acal[n])\] is fully faithful. 
    Given two complexes $K,L$ in $\D^b(\acal)$ and two complexes $P,Q$ in $\mathrm{Perf}_{\Z/n\Z}$, we have \[\map_{\D^b(\acal)\otimes_{\mathrm{Perf}_\Z}\mathrm{Perf}_{\Z/n\Z}}(K\boxtimes P,L\boxtimes Q)\simeq \map_{\D^b(\acal)}(K,L)\otimes_\Z \map_{\mathrm{Perf}_{\Z/n\Z}}(P,Q),\] as it is proven in \cite[Proposition 3.5.5]{hoRevisitingMixedGeometry2023}, where $K\boxtimes P$ is the image of $(K,P)$ under the canonical functor \[\D^b(\acal)\times\mathrm{Perf}_{\Z/n\Z} \to\D^b(\acal)\otimes_{\mathrm{Perf}_\Z}\mathrm{Perf}_{\Z/n\Z}.\] To prove that $\alpha$ is fully faithful, by d\'evissage it suffices to prove that it is fully faithful on objects of the form $A\boxtimes \Z/n\Z$ with $A$ an object of $\acal$. The image of such an object in $\D^b(\acal[n])$ is $A\otimes^\mathrm{L}_\Z\Z/n\Z$. Thus, given the above formula for the mapping spectra in the tensor product, we have to show that for two objects $A$ and $B$ of $\acal$, the map 
      \[\map_{\D^b(\acal)}(A,B)\otimes_\Z\Z/n \to \map_{\D^b(\acal[n])}(A\otimes^\mathrm{L}_\Z\Z/n\Z,B\otimes^\mathrm{L}_\Z\Z/n\Z)\] is an equivalence. This is true because 
      \begin{align*}
        \map_{\D^b(\acal)}(A,B)\otimes_\Z\Z/n & \simeq  \mathrm{cofib}(\map_{\D^b(\acal)}(A,B)\xrightarrow{\times n}\map_{\D^b(\acal)}(A,B)) \\
        &\simeq  \map_{\D^b(\acal)}(A,\mathrm{cofib}(B\xrightarrow{\times n}B))\\
        & \overset{(*)}{\simeq}  \map_{\D^b(\acal)}(A,\iota(B\otimes^\mathrm{L}_\Z\Z/n\Z)) \\
        & \simeq  \map_{\D^b(\acal[n])}(a\otimes_\Z^\mathrm{L}\Z/n\Z,b\otimes_\Z^\mathrm{L}\Z/n\Z),
      \end{align*} where $(*)$ can be checked by taking a torsion-free resolution of $b$. 
This finishes the proof.
\end{proof}
\begin{rem}
  The above functor in proposition is not essentially surjective in general, as it can be seen by taking $\acal = \mathrm{Ab}^\mathrm{ft}$ the abelian category of finitely generated abelian groups, where we get the difference between $\mathrm{Perf}_{\Z/n\Z}$ and $\D^b(\Z/n\Z\text{-}\mathrm{mod}^\mathrm{ft})$, where $\Z/n\Z\text{-}\mathrm{mod}^\mathrm{ft}$ is the abelian category of finite type $\Z/n\Z$-modules.
\end{rem}

\subsection{Change of coefficients}
For $\Lambda$ a commutative ring, we denote by $\PrL_\Lambda:=\Mod_{\Mod_\Lambda}(\PrL)$ the $\infty$-category of $\Lambda$-linear presentable $\infty$-categories and by $\catinfty^\Lambda:=\Mod_{\Perf_\Lambda}(\catinfty^\mathrm{perf})$ the $\infty$-category of $\Lambda$-linear small idempotent-complete $\infty$-categories.
\begin{lem}
  \label{LemmeMagique}
  Let $F \colon \ccal\to \dcal$ be a map in $\CAlg\left(\PrL_\Z\right)$ and let $G$ be the right adjoint to $F$.
  Assume that for any object $N$ of $\dcal$, the natural map \[G(N)\otimes_\Z \Q \to G(N\otimes_\Z \Q)\] is an equivalence. If $A$ is a commutative algebra in $\Mod_\Z$, denote by $F_A$ the functor $F\otimes_{\Mod_\Z} \Mod_A$.
  
  Then, the functor $F$ is fully faithful (\emph{resp}. an equivalence) if and only if the functors $F_\Q$ and $F_{\Z/n\Z}$ are fully faithful (\emph{resp}. an equivalence) for any nonzero integer $n$. 
\end{lem}
\begin{proof}
If $F$ is fully faithful, then \cite[Lemma 2.14]{haineNonabelianBasechangeBasechange2022} ensures that $F_A $ is for any commutative algebra in $\Mod_\Z$. The same assertion about equivalences is obvious. This proves the "only if" part of the statement.

We now prove the converse. 
Assume first that the functors $F_\Q$ and $F_{\Z/n\Z}$ are fully faithful. Let $M$ be an object of $\mc{C}$. We want to prove that the natural map \[M\to GF(M)\] is an equivalence. This can be checked after tensoring with $\Q$ and with $\Z/n\Z$. Since tensoring with $\Z/n\Z$ commutes with any exact functor, the map $M/n\to GF(M)/n$ being an equivalence amounts to the full faithfulness of $F_{\Z/n\Z}$.

On the other hand, our assumption on $G$ implies that the right adjoint of $F_\Q$ sends an object of the form $N \otimes \Q$ to $G(N\otimes \Q)$. Therefore, the full faithfulness of $F_\Q$ implies that the map \[M\otimes_\Z \Q\to GF(M)\otimes_\Z \Q\] is an equivalence. Hence, the functor $F$ is fully faithful.

Assume now further that $F_\Q$ and $F_{\Z/n\Z}$ are essentially surjective. Let $N$ be an object of $\dcal$. Then, we have an exact triangle \[N\to N\otimes_\Z \Q \to \colim_n \  N\otimes_\Z \Z/n\Z;\] the essential surjectivity of $F_\Q$ and $F_{\Z/n\Z}$ imply that $N\otimes_\Z \Q$ and $N\otimes_\Z \Z/n\Z$ belong to the image of $F$. As $F$ is fully faithful, the latter is closed under colimits and finite limits, so that $N$ belongs to it. This finishes the proof
\end{proof}

\begin{rem}
  An obvious case where the previous \cref{LemmeMagique} applies is when $F\colon\ccal\to \dcal$ is a functor between compactly generated presentable $\mathrm{Mod}_\Z$-linear $\infty$-categories that preserves compact objects. Indeed, in that case the right adjoint to $F$ commutes with colimits hence also with \[(\mathrm{Id}\to-\otimes_\Z \Q) = \colim_n(\mathrm{Id}\xrightarrow{n}\mathrm{Id}).\] However, we will apply this lemma in a case where $\ccal$ is \emph{not} compactly generated.
\end{rem}

We will also need to consider torsion objects in an $\infty$-category.
\begin{defi}
  Let $\mc{C}$ be an object of $\CAlg(\PrL_\Z)$ or $\CAlg(\mathrm{Cat}_\infty^\Z)$ and let $p$ be a prime number. An object $M$ is 
  \begin{enumerate}
    \item torsion when $M\otimes_\Z \Q=0$.
    \item $p$-torsion when $M\otimes_\Z \Z[1/p]=0$.
    \item $p$-complete when the canonical map $M\to \lim M/p^n$ is an equivalence.
  \end{enumerate}
  We denote by $\mc{C}_\mathrm{tors}$ (resp. $\mc{C}_{p-\mathrm{tors}}$, resp. $\mc{C}^\wedge_p$) the full subcategory of $\mc{C}$ made of its torsion (resp. $p$-torsion, resp. $p$-complete) objects.
\end{defi}

  In the case when $\mc{C}$ belongs to $\CAlg(\PrL_\Z)$, the inclusion $\mc{C}^\wedge_p \subseteq \mc{C}$ has a left adjoint $(-)^\wedge_p$ given by $M\mapsto \lim M/p^n$ called the $p$-completion. Furthermore, the map $(-)^\wedge_p$ induces an equivalence \[\mc{C}_{p-\mathrm{tors}}\to \mc{C}^\wedge_p\] with an inverse given by $M \mapsto M \otimes \Z(p^\infty)$ where $\Z(p^\infty)\subseteq \Q/\Z$ denotes the Pr\"ufer $p$-group.
  
  Given a map $F\colon \mc{C}\to \mc{D}$ in $\CAlg(\PrL_\Z)$, we can also construct its $p$-completion $F^\wedge_p\colon \mc{C}^\wedge_p\to \mc{D}^\wedge_p$ which sends $M=\lim M/p^n$ to $\lim F(M)/p^n$.

\section{Derived categories of Nori motives over a field.}
\subsection{Reminders on Nori motives}
\label{sect:Norimot}
Let $k$ be a subfield of $\C$. Nori constructed a symmetric monoidal abelian category $\mcal(k,\Z)$ that affords the following universal property (see \cite[Section 9 and Corollary 10.1.6]{MR3618276}):
there is a cohomological functor 
\[\HH^0_\mathrm{N}\colon \DMgm(k,\Z)\to \mcal(k,\Z)\] factoring the $\HH^0$ of the Betti realisation
\[\rho_\mathrm{B}\colon\DMgm(k,\Z)\to\Mod_\Z\]
of \cite{MR2602027} which is universal for this property: any cohomological functor 
\[\DMgm(k,\Z)\to\ccal\] to a $\Z$-linear abelian category $\ccal$, such that there exists a faithful exact functor $\ccal\to\Z\text{-}\mathrm{mod}$ giving a factorisation of the $\HH^0$ of the Betti realisation of \'etale motives, will factor uniquely through $\HH^0_\mathrm{N}$, in a way compatible with the various Betti realisations.

The definition also work with rational coefficients. In this case, Nori motives form a tannakian category:
\begin{thm}[Nori, {\cite[Theorem 9.1.5]{MR3618276}}]
  \label{NoriTannaka}
  Let $k$ be a subfield of $\C$. The Betti realisation 
  \[\mcal(k,\Q)\to\mathrm{Vect}^{\mathrm{fd}}_\Q\] of Nori motives is a fiber functor of Tannakian categories. In particular, there exists a Hopf algebra $\mathrm{H}_\mathrm{N}\in \mathrm{Vect}_\Q$ such that the Betti realisation factors as an equivalence \[\mcal(k,\Q)\simeq \mathrm{coMod}_{\mathrm{H}_\mathrm{N}}(\mathrm{Vect}^\mathrm{fd}_\Q)\]
  where $\mathrm{coMod}_{\mathrm{H}_\mathrm{N}}(\mathrm{Vect}_\Q^\mathrm{fd})$ is the abelian categories of comodules over $\mathrm{H}_\mathrm{N}$ in the category $\mathrm{Vect}_\Q^\mathrm{fd}$ of finite dimensional $\Q$-vector spaces.
\end{thm}

We will need a result about torsion Nori motives. As Artin motives embed fully faithfully in $\DMgm(k,\Z)$ and the Betti realisation on those is conservative and t-exact, the functor 
\[\HH^0_\mathrm{N}:\DMgm(k,\Z)\to\mcal(k,\Z)\] induces a faithful exact functor 
\begin{equation}\label{artinMotive}\iota \colon \mathrm{Rep}(\mathrm{Gal}(\bar{k}/k),\Z)\to \mcal(k,\Z)\end{equation}
where $\mathrm{Rep}(\mathrm{Gal}(\bar{k}/k),\Z)$ is the abelian category of $\Z$-linear continuous representations of $\mathrm{Gal}(\bar{k}/k)$ which are finitely generated as $\Z$-modules. 

\begin{thm}[Nori \cite{fakhruddinNotesNoriLectures2000}]
  \label{rigab}
The functor $\iota$ defined above induces an equivalence
\[\iota \colon \mathrm{Rep}(\mathrm{Gal}(\bar{k}/k),\Z)_{\mathrm{tors}}\xrightarrow{\sim
} \mcal(k,\Z)_{\mathrm{tors}} \]
on the full subcategories of torsion objects.
\end{thm}
\begin{proof}
  The proof can be found in \cite[Theorem 6.1]{fakhruddinNotesNoriLectures2000}, but see also the proof of \cref{compTors} which is a generalisation.
\end{proof}

\subsection{Realisation functor to the derived category of ind Nori motives.}

Let $k$ be a subfield of $\C$. In this section, we recall how to construct a realisation functor 
\[\rho'_\mathrm{N}\colon\DM(k,\Z)\to \D(\In\mcal(k,\Z)),\] compatible with the Betti realisation. 
Its construction goes back to Nori and Beilinson's Basic Lemma, and we will use the description given in \cite{MR3649230}.

\begin{constr}
By \cite[Proposition 7.1 and its proof]{MR3649230}, there is 
  a lax symmetric monoidal $1$-functor 
  \[\mathrm{SmAff}_k\to \mathrm{Ch}^b(\mcal(k,\Z)^\mathrm{eff})\]
  from the category of affine $k$-schemes to the category of bounded complexes of effective Nori motives (\cite[Definition 9.1.3]{MR3618276}).
  We can compose this functor with the 
  functor \[\mathrm{Ch}^b(\mcal(k,\Z)^\mathrm{eff})\to\mathrm{Ch}( \In\mcal(k,\Z))\] 
  given by the map $\mcal(k,\Z)^\mathrm{eff}\to\mcal(k,\Z)$. 
  The construction roughly goes as follows: 
  for each smooth affine variety $X$, one can construct (using Nori's Basic Lemma) a finite \emph{cellular filtration} $\fcal$ on $X$ by closed subschemes which computes the cohomology of $X$
  (this is similar to how the cellular filtration of a CW-complex computes its cohomology). 
  This gives a bounded complex $C_\fcal^*(X)$ in $\mathrm{Ch}^b(\mcal(k,\Z))$ whose Betti realisation computes the singular cohomology of the variety of the complex points of $X$. 
  To make a functor out of it, one has to make choices to have cellular filtrations compatible with morphisms; 
  to that end, one considers the colimit $C^*(X)$ (in $\mathrm{Ch}( \In\mcal(k,\Z))$) over all possible cellular filtrations.
  Moreover, the composition $R_\mathrm{B}\circ C^*$ with the Betti realisation 
  functor $R_\mathrm{B}$ is naturally isomorphic (as a monoidal functor) 
  to the functor 
  $C^*_\mathrm{sing}\colon X\mapsto C^*_\mathrm{sing}(X^\mathrm{an})$
   sending a $k$-variety to the complex of singular 
  chains over the analytic variety of its complex points (this is \cite[Corollary 6.4]{MR3649230}).
  Notice finally that the functor $C^*$ sends \'etale $k$-schemes to objects of $\mcal(k,\Z)$ seen as complexes placed in degree $0$ (where it is given by the functor $\iota$ of \cref{artinMotive}).

  We now use Blumberg, Gepner and Tabuda's \cite[Proposition 3.5]{MR3209352}. 
  First notice that
  \[C^*\colon\mathrm{SmAff}_k\to\mathrm{Ch}( \In\mcal(k,\Z))\] is a lax 
  symmetric monoidal functor between symmetric monoidal $1$-categories. 
  Moreover, by construction of the functor $C^*$, we have $C^*(k)
  =\Z_k$ 
  and for any affine $k$-varieties $X_1$ and $X_2$, the canonical map 
  \[C^*(X_1)\otimes C^*(X_2)\to C^*(X_1\times X_2)\] 
  is a quasi-isomorphism, because this is true when $C^*$ is replaced by 
  $C^*_\mathrm{sing}$. 
  Setting the weak equivalences on $\mathrm{SmAff}_k$ to be the isomorphisms of 
  schemes, and the class of weak equivalences $W_{\mathrm{qiso}}$ in 
  $\mathrm{Ch}( \In\mcal(k,\Z))$ to be the class of quasi-isomorphisms, 
  the functor $C^*$ preserves weak equivalences. Therefore, the 
  proposition quoted above yields
  a symmetric monoidal structure on $\D( \In\mcal(k,\Z))$ and a symmetric 
  monoidal functor 
  \[\mathrm{SmAff}_k^\times\to\D( \In\mcal(k,\Z))^\otimes\] 
  whose underlying functor is 
  \[\mathrm{SmAff}_k\xrightarrow{C^*}\mathrm{Ch}
  ( \In\mcal(k,\Z))\xrightarrow{W_{\mathrm{qiso}}^{-1}}\D( \In\mcal(k,\Z)).\]

Now, by \cite[Proposition 7.1]{MR3649230} the above functor factors through \'etale motives. If $K/k$ is an extension of subfields of $\C$, then as any cellular filtration of a variety over $k$ provides a cellular filtration of its base change to $K$, this functor is compatible with change of fields. Moreover, by \cite[Proposition 7.1]{MR3649230}, the composition of $C^*$ with the Betti realisation of Nori motives is equivalent to the Betti realisation of \'etale motives.
Thus we obtained:
\end{constr}
\begin{thm}[Choudhury-Gallauer, Harrer, Nori]
  \label{realFieldDInd}
  Let $k$ be a subfield of $\C$. 
  There exists a symmetric monoidal functor 
  \[\rho'_\mathrm{N}\colon \DM(k,\Z)\to \D(\In\mcal(k,\Z))\] 
  that gives back the Betti realisation when composed with the Betti realisation of Nori motives, and sends 
  $\DMgm(k,\Z)$ to $\D^b(\mcal(k,\Z))$. Moreover, the construction is compatible with pullback of fields.
\end{thm}

It will be crucial for us that, rationnally, the above realisation functor is in some sense compatible with the realisation functor %to $\In\D^b(\mcal(k,\Q))$ obtained by taking the indisation 
\[\rho_\mathrm{N}\colon \DM(k,\Q)\to \In \D^b(\mcal(k,\Q))\] %of the restriction $\DMgm(k,\Z)\to \D^b(\mcal(k,\Z))$ of $\rho'_\mathrm{N}$. The functor $\rho_\mathrm{N}$ is also the realisation functor constructed in 
of \cite{SwannRealisation} which is given by the universal property of $\DM$ as a six functors formalism. 

We begin with a lemma.
\begin{lem}
  \label{lem:DIndTensQ}
  Let $k$ be a subfield of $\C$. The canonical functor 
  \[\D(\In\mcal(k,\Z))\to\D(\In\mcal(k,\Q))\]
  induces an equivalence $\D(\In\mcal(k,\Z))\otimes\Mod_\Q\xrightarrow{\sim}\D(\In\mcal(k,\Q))$ where the tensor product is taken in $\PrL$.
\end{lem}
\begin{proof}
  By \cite[C.4.2.2]{lurieSpectralAlgebraicGeometry}, the functor $\D(\In\mcal(k,\Z))\otimes\Mod_\Q\to\D(\In\mcal(k,\Q))$ is the stabilisation of 
  the functor
   \[\D^{\leqslant 0}(\In\mcal(k,\Z))\otimes \Mod^{\leqslant 0}_\Q\to\D^{\leqslant 0}(\In\mcal(k,\Q)).\]
   Now, recall (\cite[Definition 7.1.10]{MR3618276}) that $\mcal(k,\Z)$ is a colimit of categories of finitely generated modules over some rings $E_F$, so that $\In\mcal(k,\Z)$
   is in fact the filtered colimit of $\mathrm{Mod}_{E_F}^\heartsuit$. Thus, the prestable category $\D^{\leqslant 0}(\mcal(k,\Z))$ is the filtered colimit  of the prestable categories 
   $\Mod_{E_F}^{\leqslant 0}$ as the functor 
   $\D^{\leqslant 0}(-)$ is a left adjoint by \cite[C.5.4.9]{lurieSpectralAlgebraicGeometry}. Finally, the tensor product with $\Mod_\Q^{\leqslant 0}$ above is the filtered colimit 
   of $\Mod_{E_F\otimes_\Z\Q}^{\leqslant 0}$, but this is $\D^{\leqslant 0}(\In\mcal(k,\Q))$ by \cite[Lemma 7.2.2 1.]{MR3618276}.
\end{proof}

% \cref{DIndQ} implies that we have a canonical functor 
% \[\mathrm{DN}(k,\Q)\to \D(\In\mcal(k,\Q))\simeq\D(\In\mcal(k,\Z))\otimes\mathrm{Mod}_\Q\]
% which is fully faithful when restricted to $\D^b(\mcal(k,\Q))$: we have some control on the rational part of $\D(\In\mcal(k,\Z))$.

We now prove that when rationalised, the funtor $\rho'_\mathrm{N}$ is very close to the functor $$\rho_\mathrm{N}\colon \DM(k,\Q)\to\In\D^b(\mcal(k,\Q))$$ constructed by the second author in \cite{SwannRealisation}.
We begin with some reminders on weak Tanakian formalism from a monadic point of view.  The next propositions consist of bookeeping Lurie's higher algebra book and Raksit's \cite{raksitHochschildHomologyDerived2020} to revisit the constructions of Ayoub \cite{MR3259031} in the setting of $\infty$-categories.

\begin{defi}
  Let $\ccal$ be a symmetric monoidal $\infty$-category.
  \begin{enumerate}
    \item A commutative bi-algebra in $\ccal$ is an object of $\mathrm{coAlg}(\CAlg(\ccal)):=\mathrm{Alg}(\CAlg(\ccal)^\op)^\op$, that is an object of the opposite category of associative algebras in the opposite category of commutative algebras in $\ccal$.
    \item A Hopf algebra is a bi-algebra $H$ in $\ccal$ equiped with a homotopy antipode, that is the bi-algebra $\mathrm{ho}(H)\in\mathrm{ho}(\ccal)$ is a Hopf algebra in the usual sense.
  \end{enumerate}
%   Given a presentably symmetric monoidal $\infty$-category $\ccal$ and $H$ a Hopf algebra in $\ccal$ (that is a commutative bi-algebra object $H\in\mathrm{Alg}(\CAlg(\ccal))$ with a homotopy antipode), we will denote by $\mathbf{coMod}_H(\ccal)$ the $\infty$-category of comodules over $H$. For us, this is 
% \[\mathbf{coMod}_H(\ccal) := \lim_{[n]\in\Delta}\mathrm{Mod}_{H_n}(\ccal),\] where $H_\bullet = \mathrm{coBar}(H)$ is the coBar construction of $H$ and the simplicial functor sending $[n]$ to $\mathrm{Mod}_{H_n}(\ccal)$ is obtained by composing the functor $(H_\bullet,\ccal)$ with the functor $\Theta$ of \cite[Theorem 4.8.5.11]{lurieHigherAlgebra2022}.
\end{defi}
\begin{rem}
  \label{rem:CAlgCoAlg=CoAlgCAlg}
  By \cite[Proposition 2.1.2]{raksitHochschildHomologyDerived2020}, there is a natural equivalence $\mathrm{coAlg}(\CAlg(\ccal))\simeq \CAlg(\mathrm{coAlg}(\ccal)):=\mathrm{CAlg}(\mathrm{Alg}(\ccal^\op)^\op)$. In particular, any commutative bi-algebra $A$ has an underlying co-algebra structure in $\ccal$. 
\end{rem}
\begin{defi}
  Let $\ccal$ be a symmetric monoidal $\infty$-category and let $A$ be a commutative co-algebra in $\ccal$. The $\infty$-category of comodules over $A$ is by definition \[\mathrm{coMod}_A(\ccal):=\mathrm{LMod}_A(\ccal^\op)^\op\] where we see $A$ as an associative algebra in $\ccal^\op$ through \cref{rem:CAlgCoAlg=CoAlgCAlg}. By \cite[Proposition 2.2.1]{raksitHochschildHomologyDerived2020} (see \cref{thm:raksit} and also \cite[Theorem 3.18]{zbMATH07771515}), the $\infty$-category $\mathrm{coMod}_A(\ccal)$ is naturally symmetric monoidal and the forgetful functor $\mathrm{coMod}_A(\ccal)\to\ccal$ is symmetric monoidal. 
\end{defi}

\begin{rem}
  Recall that a functor $F\colon \mathcal{D}\to\mathcal{C}$ is \emph{comonadic} if it has a right adjoint $G$ such that the natural factorisation through modules over the endomorphism comonad (\cite[Proposition 4.7.3.3]{lurieHigherAlgebra2022})
  \[\mathcal{D}\to \mathrm{coMod}_T(\ccal)\to \ccal,\] where $T$ is the opposite of the endomorphism monad $T_{F^\op}\in\mathrm{Alg}(\mathrm{End}(\ccal^\op))$ of $F^\op$ and $\mathrm{coMod}_T(\ccal):=\mathrm{LMod}_{T_{F^\op}}(\ccal^\op)^\op$, induces an equivalence $\mathcal{D}\to\mathrm{coMod}_T(\ccal)$. The typical example is the above $\mathrm{coMod}_A(\ccal)\to \ccal$.
\end{rem}

\begin{lem}
  \label{lem:ComonadComod}
  Let $\ccal$ be a symmetric monoidal $\infty$-category and let $H$ be a Hopf algebra in $\ccal$. Consider the co-bar construction $H_\bullet:=\mathrm{coBar}(H)\colon \Delta\to \mathrm{CAlg}(\ccal)_{/\un_\ccal}$ (\cite[Construction 4.4.2.7]{lurieHigherAlgebra2022}). Let 
  \[\mathbf{coMod}_H(\ccal) := \lim_{n\in\Delta}\mathrm{Mod}_{H_n}(\ccal).\]
  Then the canonical functor $o_H\colon \mathbf{coMod}_H(\ccal)\to\ccal$ given by the projection on $\Mod_{\un_\ccal}(\ccal)\simeq \ccal$ is comonadic, and induces a symmetric monoidal equivalence 
  \[\mathbf{coMod}_H(\ccal)\simeq \mathrm{coMod}_H(\ccal).\]
\end{lem}

\begin{proof}
  See \cite[Proposition E.1.4]{zbMATH06622449} (this is a direct application of \cite[Theorem 4.7.5.2]{lurieHigherAlgebra2022}).
\end{proof}

Recall the following result on endomorphism comonads:
\begin{thm}[{{\cite[\S 4.7.3]{lurieHigherAlgebra2022}}, \cite[Corollary 5.12]{MR4367222}}]
  \label{thm:monadsAbstr}
  The forgetful functor 
  \[\mathrm{Fun}(\Delta^1,\catinfty)_{\mathrm{comnd}\text{-}\mathrm{ladj}}\to \mathrm{Fun}(\Delta^1,\catinfty)_{\mathrm{ladj}}\] from comonadic left adjoints functors to left adjoints functors admits a left adjoint sending a left adjoint $F:\mathcal{D}\to\mathcal{C}$ to the forgetful functor $\mathrm{coMod}_{(T_{F^\op})^\op}(\ccal)\to \ccal$ from comodules over the endomorphism comonad.
\end{thm} 
When considering linear $\infty$-categories over a symmetric monoidal $\infty$-category $\ccal$, the above adjunction simplifies drastically.

\begin{prop}[Raksit]
  \label{thm:raksit}
  Let $\ccal$ be a symmetric monoidal $\infty$-category. The association $A\mapsto \mathrm{coMod}_A(\ccal):= \mathrm{LMod}_A(\ccal^\op)^\op$ assembles to give a symmetric monoidal functor 
  \[\mu'_\ccal\colon \mathrm{coAlg}(\ccal)\to\mathrm{RMod}_\ccal(\catinfty)_{/\ccal}\] from co-algebras in $\ccal$ to right $\ccal$-modules in $\catinfty$, with a map to $\ccal$ (that is, pairs $(\mathcal{M},U)$ with $\mathcal{M}$ an $\infty$-category with a $\ccal$-module structure, and $U\colon\mcal\to\ccal$ a $\ccal$-linear functor). The symmetric monoidal structure on $\mathrm{coAlg}(\ccal)$ is inherited from that of $\ccal^\op$ and the symmetric monoidal structure on $\mathrm{RMod}_\ccal(\catinfty)_{/\ccal}$ is inherited from that of $\catinfty$, which is the cartesian one. 
  
  Then this functor is fully faithful and its image consists of pairs $(\mcal,U)$ such that $U$ is comonadic and the right adjoint of $U$ is $\ccal$-linear.
  Moreover, consider the full subcategory $$\mathrm{RMod}_\ccal(\catinfty)_{/\ccal}^0\subset \mathrm{RMod}_\ccal(\catinfty)_{/\ccal}$$ consisting of pairs $(\mcal,U)$ where the right adjoint of $U$ is also $\ccal$-linear. Then the induced functor 
  \[\mu'_\ccal\colon \mathrm{coAlg}(\ccal)\to\mathrm{RMod}_\ccal(\catinfty)_{/\ccal}^0\] admits a left adjoint $\nu'_\ccal$.
\end{prop}
\begin{proof}
  The existence of the functor is \cite[Proposition 2.2.1]{raksitHochschildHomologyDerived2020}, the full faithfullness is \cite[Proposition 2.3.6]{raksitHochschildHomologyDerived2020}, and the existence of the left adjoint $\nu'_\ccal$ follows from \cite[Lemma 2.3.5]{raksitHochschildHomologyDerived2020} (see the proof of \cite[Proposition 2.3.6]{raksitHochschildHomologyDerived2020}).
\end{proof}

\begin{rem}
  \label{rem:isReallyComonades}
  Under the equivalence $\mathrm{End}_\ccal(\ccal,\ccal)\simeq \ccal$ between $\ccal$-linear endofunctors of $\ccal$ and the $\infty$-category $\ccal$ itself (the functor is given by evalutation at $\un_\ccal$, see for example \cite[Corollary 4.8.5.21]{lurieHigherAlgebra2022} applied to the $\infty$-category $\In\ccal$), the functor $\mu'_\ccal$ of \cref{thm:raksit} is induced by that of \cref{thm:monadsAbstr}. In particular, the co-multiplication of the algebra $A=\nu'_\ccal(\mcal,U)$, for $U\colon\mcal\to\ccal$ a $\ccal$-linear functor whose right adjoint $G$ is also $\ccal$-linear is given by $$A=UG(\un_\ccal)\xrightarrow{\mathrm{unit}}U(GU)G(\un_\ccal)=UG\Big((UG)(\un_\ccal)\otimes \un_\ccal\Big)\simeq UG(\un_\ccal)\otimes UG(\un_\ccal)=A\otimes A$$ where in the last equivalence we use the $\ccal$-linearity of $G$ and $U$.
\end{rem}

\begin{cor}
  \label{cor:SymMonoComonad}
  Let $\ccal$ be a symmetric monoidal $\infty$-category. The functor $\mu'_\ccal$ induces a fully faithful right adjoint functor 
  \[\mu^\otimes_\ccal\colon \CAlg(\mathrm{coAlg}(\ccal))\to \mathrm{CAlg}(\catinfty)^0_{\ccal/ /\ccal}\] from commutative co-algebras in $\ccal$ to the $\infty$-category $\mathrm{CAlg}(\catinfty)^0_{\ccal/ /\ccal}$ of symmetric monoidal $\infty$-categories $\mcal$ with a symmetric monoidal functor $U\colon\dcal\to\ccal$ that has a symmetric monoidal section $s\colon \ccal\to \mcal$ and such that the underlying right adjoint of $U$ is $\ccal$-linear. 
\end{cor}
\begin{proof}
  Because $\mu'_\ccal$ is symmetric monoidal, it induces a functor betwen commutative algebras. The identification of the commutative algebras objects in $\mathrm{RMod}_\ccal(\catinfty)^0_{/\ccal}$ is explained in \cite[Construction 2.2.2]{raksitHochschildHomologyDerived2020}. We have to check that the functor $\nu'_\ccal$ induces a left adjoint to $\mu^\otimes_\ccal$. Because $\mu'_\ccal$ is symmetric monoidal, the functor $\nu'_\ccal$ is oplax symmetric monoidal, thus it suffices to check that the canonical map $\nu'_\ccal(\mcal_1\otimes_\ccal \mcal_2,U_1\otimes_\ccal U_2)\to\nu'_\ccal(\mcal_1,U_1)\otimes\nu'_\ccal(\mcal_2,U_2)$ is an equivalence. By construction of $\nu'_\ccal$ and by definition of the relative tensor product as the geometric realisation of the bar construction, it suffices to check that if $T_i\in \mathrm{Alg}(\mathrm{Fun}(\ccal_i^\op,\ccal_i^\op))^\op$ for $i\in\{1,2\}$ are two comonads, then the natural map $\mathrm{coMod}_{T_1\times T_2}(\ccal_1\times\ccal_2)\to\mathrm{coMod}_{T_1}(\ccal_1)\times \mathrm{coMod}_{T_2}(\ccal_2)$, induced by the left adjoint of the product preserving functor of \cref{thm:monadsAbstr}, is an equivalence. This follows directly from \cite[Corollary 4.7.3.16]{lurieHigherAlgebra2022} as both forgetful functors to $\ccal_1\times\ccal_2$ are comonadic with the same comonad.
\end{proof}
% \begin{rem}
%   \label{rem:end-monads}
%   In particular, the canonical functor $o_H\colon \mathbf{coMod}_H(\ccal)\to\ccal$ exhibits the source as $\mathrm{LMod}_{T}(\ccal^{\otimes,\op})^\op$ where $T$ is the endomorphism monad (\cite[Definition 4.7.3.2]{lurieHigherAlgebra2022}) of $$o_H^{\otimes,\op}\colon \mathbf{coMod}_H(\ccal)^{\otimes,\op}\to\ccal^{\otimes,\op}.$$
% \end{rem}
With this in mind, we can enhance \cite[Section 1]{MR3259031} to $\infty$-categories.
\begin{prop}
  \label{prop:ComonadeFactori}
  Let $R\colon \mathcal{M} \to \mathcal{C}$ be a symmetric monoidal left adjoint functor between symmetric monoidal $\infty$-categories. Assume the following:
    \begin{enumerate}
      \item $R$ admits a section $e\colon \mathcal{C}\to\mathcal{M}$ which is symmetric monoidal.
      \item The right adjoint $R_*$ of $R$ is $\mcal$-linear, that is, for every $c\in\ccal$ and $m\in\mcal$, the natural map \[R_*(c)\otimes m\to R_*(c\otimes R(m))\] is an equivalence.
    \end{enumerate}
    
    Then there exists a Hopf algebra $\mathcal{H}\in\mathcal{C}$ and a factorisation 
  \[\mathcal{M} \xrightarrow{\rho_\mathrm{univ}} \mathbf{coMod}_\mathcal{H}(\ccal)\xrightarrow{o_\mathcal{H}} \ccal\] of $R$ in $\mathrm{CAlg}(\catinfty)$ with the following property:
  For any commutative bi-algebra $A\in\ccal$ and any symmetric monoidal left adjoint functor \[\rho\colon\mathcal{M}^\otimes\to\mathrm{coMod}_A(\ccal)^\otimes\] whose composition with the forgetful functor is $R$, there exists an essentially unique map $\mathcal{H}\to A$ of commutative bi-algebras such that the following diagram commutes:
  \[\begin{tikzcd}
	&& {\mathrm{coMod}_A(\ccal)} \\
	{\mathcal{M}^\otimes} &&& {\ccal} \\
	&& {\mathbf{coMod}_{\mathcal{H}}(\ccal)}
	\arrow["U",from=1-3, to=2-4]
	\arrow["\rho", from=2-1, to=1-3]
	\arrow[from=2-1, to=3-3, "\rho_\mathrm{univ}",swap]
	\arrow[from=3-3, to=1-3,"\mathrm{cores}"]
	\arrow[from=3-3, to=2-4, "o_\mathcal{H}",swap]
\end{tikzcd}.\]
\end{prop}
\begin{proof}
  The hypothesis on $R$ makes the datum $(\mcal,R)$ an object of $\mathrm{CAlg}(\catinfty)^0_{\ccal/ /\ccal}$ : the $\mcal$-linearity of $R_*$ together with the existence of the section $e$ implies that $R_*$ is $\ccal$-linear. Let $\nu^\otimes_\ccal$ be the right adjoint of $\mu^\otimes_\ccal$, and let $\mathcal{H}:=\nu^\otimes_\ccal(\mcal,R)\in\mathrm{CAlg}(\mathrm{coAlg}(\ccal))$. By construction, the homotopy bi-algebra $\mathrm{ho}(\mathcal{H})$ is exactly the same as the one considered in \cite[Th\'eor\`eme 1.21]{MR3259031}. In particular, as our hypothesis on $R$ imply \cite[Hypoth\`ese 1.40]{MR3259031}, there exists a homotopy antipode by \cite[Th\'eor\`eme 1.45]{MR3259031}. Thus, $\mathcal{H}$ is a Hopf algebra in $\ccal$. Moreover, the unit of the adjunction $(\nu^\otimes_\ccal,\mu^\otimes_\ccal)$ provides a factorisation of $R$ as 
  \[\mcal\xrightarrow{\rho_\mathrm{univ}} \mathrm{coMod}_H(\ccal) \to \ccal\] where all $\infty$-categories are symmetric monoidal and the functors symmetric monoidal and $\ccal$-linear.

  Let $A\in\CAlg(\mathrm{coAlg}(\ccal))$ be a commutative bi-algebra. Then using the adjunction, we have:
  \[\mathrm{Map}_{\CAlg(\mathrm{coAlg}(\ccal))}(\mathcal{H},A)\simeq \mathrm{Map}_{\mathrm{CAlg}(\catinfty)_{\ccal/ /\ccal}}((\mcal,R),(\mathrm{coMod}_A(\ccal),U)).\] This provides the wanted universal property. The fact that in the proposition we used $\mathbf{coMod}_{\mathcal{H}}(\ccal)$ instead of $\mathrm{coMod}_{\hcal}(\ccal)$ is allowed by \cref{lem:ComonadComod}.

\end{proof}
\begin{rem}
  \label{rem:naivecomod}
  The comodule structure on $R(M)$ provided by the factorization through $\rho_\mathrm{univ}$ is nothing more than the co-unit $R(M)\xrightarrow{R(\varepsilon)} RR_*R(M)\simeq R(M)\otimes R(R_*(\un))$, where $R_*$ is the right adjoint of $R$. The functor from comodules over $\mathcal{H}$ to comodules over $H$ is also induced by the adjunction $(\rho,\rho_*)$. In particular, we see that up to forget some structure, we recover the factorisation of \cite[Lemme 1.55]{MR3259031}.
\end{rem}

We now specify to motives. When using \cref{cor:SymMonoComonad}, we may choose a bigger universe so that our presentable $\infty$-categories become small.

\begin{lem}
  \label{lem:checkhypo}
  Consider the Betti realisation $\rho_\mathrm{B}\colon \DM(k,\Q)\to \D(\Q)$. Then using the essentially unique symmetric monoidal left adjoint $\D(\Q)\to\DM(k,\Q)$, the couple $(\DM(k,\Q),\rho_\mathrm{B})$ satifies the hypothesis of \cref{prop:ComonadeFactori}.
\end{lem}
\begin{proof}
  We have to check that the right adjoint of $\rho_\mathrm{B}$ is $\DM(k,\Q)$-linear. This follows from \cite[Proposition 4.9]{MR3607274}: the $\infty$-category $\DM(k,\Q)$ is compactly generated (see for instance \cite[Lemma 5.1]{MR4319065}) and that its compact objects are the geometric motives, and the latter are dualisable over a field (\cite[Chapter 5, Theorem 4.3.7]{MR1764197}).
\end{proof}
\begin{recoll}
  \label{univPropComodHo}
Consider the Betti realisation $\rho_\mathrm{B}\colon \DM(k,\Q)\to \D(\Q)$. In \cite{MR3259031}, Ayoub endows the object $\mathcal{H}_\mathrm{A}:=\rho_\mathrm{B}\rho_{\mathrm{B}*}(\Q)$ with the structure of a Hopf algebra in $\mathrm{ho}(\D(\Q))$ (here $\rho_{\mathrm{B}*}$ is the right adjoint of $\rho_\mathrm{B})$. By \cref{prop:ComonadeFactori} and \cref{lem:checkhypo}, this enhances to a Hopf algebra structure in $\D(\Q)$ and we can factor $\rho_\mathrm{B}$ by a symmetric monoidal functor $\underline{\rho_\mathrm{B}}\colon \DM(k,\Q)\to \mathbf{coMod}_{\mathcal{H}_\mathrm{A}}(\D(\Q))$ and the forgetful functor of comodules.

Assume now we have a Hopf algebra $H$ in $\D(\Q)$ and a symmetric monoidal left adjoint functor $$\rho_H\colon\DM(k,\Q)\to\mathbf{coMod}_H(\D(\Q))$$ factoring the Betti realisation through the forgetful functor $o_H$. Applying \cref{prop:ComonadeFactori} we obtain a unique map $\mathcal{H}_A\to H$ of Hopf algebras such that the diagram 
\[\begin{tikzcd}
	&& {\mathbf{coMod}_H(\D(\Q))} \\
	{\DM(k,\Q)} &&& {\D(\Q)} \\
	&& {\mathbf{coMod}_{\mathcal{H}_A}(\D(\Q))}
	\arrow["{o_H}", from=1-3, to=2-4]
	\arrow["\rho", from=2-1, to=1-3]
	\arrow[from=2-1, to=3-3,"\underline{\rho_\mathrm{B}}",swap]
	\arrow[from=3-3, to=1-3,"\mathrm{cores}"]
	\arrow[from=3-3, to=2-4]
\end{tikzcd}\] of symmetric monoidal left adjoints commutes.

Finally, if $H$ is concentrated in degre zero, then by \cite[Proposition 2.4.3]{zbMATH06780490}, there is a natural functor 
\[\D(\mathrm{coMod}_H(\mathrm{Vect}_\Q))\to\mathbf{coMod}_H(\D(\Q))\] which induces an equivalence between left bounded complexes (so that its restriction to $\D^+$ is fully faithful).
\end{recoll}
\begin{rem}
  Using the universal property of $\mathcal{H}$ (\cref{cor:SymMonoComonad}), the universal property of the motivic Galois group of Iwanari in \cite{iwanari}, and the fact that $\Spec(\mathcal{H})$ coincides with Iwanari's motivic Galois group (\cite[Remark 1.3.22]{ayoubAnabelianPresentationMotivic2022}), it is not hard to check that the two possible universal functors $\DM(k,\Q)\to \mathbf{coMod}_\mathcal{H}(\D(\Q))$ defined using the two universal properties agree. 

  However, in \cite{zbMATH08143086,zbMATH06695841}, Ayoub enhances $\mathcal{H}$ to a stronger version of Hopf algebra (that is, $\mathcal{H}$ is a group object in $\mathrm{CAlg}(\D(\Q))^\op$) and constructs a symmetric monoidal realisation functor $\DM(k,\Q)\to\mathbf{coMod}_\mathcal{H}(\D(\Q))$. We did not try comparing this functor with the one we just constructed. 
\end{rem}

With the above, we can interpret the main theorem of \cite{MR3649230} as a description of the universal property of $\underline{\rho_\mathrm{B}}$:
\begin{constr}
  \label{constr:HAHN}
  Let $k$ be a subfield of $\C$. By \cite[Proposition 2.2]{MR1246272}, the functor $\D^b(\mcal(k,\Q))\to \D(\In\mcal(k,\Q))$ is fully faithful thus using \cref{NoriTannaka} we can see $\D^b(\mcal(k,\Q))$ as the full subcategory of $\mathbf{coMod}_{\mathrm{H}_\mathrm{N}}(\D(\Q))$ consisting of objects whose underlying object $o_{\mathrm{H}_\mathrm{N}}$ is a perfect complex. Consider the functor 
  \[\rho_{\mathrm{CG}}\colon \DMgm(k,\Q)\to \D^b(\mcal(k,\Q))\] obtained by restricting to $\DMgm(k,\Q)$ the functor $\rho'_\mathrm{N}$ of \cref{realFieldDInd} (where we identify the target using \cref{lem:DIndTensQ}). This induces a functor 
  $$\rho_\mathrm{CG}^\mathrm{big}\colon \DM(k,\Q)\to\mathbf{coMod}_{\mathrm{H}_N}(\D(\Q)).$$ It factors the Betti realisation, thus, by \cref{prop:ComonadeFactori}, this induces a map of Hopf algebras $\varphi\colon \mathcal{H}_\mathrm{A}\to \mathrm{H}_\mathrm{N}$. By \cite[Corollary 2.105]{MR3259031}, the object $\mathcal{H}_\mathrm{A}$ is connective in $\D(\Q)$, thus as $\mathrm{H}_\mathrm{N}$ is concentrated in degree zero, the map $\varphi$ factors trough the co-connective cover $\tau\colon \mathcal{H}_\mathrm{A}\to\mathrm{H}_\mathrm{A}$, giving a map $\overline{\varphi}\colon \HH_\mathrm{A}:=\HH^0(\mathcal{H}_\mathrm{A})\to \mathrm{H}_\mathrm{N}$, which is a map of Hopf algebras (this is because $\mathcal{H}_\mathrm{A}$ is connective).
\end{constr}
\begin{thm}[Choudhury-Gallauer]
  \label{thm:TrueCG}
  Let $k$ be a subfield of $\C$. The map $\overline{\varphi}\colon\HH_\mathrm{A}\to \mathrm{H}_\mathrm{N}$ of \cref{constr:HAHN} is an equivalence of Hopf algebras. In particular, the functor $\rho_\mathrm{CG}^\mathrm{big}$ fits in the following commutative diagram of symmetric monoidal left adjoint functors:
\[\begin{tikzcd}
	{\DM(k,\Q)} & {\mathbf{coMod}_{\mathcal{H}_\mathrm{A}}(\D(\Q))} & {\mathbf{coMod}_{\mathrm{H}_\mathrm{A}}(\D(\Q))} & {\D(\Q)} \\
	&& {\mathbf{coMod}_{\mathrm{H}_\mathrm{N}}(\D(\Q))}
	\arrow["{\underline{\rho_\mathrm{B}}}", from=1-1, to=1-2]
	\arrow["{\rho_\mathrm{B}}", from=1-1, to=1-4,bend left = 20]
	\arrow["{\rho_{\mathrm{CG}}^\mathrm{big}}"', from=1-1, to=2-3]
	\arrow["{\mathrm{cores}_\tau}", from=1-2, to=1-3]
	\arrow[from=1-3, to=1-4,"o_{\mathrm{H}_\mathrm{A}}"]
	\arrow["\mathrm{cores}_{\overline{\varphi}}", from=1-3, to=2-3]
	\arrow[from=2-3, to=1-4,"o_{\mathrm{H}_\mathrm{N}}",bend right=20,swap]
\end{tikzcd},\] where the vertical functor $\mathrm{cores}_{\overline{\varphi}}$ is an equivalence.

\end{thm}
\begin{proof}
  By the very construction of $\underline{\rho_\mathrm{B}}$ and \cref{prop:ComonadeFactori} (see \cref{rem:naivecomod}), we see that the above triangle 
  \[\begin{tikzcd}
	\mathbf{coMod}_{\mathrm{H}_\mathrm{A}}(\D(\Q)) \ar[r]\ar[d] & \D(\Q) \\
  \mathbf{coMod}_{\mathrm{H}_\mathrm{N}}(\D(\Q))\ar[ru]&
\end{tikzcd},\] composed with $\D(\Q)\to\mathrm{ho}(\D(\Q))$, can be factored as 
\[\begin{tikzcd}
	\mathbf{coMod}_{\mathrm{H}_\mathrm{A}}(\D(\Q)) \ar[r]\ar[d,] & \mathrm{coMod}^\mathrm{naive}_{\mathrm{H}_\mathrm{A}}(\mathrm{ho}(\D(\Q))) \ar[r]\ar[d] & \mathrm{ho}(\D(\Q)) \\
  \mathbf{coMod}_{\mathrm{H}_\mathrm{N}}(\D(\Q))\ar[r] & \mathrm{coMod}^\mathrm{naive}_{\mathrm{H}_\mathrm{N}}(\mathrm{ho}(\D(\Q))) \ar[ru]&
\end{tikzcd}\] with $\mathrm{coMod}^\mathrm{naive}$ naive (1-categorical) comodules. Moreover, the obtained functor 
$$\rho^\mathrm{univ}\colon \DM(k,\Q)\to \mathrm{coMod}^\mathrm{naive}_{\mathcal{H}_A} $$ is the functor constructed by Ayoub in \cite{MR3259031}, and thus by \cite[Lemme 1.55]{MR3259031} (that we may apply thanks to \cite[Lemma 7.6]{MR3649230}) the map $\overline{\varphi}$ is the only map of Hopf algebras whose composition with $\tau$ makes the diagram 
\[\begin{tikzcd}
	{\DM(k,\Q)} & {\mathrm{coMod}^\mathrm{naive}_{\mathcal{H}_\mathrm{A}}(\mathrm{ho}(\D(\Q)))} & {\mathrm{ho}(\D(\Q))} \\
	& {\mathrm{coMod}^\mathrm{naive}_{\mathrm{H}_\mathrm{N}}(\mathrm{ho}(\D(\Q)))}
	\arrow["{\rho^\mathrm{univ}}", from=1-1, to=1-2]
	\arrow["{ \rho_\mathrm{B}}"{description}, from=1-1, to=1-3, bend left=20]
	\arrow[from=1-1, to=2-2]
	\arrow["o", from=1-2, to=1-3]
	\arrow["{\mathrm{cores}_\varphi}"', dashed, from=1-2, to=2-2]
	\arrow["o"', from=2-2, to=1-3]
\end{tikzcd}\] commute. Thus the map $\overline{\varphi}$ is exactly the map $\varphi_\mathrm{N}$ considered in \cite{MR3649230}, which is thus an equivalence of Hopf algebras by Theorem 9.1 of \emph{loc. cit.}
\end{proof}

We now prove the expected comparison result.
\begin{prop}
  \label{compatRealRat}
  Let $k$ be a subfield of $\C$. Then the composition 
  \[\DM(k,\Z)\otimes\Mod_\Q\simeq\mathrm{DM}^\et(k,\Q)\xrightarrow{\rho_\mathrm{N}}\In\D^b(\mcal(k,\Q))\to\D(\In\mcal(k,\Z))\otimes\Mod_\Q \] 
  is naturally isomorphic to the functor $\rho'_\mathrm{N}$ constructed in \cref{realFieldDInd} tensored by $\mathrm{Id}_{\Mod_\Q}$.
\end{prop}
\begin{proof}
  First, by functoriality of the Lurie tensor product, the composition of $\rho'_\mathrm{N}\otimes\mathrm{Id}_{\Mod_\Q}$ with the functor 
  $\DM(k,\Z)\to\DM(k,\Q)$ is the functor obtained out of the universal property of $\DM(k,\Z)$ with $C^*$ replaced by $C^*\otimes \Q$. Moreover as $\DM(k,\Q)$ is compactly generated, it suffices to show that the two functors coincide when restricted to $\DMgm(k,\Q)$.
  
We now apply \cref{prop:ComonadeFactori} to the realisation functor $\rho_\mathrm{N}^\mathrm{big}\colon \DM(k,\Q)\to \mathbf{coMod}_{\mathrm{H}_\mathrm{N}}(\D(\Q))$ obtained by extending by colimits the functor $$\rho_\mathrm{N}\colon \DMgm(k,\Q)\to \D^b(\mcal(k,\Q))$$ constructed in \cite{SwannRealisation} using the 6 operations. This provides a Hopf algebra morphism $\alpha\colon\mathcal{H}_\mathrm{A}\to \mathrm{H}_\mathrm{N}$, that we can factor as $\alpha \simeq \overline{\alpha}\circ\tau$ with $\overline{\alpha}\colon \HH^0(\mathcal{H}_\mathrm{A})\to\mathrm{H}_\mathrm{N}$. Moreover, using that $\rho_\mathrm{CG} \simeq \mathrm{cores}_{\overline{\varphi}}\circ\mathrm{cores}_\tau\circ \underline{\rho_\mathrm{B}}$ (\cref{thm:TrueCG}), we have a commutative diagram 
\[\begin{tikzcd}
	{\DMgm(k,\Q)} & {\D^b(\mcal(k,\Q))} & {\mcal(k,\Q)} & {\mathrm{Vect}_\Q^\mathrm{fd}} \\
	& {\D^b(\mcal(k,\Q))} & {\mcal(k,\Q)} & {\mathrm{Vect}_\Q^\mathrm{fd}}
	\arrow["{\rho_\mathrm{CG}}", from=1-1, to=1-2]
	\arrow["{\rho_\mathrm{N}}"', from=1-1, to=2-2]
	\arrow["{\HH^0}", from=1-2, to=1-3]
	\arrow["{\mathrm{cores}_{\overline{\alpha}(\overline{\varphi}^{-1})}}", from=1-2, to=2-2]
	\arrow["{R_\mathrm{B}}", from=1-3, to=1-4]
	\arrow["{\mathrm{cores}_{\overline{\alpha}(\overline{\varphi}^{-1})}}", from=1-3, to=2-3]
	\arrow[from=1-4, to=2-4,equal]
	\arrow["{\HH^0}"', from=2-2, to=2-3]
	\arrow["{R_\mathrm{B}}"', from=2-3, to=2-4]
\end{tikzcd}.\] By \cite[Corollary 6.20]{SwannRealisation}, the universal factorisation functor $\HH^0_\mathrm{N}$ of \cref{sect:Norimot} is the composition $\HH^0\circ\rho_\mathrm{N}$. As this is also true for $\HH^0\circ \rho_\mathrm{CG}$ (see the proof of \cite[Corollary 10.1.7]{MR3618276}), by the universal property of Nori motives (\cite[Section 1]{ivorraFourOperationsPerverse2022}), this implies that $\mathrm{cores}_{\overline{\alpha}(\overline{\varphi}^{-1})}\colon\mcal(k,\Q)\to\mcal(k,\Q)$ is the identity functor. By \cite[Corollary 7.4.12]{bunkeControlledObjectsLeftexact2019}, the t-exactness also implies that the functor on $\D^b(\mcal(k,\Q))$ is the identity functor, finishing the proof.
\end{proof}

\subsection{The algebra representing Nori motivic cohomology.}
Let $k$ be a subfield of $\C$. Thanks to \cref{realFieldDInd} we have at hand a symmetric monoidal functor 
\[\rho'_\mathrm{N}\colon \DM(k,\Z)\to \D(\In\mcal(k,\Z))\] compatible with the Betti realisation. 
\begin{defi}
    \label{defi:NoriAlgebra}
  Let $F$ be the right adjoint of $\rho'_\mathrm{N}$. We denote by $\nscr_k$ the $\mathbb{E}_\infty$-ring object in $\DM(k,\Z)$ given by applying the lax symmetric monoidal functor $F$ to 
  the unit object $\Z_k$ of $\D(\In\mcal(k,\Z))$.
\end{defi}

\begin{lem}
  \label{RAdjRatRN}
  Let $M\in\mcal(k,\Z)$. Then the canonical map 
  \[F(M)\otimes\Q\to F(M\otimes \Q)\] is an equivalence.
\end{lem}
\begin{proof}
  It suffices to prove that for any smooth $k$-scheme $X$ and any integer $i\in\Z$, the map 
  \[\mathrm{map}_{\DM(k,\Z)}(M_k(X)(i),F(M)\otimes\Q)\to\mathrm{map}_{\DM(k,\Z)}(M_k(X)(i),F(M\otimes\Q))\]
  is an equivalence. Using that, for constructible \'etale motives, tensorisation by the rationals can be set out of the mapping spectra (see for example \cite[Corollary 5.4.9]{MR3477640}), and using the adjunction, we see that it suffices to prove that 
  the map 
  \[\mathrm{map}_{\D(\In\mcal(k,\Z))}(\rho'_\mathrm{N}(M_k(X)(i)),M)\otimes\Q\to\mathrm{map}_{\D(\In\mcal(k,\Z))}(\rho'_\mathrm{N}(M_k(X)(i)),M\otimes\Q)\]
is an equivalence. 
The right-hand side can be rewritten as 
\[\mathrm{map}_{\D(\In\mcal(k,\Z))\otimes\Mod_\Q}(\rho'_\mathrm{N}(M_k(X)(i))\otimes\Q,M\otimes\Q)\]
thus using \cref{lem:DIndTensQ}, the above map is in fact 
\[\mathrm{map}_{\D^b(\mcal(k,\Z))}(\rho'_\mathrm{N}(M_k(X)(i)),M)\otimes\Q\to\mathrm{map}_{\D^b(\mcal(k,\Q))}(\rho'_\mathrm{N}(M_k(X)(i))\otimes\Q,M\otimes\Q)\]
which is an equivalence by \cite[Proposition B.4.1]{MR3545132}.
\end{proof}

Let $k\subset K$ be subfields of $\C$ and denote by $f\colon\Spec(K)\to \Spec(k)$ the associated map of schemes. By construction we have a commutative square 
\[\begin{tikzcd}
	{\DM(k,\Z)} & { \D(\In\mcal(k,\Z))} \\
	{\DM(K,\Z)} & { \D(\In\mcal(K,\Z))}
	\arrow[from=1-1, to=2-1,"f^*"]
	\arrow["{\rho'_\mathrm{N}}", from=1-1, to=1-2]
	\arrow[from=1-2, to=2-2,"f^*"]
	\arrow["{\rho'_\mathrm{N}}"', from=2-1, to=2-2]
\end{tikzcd}\] in $\mathrm{Pr}^\mathrm{L}$. It induces an exchange transformation between $f^*$ and $F$.
In particular, we have a map 
\begin{equation}
  \label{comFieldsalgMap}
  f^*\nscr_k\to\nscr_K
\end{equation}
of algebras in $\DM(K,\Z)$. We will prove that this map is an equivalence, but first we need two important propositions:
\begin{prop}
  \label{algRat}
  Let $\nscr_{k,\Q}$ be the algebra constructed in \cite{SwannRealisation}, that is the algebra obtained by applying the right adjoint of the functor $\rho_\mathrm{N}$ to the unit object.
  By \cref{compatRealRat} and \cref{RAdjRatRN} we have a natural ring map 
  \[\nscr_{k,\Q}\to \nscr_k\otimes_\Z\Q\] in $\DM(k,\Q)$. This map is an equivalence. 
\end{prop}
\begin{proof}
  It suffices to prove that the map is an equivalence after applying the functors 
  $$\mathrm{map}_{\mathrm{DM}(k,\Q)}(M_k(X)(i),-)\colon\mathrm{DM}^\et(k,\Q)\to\D(\Q)$$ for $X$ a smooth $k$-scheme and $i$ an integer. Using adjunctions and \cref{RAdjRatRN},
  we see that it suffices to check that the map 
  \[\mathrm{map}_{\mathrm{DN}(k,\Q)}(\rho_\mathrm{N}(M_k(X)(i)),\Q_k)\to\mathrm{map}_{\mathrm{D}(\In\mcal(k,\Q))}(\rho'_\mathrm{N}(M_k(X)(i)),\Q_k)\]
  is an equivalence, which is the case because on $\D^b(\mcal(k,\Q))$ the functor 
  \[\In\D^b(\mcal(k,\Q))\to\D(\In\mcal(k,\Q))\] is fully faithful.
\end{proof}
Recall that the chang of site from the small \'etale site to the smooth \'etale site induces a functor 
\[\iota\colon \D(X_\et,R)\to \DM(X,R)\] for any commutative ring $R$ and any scheme $X$. By the rigidity theorem (see \cite{MR4224739} for the most general statement), this functor is an equivalence when $R$ is a torsion ring.
\begin{prop}
\label{rigAlgebre}
  The ring map $\un_k\to\nscr_k$ in $\DM(k,\Z)$ becomes an equivalence after passing mod $n$ for any nonzero integer $n\in\N^*$:
  $\un_k/n\simeq \nscr_k/n$.
\end{prop}
\begin{proof}
  Again to prove this, it suffices to prove that for every smooth $k$-scheme $X$ and any integer $i\in\Z$, the map 
  \[ \mathrm{map}_{\mathrm{DM}(k,\Z)}(M_k(X)(i),\un_k\otimes_\Z\Z/n\Z)\to \mathrm{map}_{\mathrm{DM}(k,\Z)}(M_k(X)(i),\nscr_k\otimes_\Z\Z/n\Z)\]
  is an equivalence. Using adjunctions and rigidity for \'etale motives, together with the fact that tensoring with $\Z/n\Z$ is a finite colimit thus goes out of the mapping spectra, 
  it suffices to prove that the map 
  \[ \mathrm{map}_{\D^b_c(k_\et,\Z/n\Z)}(\iota^{-1}(M_k(X)(i)),\un_k\otimes_\Z\Z/n\Z)\to \mathrm{map}_{\mathrm{D}^b(\mcal(k,\Z))}(\rho'_\mathrm{N}(M_k(X)(i)),\Z_k)\otimes_\Z\Z/n\Z\]
  is an equivalence. The right-hand side is in fact by \cref{torsderive} (that we can use because $\mcal(k,\Z)$ has enough torsionfree objects) the mapping complex in the category 
  $\D^b(\mcal(k,\Z)[n])$. By \cref{rigab} the category $\mcal(k,\Z)[n]$ is the category of finite type representations of the absolute Galois group of $k$ with $\Z/n\Z$ coefficients. 
  In particular, $\D^b(k_\et,\Z/n\Z)$ embeds in $\D^b(\mcal(k,\Z)[n])$, its image being exactly complexes of representations with perfect underlying complex of $\Z/n\Z$-modules. This finishes the proof.
\end{proof}
\begin{cor}
  \label{indfieldsalg}
  Let $k\subset K$ be a map of subfield of the complex numbers. The natural map 
  $(\nscr_k)_{\mid K}\to\nscr_K$ of \cref{comFieldsalgMap} is an equivalence.
\end{cor}
\begin{proof}
  Indeed, it suffices to prove this after rationalisation and passing mod $n$ for all nonzeros integers $n$. The rational statement is proven in \cite[Corollary 6.16]{SwannRealisation}, and the torsion statement follows from the fact that $f^*(\Z/n\Z)\simeq \Z/n\Z$ for all $f$, as any exact functor commutes with cofibers.
\end{proof}

We finish this section with the following result:
\begin{thm}\label{geomfields}
  Let $k$ be a subfield of the complex numbers. Consider the factorisation of $\rho'_\mathrm{N}$ through the modules category 
  \[\overline{\rho'_\mathrm{N}}\colon\Mod_{\nscr_k}(\DM(k,\Z))\to\D(\In\mcal(k,\Z)).\]
  This functor is fully faithful when restricted to the thick subcategory $\DNgm(k,\Z)$ of $\DN(k,\Z):=\mathrm{Mod}_{\nscr_k}(\DM(k,\Z))$ generated by the $M_k(X)(i)\otimes \nscr_k$ for $X$ smooth over $k$ and induces an equivalence between the latter and $\D^b(\mcal(k,\Z))$.
\end{thm}
\begin{proof}
  %Denote by $\DNc(k,\Z)$ the category of dualisable objects in $\Mod_{\nscr_k}(\DM(k,\Z))$. 
  Note that $\DNgm(k,\Z)$ is an idempotent complete stable $\infty$-category, also note that its image through $\overline{\rho_\mathrm{N}'}$ is contained in $\D^b(\mcal(k,\Z))$ because the image of the $M_k(X)$ through $\rho_\mathrm{N}'$ belongs to $\D^b(\mcal(k,\Z))$. 
  
  We first prove that 
  $\overline{\rho'_\mathrm{N}}$ is fully faithful when restricted to $\DNgm(k,\Z)$. 
  To do so, it suffices to check that it is the case after being tensored over $\Perf_\Z$ by 
  $\Perf_\Q$ and $\Perf_{\Z/n\Z}$ for all nonzero integers $n\in\N^*$. 
  When tensored by $\mathrm{Perf}_\Q$, using \cref{algRat} and \cite[Proposition 6.16]{SwannRealisation} we see that the functor is fully faithful.
  When tensored by $\mathrm{Perf}_{\Z/n\Z}$, this follows from \cref{rigab}, \cref{torsderive}, \cref{rigAlgebre} and rigidity for \'etale motives.
  We now prove essential surjectivity. In fact, the proof of fully faithfulness rationally gives that the functor 
  \[\DNgm(k,\Z)\to\D^b(\mcal(k,\Z))\] is an equivalence when tensored by $\Perf_\Q$. In particular, if $K\in\D^b(\mcal(k,\Z))$ is a complex of Nori motives, we see that there exist an object 
  $P\in\DNgm(k,\Z)$ and an map $\overline{\rho'_\mathrm{N}}(P)\to M$ that becomes an equivalence after being tensored by $\Q$. Said in an other way, 
  this map $\overline{\rho'_\mathrm{N}}(P)\to M$ has torsion cofibre, thus we may assume that $M$ is torsion: there exists an integer $n\in\N^*$ such that the multiplication by $n$ on $M$ is zero.
  In this case we have that $M\oplus M[1]\simeq M\otimes_\Z\Z/n\Z$, so it suffices to reach $M\otimes_\Z\Z/n\Z$.
  But because $M$ is dualisable, the image $M\otimes^L_\Z\Z/n\Z$ of $M$ in $\D^b(\mcal(k,\Z)[n])$ has perfect underlying complex of $\Z/n\Z$-modules, thus lies in $\D^b_c(k_\et,\Z/n\Z)$. By rigidity of \'etale motives again, we 
  see that $M\otimes_\Z\Z/n\Z$ is in the image of $\rho'_\mathrm{N}$ and therefore in that of $\overline{\rho'_\mathrm{N}}$, finishing the proof.
\end{proof}

\begin{cor}
  \label{Corgeomfields}
  Let $k$ be a subfield of $\C$. Then the thick category  $\DNgm(k,\Z)$ of $\DN(k,\Z)$ generated by the $M_k(X)(i)\otimes \nscr_k$ for $X$ smooth over $k$ is endowed with a t-structure such that the Betti realisation is t-exact.
\end{cor}

\begin{rem}\label{WhyDnComplicated} The above construction of $\DN(k,\Z)=\mathrm{Mod}_{\nscr_k}(\DM(k,\Z))$ is a bit more involved that what the reader might have expected; there are indeed two naive possibilities for a presentable $\infty$-category of Nori motives over $k$: one could take the unbounded derived category $\D(\In\mcal(k,\Z))$ of the indisation of Nori motives, or take the indization $\In\D^b(\mcal(k,\Z))$ of the bounded derived category of Nori motives. However, they do not give the correct presentable category because they cannot coincide with $\DM(k,\Z)$. 
  
More precisely, it is believed that the $\infty$-category $\DM(k,\Z)$ affords a t-structure, that restricts to the subcategory $\DMgm(k,\Z)$ whose heart would then be $\mcal(k,\Z)$. But then, the above naive definitions cannot give the right answer: firstly, it is known by a counterexample of Ayoub (\cite[Lemma 2.4]{MR3751289}) that the Betti realisation of the big category $\DM(k,\Q)$ is not conservative, so that it is not possible in general (at least if $k$ is of infinite transcendence degree over $\Q$) that this category is of the form $\D(\In\mcal(k,\Q))$ because the Betti realisation of the latter \emph{is} conservative; secondly, the category $\DM(k,\Z)$ is not compactly generated by $\DMgm(k,\Z)$ when the field $k$ has infinite cohomological dimension (take $k=\R$ the field of real numbers), so that $\In\D^b(\mcal(k,\Z))$ cannot be the right answer either. However see \cref{DIndQ} for a result related to this discussion.
\end{rem}

\section{Relative Nori motives.}
\subsection{Relative Nori motives as modules in \'etale motives}

In \cite{SwannRealisation} it was proven that the construction of Ivorra and Morel in \cite{ivorraFourOperationsPerverse2022} has a natural interpretation in terms of modules in \'etale motives. Indeed, if for $X$ a finite type $k$-scheme, one denote by \[\mathrm{DN}(X,\Q):=  \In \D^b(\Mp(X))\] the indization of the bounded derived category of perverse motives, the realisation functor 
\[\rho_\mathrm{N}\colon\DM(X,\Q)\to \mathrm{DN}(X,\Q) \] factors through
the $\infty$-category of modules in $\DM(X,\Q)$ over the algebra $\nscr_{X,\Q} := \pi_X^*\nscr_{k,\Q}$, with $\pi_X\colon X \to \Spec(k)$ the structural morphism, and induces an equivalence with the latter and $\mathrm{DN}(X,\Q)$: there is a natural equivalence of functors 
\[\Mod_{\nscr_{(-)}}(\DM(-,\Q))\xrightarrow{\sim} \mathrm{DN}(-,\Q).\]

In the previous section, we constructed a commutative algebra $\nscr_k$ in $\DM(k,\Z)$ which is equivalent to $\nscr_{k,\Q}$ when rationalised, and to $\un_k/n$ when tensored with $\Z/n\Z$. For $X$ a finite type $k$-scheme, we define $\nscr_X:=\pi_X^*\nscr_k$ to be the pullback of $\nscr_k$ to $X$. Clearly, this algebra also has the property that when rationalised, it is $\nscr_{X,\Q}$, and that the natural map $\un_X/n\to \nscr_X/n$ is an equivalence for $n$ a nonzero integer. Furthermore, if $f\colon X\to \Spec(\Q)$ is the structural map, \cref{indfieldsalg} implies that the natural transformation $f^*\mathscr{N}_{\Spec(\Q)} \to \mathscr{N}_X$ is an equivalence. We  therefore extend the definition of $\mathscr{N}_X$ to any $\Q$-scheme $f\colon X\to \Spec(\Q)$ by setting $\nscr_X:=f^* \nscr_{\Spec(\Q)}$.

\begin{defi}
  \label{defi:DNbig}
  Let $X$ be a $\Q$-scheme. The 
\emph{presentable $\infty$-category of Nori motives over $X$} is the $\infty$-category 
\[\mathrm{DN}(X,\Z) := \Mod_{\nscr_X}\DM(X,\Z)\] of modules in $\DM(X,\Z)$ over the algebra $\nscr_X$. 

Let $\Lambda$ be a commutative ring.  We then define \[\mathrm{DN}(X,\Lambda) := \Mod_{\Lambda}\DN(X,\Z).\]
We denote by $\Lambda_X$ unit object of $\DN(X,\Lambda)$
\end{defi}

We can promote this construction into a functor \[\mathrm{DN}(-,\Lambda)\colon \Sch^\op_\Q \to \mathrm{CAlg}(\mathrm{Pr}^\mathrm{L}_\Lambda),\] and there is a natural transformation $\DM(-,\Lambda)\to \mathrm{DN}(-,\Lambda)$:
as $\Spec(\Q)$ is the final object of $\Sch_\Q$, the functor $\DM(-,\Lambda)$ actually takes values in 
the $\infty$-category of $\DM(\Spec(\Q),\Lambda)$-algebras in $\mathrm{Pr}^\mathrm{L}$. There is a natural functor \[-\otimes(\nscr_{\Spec(\Q)}\otimes\Lambda)\colon\DM(\Spec(\Q),\Lambda)\to \Mod_{\nscr_{\Spec(\Q)}\otimes \Lambda}(\DM(\Spec(\Q),\Lambda))\] and we can define 
\[\mathrm{DN}(-,\Lambda):=\DM(-,\Lambda) \otimes_{\DM(\Spec(\Q),\Lambda)}\Mod_{\nscr_{\Spec(\Q)}}(\DM(\Spec(\Q),\Lambda)).\]

On $\DM(-,\Lambda)$, the operations $f^*$, $g_\sharp$ and $\otimes$ for $f$ a morphism of schemes and $g$ a smooth morphism of schemes are $\DM(\Spec(\Q),\Lambda)$-linear. Thus, they induce analogous functors for $\mathrm{DN}$: they are the $\nscr$-linearisation of the functors on $\DM(-,\Lambda)$. Moreover, if $g$ is smooth, the adjunction $(g_\sharp,g^*)$ is then an internal adjunction to the ($\infty$,2)-category of $\DM(\Spec(\Q),\Lambda)$-linear presentable $\infty$-categories, hence the $\nscr$-linearisation of each of the two functor actually form an adjunction for $\DN$.
By \cite[Theorem 4.6.1 and Remark 4.6.2]{MR4466640}, we then have:
\begin{prop}\label{SixFunBig}
  The functor \[\DN(-,\Lambda)\colon \Sch^\op_\Q \to \mathrm{CAlg}(\mathrm{Pr}^\mathrm{L}_\Lambda)\] is naturally part of a six functors formalism such that the functors of type $f^*$, $f_!$ and $\otimes$ are compatible with the forgetful functor to $\DM$. We denote the tensor structure by $\otimes_\mathrm{N}$ and the internal Hom by $\underline{\Hom}_\mathrm{N}$. 
\end{prop}
Furthermore $\DN$ also has some continuity properties when $\DM$ has.
\begin{defi}\label{LambdaFin}
  Let $X$ be a qcqs scheme and let $\Lambda$ be a commutative ring. 
  We say that $X$ is $\Lambda$-finite if it has finite Krull dimension and
  \[\sup_{x\in X, p\in \Lambda^\times}\mathrm{cd}_p(\kappa(x))<\infty\]
where $\mathrm{cd}_p(k)$ is the Galois cohomological dimension of a field $k$ with $\F_p$-coefficients.
\end{defi}

\begin{prop}
  \label{continu}
  Let $X=\lim X_i$ be a limit of qcqs \'etale-locally $\Lambda$-finite $\Q$-schemes with affine transition maps. 
  The natural map 
  \[\colim_i \DN(X_i,\Lambda)\to\DN(X,\Lambda)\] in $\mathrm{CAlg}(\mathrm{Pr}^\mathrm{L})$ is an equivalence.
\end{prop}
\begin{proof} The functor sending a pair $(\ccal,A)$ consisting of a symmetric monoidal presentable $\infty$-category $\ccal$ together with a commutative algebra $A$ in $\mathrm{CAlg}(\ccal)$ to the $\infty$-category $\Mod_A(\ccal)$ of modules over $A$ in $\ccal$ commutes with colimits: this is \cite[Lemma 3.5.6]{MR4466640}. The proposition then follows from the same result for \'etale motives which is \cite[Lemma 5.1]{MR4319065}.
\end{proof}

\begin{prop}\label{hdescGro}  
  The functor $\DN(-,\Lambda)$ has the same descent properties as $\DM(-,\Lambda)$. In particular:
  \begin{enumerate}
    \item It is an \'etale hypersheaf and a cdh-sheaf.
    \item It is a cdh-hypersheaf over qcqs $\Q$-schemes of finite valuative dimension.
    \item It is an h-hypersheaf over Noetherian schemes of finite dimension.
  \end{enumerate}
\end{prop}
\begin{proof}
  Note that the proposition is true if we replace $\DN$ by $\DM$. Indeed, since the presentable $\infty$-category $\DM(-,\Lambda)$ satisfies Nisnevich descent, the localisation property and proper base change, it satisfies cdh-descent as the proof of \cite[Theorem 3.3.10]{MR3971240} applies in this level of generality; for hyperdescent, this follows from the fact that the cdh topos is hypercomplete over qcqs schemes of finite valuative dimension by \cite[Corollary 2.4.16]{zbMATH07335443}. Finally, h-descent follows from \cite[Corollary 5.5.7]{MR3477640} as Bachmann's Rigidity Theorem \cite[Theorem 3.1]{bachmannrigidity} allows to remove the hypothesis on the base scheme in \textit{loc. cit.} 
  
  Now, let $f_\bullet\colon X_\bullet\to X$ be an hyper covering of one of the topologies considered in the proposition. By \cite[Proposition 2.12]{aranhaChowWeightStructure2023a} we have an equivalence of $\infty$-operads 
  \begin{equation}\label{limMod}\mathrm{Mod}(\DM(X,\Lambda))^\otimes\xrightarrow{\sim} \lim_\Delta \mathrm{Mod}(\DM(X_i,\Lambda))^\otimes\end{equation} where if $\ccal$ is a symmetric monoidal $\infty$-category, the $\infty$-category $\mathrm{Mod}(\ccal)$ can be informally described as pairs $(A,M)$ with $A$ in $\mathrm{CAlg}(\ccal)$ and $M$ in $\mathrm{Mod}_A(\ccal)$.
  We know that the diagram 
  \[\begin{tikzcd}
    \nscr_X
          \arrow[r]
        &  (f_0)_*\nscr_{X_0}
            \arrow[r, shift left = 0.25em]
            \arrow[r, leftarrow]
            \arrow[r, shift right = 0.25em]
          &  (f_1)_*\nscr_{X_1}
              \arrow[r, shift left = 0.5em]
              \arrow[r, leftarrow, shift left = 0.25em]
              \arrow[r]
              \arrow[r, leftarrow, shift right = 0.25em]
              \arrow[r, shift right = 0.5em]
            &  \dots
  \end{tikzcd}\] is a limit diagram in $\DM(X,\Lambda)$. Indeed, as $\nscr$ is pulled back from $X$ the above diagram is in fact 
  \[\begin{tikzcd}
    \nscr_X
          \arrow[r]
        &  (f_0)_*f_0^*\nscr_X
            \arrow[r, shift left = 0.25em]
            \arrow[r, leftarrow]
            \arrow[r, shift right = 0.25em]
          &  (f_1)_*f_1^*\nscr_{X}
              \arrow[r, shift left = 0.5em]
              \arrow[r, leftarrow, shift left = 0.25em]
              \arrow[r]
              \arrow[r, leftarrow, shift right = 0.25em]
              \arrow[r, shift right = 0.5em]
            &  \dots
  \end{tikzcd},\] so that it is a limit diagram by descent in $\DM(X,\Lambda)$. Thus, the fiber $\DN(X,\Lambda)$ over $\nscr_X$ of the left-hand side of \cref{limMod} corresponds to $\lim_\Delta\DN(X_i,\Lambda)$, finishing the proof.
\end{proof}

\begin{rem}[Universal property of Nori motives]
  Let $k$ be a subfield of $\C$ and let $\Lambda$ be a regular ring.
  By Drew and Gallauer's work \cite{MR4560376}, we know that $\DM(-,\Lambda)$ is universal among functors 
  \[\mathrm{Sch}_k^\mathrm{ft}\to\mathrm{CAlg}(\mathrm{Pr}^\mathrm{L}_\Lambda)\] that satisfy $\A^1$-invariance, $\mathbb{P}^1$-stability, smooth base change and \'etale hyperdescent. We claim that $\DN(-,\Lambda)$ is universal among the above class of functors $\ccal$ that verify the following additional conditions: the $\infty$-category $\ccal_\mathrm{liss}(k)$ of dualisable objects in $\ccal(k)$ affords a t-structure and the induced functor $\DMgm(k,\Lambda) \to \ccal_\mathrm{liss}(k)$ factors the Betti realisation in a way that the functor $\ccal_\mathrm{liss}(k) \to \Perf_\Lambda$ is t-exact and conservative. Indeed, for such a $\ccal$, the universal property of Nori motives over $k$ gives us a faithful exact functor $\mcal(k)\to \ccal_\mathrm{liss}(k)^\heartsuit$ such that the induced functor $\D^b(\mcal(k))\to \ccal_\mathrm{liss}(k)\to \Perf_\Lambda$ is the Betti realisation of Nori motives. In particular, we have a commutative diagram 
  \[
    \begin{tikzcd}
      {\DM(k,\Lambda)} & {\ccal(k)} \\
      {\DN(k,\Lambda)}
      \arrow[from=1-1, to=1-2]
      \arrow[from=1-1, to=2-1]
      \arrow[from=2-1, to=1-2]
    \end{tikzcd}
  \] which induces, if one denotes by $\mathscr{C}$ the commutative algebra representing cohomology with values in $\ccal(k)$, a map of algebras $\nscr_k\to \mathscr{C}$ in $\DM(k,\Lambda)$. By pulling back $\mathscr{C}$ over finite type $k$-schemes we obtain morphisms of functors 
  \[\DM(-,\Lambda)\to\DN(-,\Lambda)\to\mathrm{Mod}_\mathscr{C}(\DM(-,\Lambda))\to \ccal,\] thus the promised morphism of functors 
  $\DN\to \ccal$.

\end{rem}
\subsection{Nori motives of geometric origin}

The goal of this section is to prove the statements of \cite{DMpdf} in the setting of Nori motives. Namely, we define Nori motives of geometric origin and show that they can be characterised by the property of being \emph{constructible} which is more categorical. We also show descent properties as well as the continuity property.
\begin{defi}\label{DNgm}
  Let $X$ be a $\Q$-scheme and let $\Lambda$ be a commutative ring. The category $\DNgm(X,\Lambda)$ \emph{of Nori motives of geometric origin} (or simply \emph{geometric Nori motives}) is the thick subcategory  of $\DN(X,\Lambda)$ generated by the $f_\sharp(\Lambda_Y)(n)$ for $f\colon Y\to X$ smooth and $n$ an integer.
\end{defi}

\begin{thm}\label{megathmcons}
  Let $\Lambda$ be a commutative ring. Let $f\colon Y\to X$ be a morphism of qcqs $\Q$-schemes of finite dimension.
  \begin{enumerate}
    \item The functors of type $f^*$, $f_!$ (the latter for $f$ of finite type) and $\otimes_\mathrm{N}$ preserve geometric objects.
    \item  Let $M$ and $N$ be Nori motives over $X$ with $M$ geometric and let $P$ be a Nori motive over $Y$. Then, the natural transformations below are equivalences:. Then, the natural transformations below are equivalences:
    \begin{enumerate}
      \item $\map_\DN(M,N)\otimes \Q \to \map_\DN(M,N\otimes \Q)$
      \item $\underline{\Hom}_\mathrm{N}(M,N)\otimes \Q \to \underline{\Hom}_\mathrm{N}(M,N\otimes \Q)$
      \item $f_*(P)\otimes \Q \to f_*(P\otimes \Q)$
      \item $f^!(N)\otimes \Q \to f^!(N\otimes \Q)$ (for $f$ of finite type)
    \end{enumerate}
    \item Consider a qcqs $\Q$-scheme $X$ of finite dimension that is the limit of a projective system of qcqs $\Q$-schemes of finite dimension $(X_i)$ with affine transition maps.
    Then, the natural map
    \[\colim  \DNgm(X_i,\Lambda) \to \DNgm(X,\Lambda)\]
    is an equivalence.    
    \item Let $X$ be a qcqs $\Lambda$-finite $\Q$-scheme, then, a Nori motive $M$ is of geometric origin if and only if it is a compact object of $\DN(X,\Lambda)$.
    \item Assume that the ring $\Lambda$ is ind-regular\footnote{\textit{i.e.} it is a filtered colimit of regular rings.}. A Nori motive $M$ over a qcqs $\Q$-scheme of finite dimension is of geometric origin if and only if it is \emph{constructible}, that is for any open affine $U\subseteq X$, there is a finite stratification $U_i\subseteq U$  made of constructible locally closed subschemes such that each $M|_{U_i}$ is dualisable. 
  \end{enumerate}
\end{thm}
\begin{proof} In the course of this proof, we will denote by $\DNc$ the subcategory of constructible Nori motives. 
  
  The functors of type $f^*$, $f_\sharp$, $\otimes_\mathrm{N}$ and $i_*$ for $i$ a closed immersion readily preserve geometric objects. Using Nagata compactifications, the last remaining part of Assertion 1. is to prove that functors of type $p_*$ for $p$ proper preserve geometric objects which we will prove later. 
  
  Now, if $f\colon Y \to X$ is smooth and $n$ is an integer, \cref{SixFunBig} implies that $f_\sharp(\Lambda_Y)(n)=\rho_\mathrm{N}(f_\sharp(\un_Y)(n))$. Now, the functor $\rho_\mathrm{N}$ being monoidal, it sends constructible motives to constructible Nori motives. Since geometric motives are constructible using \cite[Corollary 3.7]{DMpdf}, the Nori motive $f_\sharp(\Lambda_Y)(n)$ is constructible. Hence, geometric Nori motives are constructible. 

We now prove Assertion 2.(a). By adjunction, we can reduce to the case where $M=\Lambda_X$.
  But we have a commutative diagram:
\[\begin{tikzcd}
  \map_\mathrm{DN}(\Lambda_X,N)\otimes \Q \ar[d] \ar[r] & \ar[d]\map_\mathrm{DN}(\Lambda_X,N\otimes \Q)\\
  \map_{\DM}(\Lambda_X,\rho^\mathrm{N}_*N)\otimes \Q \ar[r]& \map_{\DM}(\Lambda_X,\rho^\mathrm{N}_*(N\otimes \Q))
\end{tikzcd}\] 
where the vertical arrows are equivalences by definition of the right adjoint $\rho^\mathrm{N}_*$ to the Nori realisation and the bottom horizontal arrow is an equivalence by \cite[Proposition 3.1]{DMpdf} together with the fact that $\rho^\mathrm{N}_*$ commutes with colimits as it is a forgetful functor.

Assertions 2.(b) and 2.(c) as well as Assertion 2.(d) for $f$ a closed immersion then follow from 2.(a) by exactly the same proof as \cite[Proposition 3.1]{DMpdf}. Note that Assertion 2.(d) for a closed immersion implies that Assertion 2.(a) holds when $M$ is only assumed to be constructible: it allows by d\'evissage to reduce to the case where $M=\Lambda_X$.

  We now prove Assertion 4: in the $\Lambda$-finite case, the equivalence between geometric objects and compact objects can be deduced from the same result for $\DM$ which is \cite[Proposition 3.8]{DMpdf} and \cite[Lemma 2.6]{iwanari} which states that if a category is $\mc{C}$ is compactly generated, the category $\Mod_A(\mc{C})$ with $A$ some commutative algebra is still compactly generated and that the $A\otimes M$ with $M$ compact form a set of compact generators. This also implies Assertion 5. in the $\Lambda$-finite case: the fact that constructible objects are compact can then be proved exactly as in \cite[Proposition 3.8]{DMpdf}.

  The next step is to prove the continuity property both for geometric and for constructible Nori motives, \textit{i.e.} that the natural maps
  \[\colim  \DNgm(X_i,\Lambda) \to \DNgm(X,\Lambda)\]
  \[\colim  \DNc(X_i,\Lambda) \to \DNc(X,\Lambda)\]
are equivalences (here $\DNc$ denotes constructible Nori motives). This will in particular prove Assertion 3. As geometric motives are constructible, it suffices to prove the full faithfulness for constructible motives which can then be tested after tensoring with $\Q$ and with $\Z/p\Z$ for any prime number $p$. In the first case, it is then a consequence of \cref{continu}, while in the second case, it is a consequence of Bachmann's Rigidity Theorem in the form given in \cite[Theorem 3.2]{DMpdf} and of Bhatt and Matthew's continuity property for torsion \'etale sheaves in the form stated in \cite[Proposition 3.4]{DMpdf} which both apply for derived rings of coefficients. The essential surjectivity of the first functor can then be proved Zariski locally so that we can assume $X$ and the $X_i$ to be affine in which case, the usual spreading-out properties of \cite[Th\'eor\`eme 8.10.5, Proposition 17.7.8]{ega4} give the result (see the proof of \cite[Theorem 3.5]{DMpdf} for more details). The essential surjectivity of the second functor is then a consequence of the analogous of \cite[Lemma 3.9]{DMpdf} which can be proved exactly the same way:
\begin{lem}\label{bestlemma}
  Let $M$ be a constructible Nori motive over a qcqs $\Q$-scheme $X$ of finite dimension. Then, there is a geometric Nori motive $P$ as well as a map $P\to M\oplus M[1]$ whose cone $C$ is such that $C/n\simeq C\oplus C[1]$ for some integer $n$. 
\end{lem}

We can now finish the proof of Assertion 1: we have to prove that the functors of type $p_*$ for $p$ proper preserve geometric objects. This is a consequence of the continuity property for geometric objects and can be proved in the same fashion as in \cite[Corollary 3.6]{DMpdf}. Now, Assertion 1. implies Assertion 2.(d) for a general $f$ of finite which finishes the proof of Assertion 2.

It remains to prove that constructible Nori motives are geometric. As in the case of $\DM$, this can be reduced by continuity, localisation and d\'evissage to the case where $X$ is the spectrum of a field and $\Lambda$ is regular (see the proof of \cite[Theorem 4.1]{DMpdf}). 

\begin{lem} Let $k$ be a field of characteristic zero and let $\Lambda$ be a commutative ring. Then, the Artin motive functor \[\iota\colon\D(k_{\et},\Lambda)\rar \DN(k,\Lambda)\] is a fully faithful monoidal functor. Furthermore any torsion Nori motive lies in its essential image. 
\end{lem}
\begin{proof}
  If the result is known for $\Lambda=\Z$, then the result for an arbitrary ring of coefficients $\Lambda$ follows from tensoring both sides with $\Mod_\Lambda$: this operation preserves full faithfulness according to  \cite[Lemma 2.14]{haineNonabelianBasechangeBasechange2022}. 
  Now, using \cref{LemmeMagique}, \cref{rigAlgebre} and rigidity of \'etale motives, it suffices to show that this result holds with $\Lambda=\Q$. Then, the functor $\iota$ is the indization of the Artin motive functor \[\D(k_{\et},\Q)^\omega \to \D^b(\mcal(k,\Q))\] which is fully faithful. Indeed we have that the canonical functor \[\D^b(\mathrm{Rep}(\mathrm{Gal}(\overline{k}/k),\Lambda))\to\D(k_{\et},\Q)^\omega\] is an equivalence, and as all Artin motives are of weight $0$ in Nori motives, the abelian category of Artin motives is semi-simple, so that the above functor is fully faithful whenever it is fully faithful on the hearts, which is proven in \cite[Proposition 9.4.3]{MR3618276}. 
\end{proof}
  Let $M$ be a constructible Nori motive. We have to show that $M$ is geometric. \cref{bestlemma} yields a geometric Nori motive $P$ and a map $P\to M\oplus M[1]$ whose cone $C$ is such that $C\otimes \Q=0$.

We can then write $C=\iota D$, with $D$ a dualisable object of $\D(k_{et},\Lambda)$. In particular, it belongs to the thick stable subcategory generated by the $f_\sharp \underline{\Lambda}$ for $f$ finite \'etale according to \cite[2.1.2.15]{ruimyAbelianCategoriesArtin2023}. Thus, its image $C$ through $\iota$ is constructible as $\iota$ is compatible with the functors $f_\sharp$ and $f^*$ for $f$ finite \'etale: indeed $\iota=\rho_\mathrm{N} \rho_!$ 
  with $\rho_!$ the Artin motive functor to $\DM$; the same result is true for $\rho_!$ using \cite[§4.4.2]{MR3477640} and for $\rho_\mathrm{N}$ by \cref{SixFunBig}.
  Finally, $M\oplus M[1]$ is geometric and therefore $M$ is geometric.
\end{proof}
    \begin{cor}
      \label{colimrings}
      Let $\Lambda = \colim_i\Lambda_i$ be a filtered colimit of rings and let $X$ be a qcqs $\Q$-scheme of finite dimension. Then the canonical map 
      \[\colim_i\DNgm(X,\Lambda_i)\to\DNgm(X,\Lambda)\] is an equivalence.
\end{cor}
\begin{proof}
  As generators of $\DNgm(X,\Lambda)$ are clearly in the image, it suffices to prove fully faithfulness. Moreover by continuity we can assume that $X$ is of finite type over $\Q$. Now this can be proven after tensorisation by $\Q$ and $\Z/p$. The rational part follows from the fact that $\DNgm(X,\Lambda\otimes_\Z \Q)=\DNgm(X,\Lambda)\otimes\mathrm{Perf}_\Q$ consists of the compact objects of $\DN(X,\Lambda\otimes_\Z \Q)=\DN(X,\Q)\otimes \Mod_{\Lambda\otimes_\Z \Q}$ and the functor $\DN(X,\Q)\otimes\Mod_{\Q\otimes(-)}$ commutes with filtered colimits of rings. If we mod out by $p$, then the result follows from \cite[Lemma 4.3]{DMpdf}.
\end{proof}

    \begin{cor}\label{etdesc}
      Let $\Lambda$ be an ind-regular ring. The functor $\DNgm(-,\Lambda)$ is an \'etale hypersheaf over qcqs $\Q$-schemes of finite dimension. 
    \end{cor}
    \begin{proof}
      The condition of being constructible is \'etale local exactly as in \cite[Lemma 2.5]{DMpdf}.
    \end{proof}

    \begin{cor}\label{TensPerf}
      Let $\Lambda\to \Lambda'$ be a morphism between commutative rings and let $X$ be a qcqs $\Q$-scheme of finite dimension. The canonical map \[\DNgm(X,\Lambda)\otimes_{\Perf_\Lambda}\Perf_{\Lambda'} \to \DNgm(X,\Lambda')\] 
      is an equivalence. 
    \end{cor}
    \begin{proof}
      We claim that it is sufficient to deal with the case $\Lambda=\Z$. Indeed, assume that \[\DNgm(X,\Z)\otimes_{\Perf_\Z}\Perf_{A} \to \DNgm(X,A)\] is an equivalence for all rings $A$, then we have 
      \begin{align*}\DNgm(X,\Lambda)\otimes_{\Perf_\Lambda}\Perf_{\Lambda'} &\simeq \DNgm(X,\Z)\otimes_{\Perf_\Z}\Perf_{\Lambda}\otimes_{\Perf_\Lambda}\Perf_{\Lambda'}\\& \simeq \DNgm(X,\Z)\otimes_{\Perf_\Z}\Perf_{\Lambda'}\\&\simeq\DNgm(X,\Lambda').\end{align*}
      It suffices to prove fully faithfulness as generators of $\DNgm(X,\Lambda')$ clearly are in the image.
      Let $A$ be a ring and let $M,N$ be objects of $\DNgm(X,A)$. We will use that for any perfect complex of $A$-modules $P$, the canonical map 
      \begin{equation}\label{tensPerfMap}\map_{\DNgm(X,A)}(M,N)\otimes_A P\to\map_{\DNgm(X,A)}(M,N\otimes_A P)\end{equation} is an equivalence: this follows from the fact that $P$ is a finite colimit of free $A$-modules. By \cref{colimrings} we can assume that $\Lambda'$ is a ring of finite type over $\Z$, that is a quotient $\Z[T_1,\dots,T_r]/I$ of a polynomial ring by an ideal. Say we knew the result for $\Lambda' = \Z[T_1,\dots,T_r]$. Then as $\Z[T_1,\dots,T_r]/I$ is a perfect $\Z[T_1,\dots,T_r]$, the proof would be finished by \cref{tensPerfMap}. By induction it suffices to prove that if the result is known for a Noetherian ring $A$, then we can show that it holds for $A[T]$. But then writing \[A[T]=\colim_n A[X]_{\leqslant n}\] with $A[X]_{\leqslant n}$ the polynomials of degree $\leqslant n$, using \cref{tensPerfMap} we see that is suffices to prove that for $M,N$ in $\DNgm(X,A)$, the map 
      \[\colim_n \map_{\DNgm(X,A)}(M,N\otimes_A A[X]_{\leqslant n})\to \map_{\DNgm(X,A)}(M,\colim_n N\otimes_A A[X]_{\leqslant n})\] is an equivalence. Finally, this can be checked after tensorisation by $\Q$ (where this holds because $M$ becomes compact), and after tensorisation by $\Z/p$ where this follows from \cite[Lemma 5.3]{MR4609461} (note that \emph{loc. cit.} is stated for a ring and not a derived ring, but the proof is the same for our derived ring $A\otimes_\Z \Z/p$: pushforwards commutes with uniformly bounded below colimits by \cite[Corollary 3.10.5]{barwick2020exodromy} and then one has to prove the statement for the internal Hom, which is obvious is $M$ is dualisable, and follows in general by d\'evissage.)
    \end{proof}
    
  \begin{rem}
    The same proof gives the same result for $\DMgm$.
  \end{rem}

\subsection{Compatibilities of the six functors and Verdier duality}
We now study the compatibility of the Nori realisation with the six functors. We also show that the six functors induce functors on constructible motives when restricted to quasi-excellent schemes. Recall that by \cite[Theorem 4.6.1]{MR4466640} this is also the case for $\DM$.

\begin{thm}\label{compatibilite sif}
  Let $\Lambda$ be a commutative ring. Over the category of qcqs $\Q$-schemes of finite dimension, the Nori realisation
  \[\rho_\mathrm{N}:=-\otimes \mathscr{N} \colon \DM(-,\Lambda)\to \DN(-,\Lambda)\] is compatible with the operations
  \begin{itemize}
    \item $f^*$, $f_*$, for any morphism $f$ between qcqs $\Q$-schemes of finite dimension.
    \item $f_!$, $f^!$, for any morphism $f$ of finite type between qcqs $\Q$-schemes of finite dimension.
    \item $\otimes$ and $\underline{\Hom}(M,-)$ for $M$ of geometric origin.
  \end{itemize}
\end{thm}
\begin{proof}
The case of left adjoints is given by the definitions of the operations of type $f^*$, $f_!$ and $\otimes$ on $\DN(-,\Lambda)$. By adjunction, this yields exchange transformations 
\[\rho_\mathrm{N} R \to R \rho_\mathrm{N}\] for $R$ of type $f_*$, $f^!$  nd $\underline{\Hom}(M,-)$ for $M$ of geometric origin. The theorem amounts to proving that this exchange transformations are equivalences, which can be tested after tensoring those maps with $\Q$ or $\Z/p\Z$ for any prime number $p$. After tensoring with $\Z/p\Z$, the Nori realisation becomes an equivalence so the exchange maps are also equivalences. After tensoring with $\Q$, those maps can be interpreted as the same maps but with coefficients $\Lambda \otimes_\Z \Q$ since the operations of type $f_*$, $f^!$ and $\underline{\Hom}(M,-)$ for $M$ of geometric origin are compatible with rationalisation by Assertion 2. of \cref{megathmcons}. Hence, we can assume that $\Lambda$ is a $\Q$-algebra.

  Now recall that if $X$ is a qcqs $\Q$-scheme, we have \[\DN(X,\Lambda)=\DM(X,\Lambda)\otimes_{\DM(\Spec(\Q),\Lambda)} \DN(\Spec(\Q), \Lambda)\] and that the operation $f^*$ on $\DN$ is defined as $f^* \otimes \id_{\DN(\Spec(\Q), \Lambda)}$. Recall that $\DM(\Spec(\Q),\Lambda)$ is compactly generated by \cite[Lemma 5.1]{MR4319065} and that $\DN(\Spec(\Q),\Lambda)$ is also compactly generated by \cite[Lemma 2.6]{iwanari}. Recall also that if $f\colon Y\to X$ is a morphism of schemes, the functor $f^*$ preserves compact objects and therefore the functor $f_*$ is colimit-preserving. We can therefore apply \cref{BlackMagic} to show that $f_*\otimes \id_{\DN(\Spec(\Q), \Lambda)}$ is right adjoint to this functor which exactly means that $f_*$ is compatible with the Nori realisation. The case of the functor $\underline{\Hom}(M,-)$ can be deduced by the same argument applied to the functor $M\otimes -$ because geometric objects are compact (as the generators are the images of compact \'etale motives).
\end{proof}

\begin{prop}\label{sifgm}
  Let $\Lambda$ be a commutative ring. The six functors preserve the category
$\DNgm(-, \Lambda)$ when restricted to quasi-excellent $\Q$-schemes of finite dimension and to morphisms of finite type between these schemes. More precisely
\begin{enumerate}
  \item If $X$ is a qcqs $\Q$-scheme, then $\otimes_\mathrm{N}$ induces a monoidal structure on $\DNgm(X,\Lambda)$. It is closed when $X$ is quasi-excellent and of finite dimension.
  \item If $f$ is a morphism between qcqs $\Q$-schemes, then $f^*$ preserves geometric objects. If $f$ is furthermore of finite type, then $f_!$ also preserves geometric objects. 
  \item If $f$ is a morphism of finite type between quasi-excellent $\Q$-schemes of finite dimension, then $f_*$ and $f^!$ preserve geometric objects.
\end{enumerate}
\end{prop}
\begin{proof}
  The first part of (1), as well as (2) have already been proved in \cref{megathmcons}. For the rest of the proposition, we can work over quasi-excellent $\Q$-schemes of finite dimension which are Noetherian by assumption, so that the results of \cite{MR3477640} apply. 

  Let $X$ be a quasi-excellent $\Q$-schemes of finite dimension. 
  The motives of type $\rho_\mathrm{N}(M)$ with $M$ in $\DMgm(X,\Lambda)$ generate $\DNgm(X,\Lambda)$ as a thick subcategory of $\DN(X,\Lambda)$. Therefore, the result follows from the same result for $\DMgm$ which is \cite[Corollary 6.2.14]{MR3477640} and from the compatibility of the six functors with the Nori realisation which is \cref{compatibilite sif}.
\end{proof}

\begin{prop}\label{Verdier}
  Let $\Lambda$ be a commutative ring.  Consider an excellent scheme $B$ of dimension $2$ or less as well as a regular separated $B$-scheme of finite type $S$, endowed with a geometric and $\otimes$-
invertible object $\omega_S$ in $\DN(S, \Lambda)$. For any separated morphism of finite type $f\colon X\to S$,
we define the \emph{Verdier duality functor} $\mathbb{D}_X$ by the formula 
\[\mathbb{D}_X(M) := \underline{\Hom}_\mathrm{N} (M, f^!(\omega_S)).\] Then, for
any geometric Nori motive $M$, the canonical map $M\to \mathbb{D}_X(\mathbb{D}_X(M))$
is an isomorphism.
\end{prop}
\begin{proof}
  From the compatibility of the Nori realisation with the six functors (\cref{compatibilite sif}) and the same result for \'etale motives \cite[Theorem 6.2.17]{MR3477640}, we deduce the result when $\omega_S=\Lambda_S$. This implies the result in general using \cite[Proposition 4.4.22]{MR3971240}. 
\end{proof}

\begin{rem}
  A priori with rational coefficients we have two Verdier duality functors, at least rationally: the one constructed using the above formula, and the one obtained using the universal property of Nori motives. Thankfully, these two functors are canonically isomorphic by Terenzi's \cite[Proposition 5.19]{terenziTensorStructurePerverse2024a}. The proof uses the existence of a Nori realisation of rational \'etale motives.
\end{rem}

\subsection{The \texorpdfstring{$\ell$}{\textell}-adic realisation}
In this paragraph, we define the $\ell$-adic realisation on Nori motives. 
\begin{prop}
  \label{rigtorscomplet}
  Let $\Lambda$ be a commutative ring, let $X$ be a qcqs $\Q$-scheme and let $\ell$ be a prime number. The maps 
  \[\D(X_\et,\Lambda)_\mathrm{tors}\xrightarrow{\rho_!}\DM(X,\Lambda)_\mathrm{tors} \xrightarrow{\rho_\mathrm{N}} \DN(X,\Lambda)_\mathrm{tors} \]
  \[\D(X_\et,\Lambda)^\wedge_\ell\xrightarrow{(\rho_!)^\wedge_\ell}\DM(X,\Lambda)^\wedge_\ell \xrightarrow{(\rho_\mathrm{N})^\wedge_\ell} \DN(X,\Lambda)^\wedge_\ell \]
  are equivalences.
\end{prop}
\begin{proof}
  The second row of equivalences follows from the first and from the identification of $p$-torsion objects with $p$-complete objects. By \cref{LemmeMagique}, it suffices to prove the first identification after tensoring with $\Mod_{\Z/n\Z}$. The proposition then follows from the rigidity theorem \cite[Theorem 3.1]{bachmannrigidity} and from the fact that the canonical map $\un_X/n\to \nscr_X/n$ is an equivalence.
\end{proof}

\begin{defi}
  \label{defi:ell-adic-real}
  Let $X$ be a qcqs $\Q$-scheme. The big $\ell$-adic realisation of Nori motives is the functor \[R_\ell^\mathrm{big}\colon\DN(X,\Z)\xrightarrow{(-)^\wedge_\ell}\DN(X,\Z)^\wedge_\ell \xrightarrow{\sim} \D(X_{\et},\Z)^\wedge_\ell.\]
\end{defi}
\begin{rem}
  As in \cite[Section 7.2]{MR3477640}, we can endow $\D(-_{\et},\Z)^\wedge_\ell$ with the six functors formalism and show that the big $\ell$-adic realisation is compatible with the six functors.
\end{rem}

We now construct the $\ell$-adic realisation for constructible motives. It will land into the category of constructible sheaves of \cite{zbMATH06479630,MR4609461}.

\begin{constr}\label{LAdicSmall}
  If $X$ is a qcqs $\Q$-schemes of finite dimension, we have maps:
  \[\DN(X,\Z)\xrightarrow{R_\ell^\mathrm{big}}\D(X_\et,\Z)^\wedge_\ell\xrightarrow{-\otimes \Z/\ell^n\Z} \D(X_\et,\Z/\ell^n\Z).\]
The image of $\DNgm(X,\Z)$ through the latter is contained in the subcategory of constructible sheaves $\dconset(X_\et,\Z/\ell^n\Z)$ (that is those complexes whose restriction to any open affine subset are dualisable  on a finite stratification).
  The change of sites from the \'etale to the pro-\'etale site then induces an equivalence \[\dconset(X_\et,\Z/\ell^n\Z)\to \dconspro(X_\proet,\Z/\ell^n\Z)\] using \cite[Proposition 7.1]{MR4609461}. We therefore get a map
  \[\DNgm(X,\Z)\to \lim \dconset(X_\et,\Z/\ell^n\Z).\]
  \cite[Proposition 5.1]{MR4609461} ensures that the right-hand side is equivalent to $\dconspro(X_\proet,\Z_\ell)$ the $\infty$-category of constructible pro\'etale sheaves (again, those complexes whose restriction to any open affine subset are dualisable  on a finite stratification but in the pro-\'etale world). The categories $\dconspro((-)_\proet,\Z_\ell)$ are endowed with the six functors over quasi-excellent base schemes using \cite[Section 6.7]{zbMATH06479630}; these are defined by passing to the limit from the functors over $\dconset((-)_\et,\Z/\ell^n\Z)$ which are induced from those over $\D(-_\et,\Z/\ell^n\Z)$. This yields a map \[R_\ell\colon\DNgm(-,\Z)\to \dconspro((-)_\et,\Z_\ell)\] compatible with the six functors over quasi-excellent schemes which we call the \emph{$\ell$-adic realisation.}

  Finally, if $\Lambda$ is a localisation of a ring integral and flat over $\Z$, we can tensor the above functor by $\Perf_\Lambda$ over $\Perf_\Z$ to get the $\ell$-adic realisation \[R_\ell\colon\DNgm(-,\Lambda)\to \dconspro((-)_\proet,\Z_\ell)\otimes_{\Perf_\Z}\Perf_\Lambda \xrightarrow{\sim} \dconspro((-)_\proet,\Lambda_\ell),\] where the last equivalence is \cite[Corollary 2.3]{MR4630128}. 
  
  This functor is compatible with the functors of type $\otimes$, $f^*$ and $f_!$ and when we restrict to quasi-excellent base schemes and to finite type morphisms, it is compatible with the functors of type $f_*$ and $f^!$. Indeed, the same is true for the $\ell$-adic realisation of \'etale motives and therefore this follows from \cref{compatibilite sif}.
\end{constr}

\begin{prop}\label{Conservativity}
  Let $\Lambda$ be a localisation of a flat integral extension of $\Z$. Let $X$ be a qcqs $\Q$-scheme of finite dimension.
  \begin{enumerate}
    \item If $\Lambda$ is a $\Q$-algebra, and $\ell$ is a prime number, then the $\ell$-adic realisation functor \[R_\ell\colon\DNgm(X,\Lambda)\to \dconspro(X_\proet,\Lambda_\ell)\] is conservative for any prime $\ell$.%\todoeasy{ici j ai enleve "invertible on $X$".}
    \item If $\Lambda$ is not a $\Q$-algebra, the family of the $R_\ell$ where $\ell\notin \Lambda^\times$ is conservative. This means that if $M$ is an object  $\DNgm(X,\Lambda)$ such that for all $\ell\not\in\Lambda^\times$, the object $R_\ell(M)$ vanishes, then $M\simeq 0$. 
  \end{enumerate}
\end{prop}
\begin{proof}
  We first show the first assertion. 
  By continuity (\cref{megathmcons}), we can reduce to the case where $X$ is a finite type $\Q$-scheme. As $\Lambda$ is a $\Q$-algebra, we have by \cite[Corollary 4.19]{SwannRealisation} an equivalence 
  \[\DN_\mathrm{gm}(X,\Lambda) \xrightarrow{\sim} \D^b(\mcal_\mathrm{perv}(X,\Q)) \otimes_{\Perf_\Q}\mathrm{Perf}_\Lambda.\]
  Moreover using the description of mapping spectra in tensor products provided by \cite[Proposition 3.5.5]{hoRevisitingMixedGeometry2023} we see that if the functor 
  \[R_\ell\colon \DNgm(X,\Q)\to\dconspro(X_\proet,\Q_\ell)\] is conservative, so is the functor of the proposition. Thus this follows from \cite[Proposition 6.11]{ivorraFourOperationsPerverse2022}. 

  We now prove the second assertion. As the family $\{-\otimes \Q$, $-\otimes \Z/\ell\Z \mid \ell \notin \Lambda^\times\}$ is conservative, by the first point, it suffices to show that if $R_\ell(M/\ell)$ vanishes, then so does $M/\ell$ for $M$ a geometric motive. This amounts to the rigidity theorem.
\end{proof}
We finish by a conjectural aspect of our construction: let $X$ be a qcqs $\Q$-scheme of finite dimension and let $\Lambda$ be a commutative ring. By construction we have a functor 
\[\rho_\mathrm{N}\colon\DM(X,\Lambda)\to\DN(X,\Lambda).\]
\begin{prop}
  \label{ifstdconjthengood}
  Assume that for all fields $K$ of finite transcendence degree over $\Q$, the $\infty$-category $\DMgm(K,\Q)$ is endowed with a  motivic t-structure (see the discussion around \cref{motivic t-structure}). Then for $X$ as above, the functor $\rho_\mathrm{N}$ is an equivalence. 
\end{prop}
\begin{proof}
  First, note that the functor $\rho_\mathrm{N}$ is the tensorisation by $\Mod_\Lambda$ of the functor with coefficients $\Z$, so that it suffices to deal with this case. Moreover as this functor is an equivalence after being tensored by $\Mod_{\Z/n\Z}$ for all $n\in\N^*$, it suffices to check that it becomes an equivalence after being tensored by $\Mod_\Q$. As $\nscr_X$ is pulled back from $\Spec(\Q)$, it also suffices to prove that $\rho_\mathrm{N}$ is an equivalence over $\Spec(\Q)$. The result then follows from \cite[Theorem 6.22]{SwannRealisation}.
\end{proof}
\begin{rem}
  \label{rem:Pridham}
  The hypothesis in the previous proposition could be weakened. Indeed it is very possible that using Pridham's \cite[Section 2.5]{MR4130840}, one can prove that if $\DMgm(\Spec(\Q),\Q)$ has a motivic t-structure, then the functor 
  \[\DMgm(\Spec(\Q),\Q)\to\D(\In\mcal(\Q,\Q))\] is fully faithful, so that, because its image is $\D^b(\mcal(\Q,\Q))$ (because $\mcal(\Q,\Q)$ is Noetherian, see \cite{MR1246272}), this induces an equivalence $\DMgm(\Spec(\Q),\Q)\simeq\D^b(\mcal(\Q,\Q))$, thus the functor \[\DM(\Spec(\Q),\Q)\to\DN(\Spec(\Q),\Q)\] is an equivalence, meaning that the map of rings $\un_k\otimes_\Z\Q\to\nscr_{k,\Q}$ is an equivalence. In other terms, to prove that $\DMgm(X,\Lambda)$ has a motivic t-structure for all $X$ as above and $\Lambda$, it ``suffices'' to prove that $\DMgm(\Spec(\Q),\Q)$ has a motivic t-structure. 
\end{rem}

Another aspect of the construction is that it gives evidence for a conjecture of Ayoub (\cite[Conjecture 2.8]{MR3751289}), up to the above \cref{ifstdconjthengood}:

\begin{prop}
  \label{DIndQ}
  Let $k$ be a subfield of $\C$.
   Denote by $\DN(k,\Q)^\flat$ the Verdier quotient of $\DN(k,\Q)$ by the kernel of the Betti realisation. 
   Then the presentable $\infty$-category $\DN(k,\Q)^\flat$ admits a separated t-structure whose heart is 
   canonically equivalent to $\In\mcal(k,\Q)$, and the natural functor 
  \[\D(\In\mcal(k,\Q))\to\DN(k,\Q)^\flat\] is an equivalence.
\end{prop}
\begin{proof}
    As $\mcal(k,\Q)$ is Noetherian, we can apply \cite[Proposition 3.6]{zbMATH02212628} to its indisation. This implies that the canonical functor \[J\colon \In \D^b(\mcal(k,\Q))\to\D(\In\mcal(k,\Q))\] is a Bousfield localisation (note that what is called $\mathrm{K}(\mathrm{Inj}(-))$ in \textit{loc. cit.} is nothing but $\In \D^b((-)^\omega)$ by \cite[Proposition 2.3]{zbMATH02212628}).

We now only have to show that $J$ and the Betti realisation have the same kernel. This amounts to the conservativity of the functor \[R'_\mathrm{B}\colon \D(\In \mcal(k,\Q))\to \Mod_\Q\] obtained by deriving the Betti realisation. Take an object $M$ to be killed by $R'_\mathrm{B}$. As the t-structure on $\D(\In \mcal(k,\Q))$ is separated, we can assume $M$ to be in the heart. By \cite[Lemma 1.3]{MR1246272} we can write $M=\colim M_i$ as a filtered colimit of objects $M_i$ of $\mcal(k,\Q)$ with injective transition maps. In particular, because $R'_\mathrm{B}(M)$ vanishes, so does each $R'_\mathrm{B}(M_i)$ and therefore $M_i$ by conservativity of the Betti realisation on Nori motives.
\end{proof}
\begin{rem}
  \label{conjAyoubDN}
  Using this, we can show that a conjecture of Ayoub (\cite[Conjecture 2.8]{MR3751289}) holds for Nori motives:
  Let $k$ be a subfield of $\C$. Let $M\in\D^b(\mcal(k,\Q))$ be a complex of Nori motives and let $\fcal\in\DN(k,\Q)$ be killed by the Betti realisation. Then the group 
  \[\mathrm{Hom}_{\DN(k,\Q)}(\fcal,M)=0\] vanishes.

  Indeed, because $\mcal(k,\Q)$ is Noetherian the above \cref{DIndQ} and \cite[Proposition 2.2]{MR1246272} imply that the functor \[\pi\colon \DN(k,\Q)\to\DN(k,\Q)^\flat\] is fully faithful when restricted to $\D^b(\mcal(k,\Q))$.
  We can write $\fcal=\colim_i N_i$ as a filtered colimit of objects in $\D^b(\mcal(k,\Q))$, and we have 
  \begin{align*}
    \mathrm{Hom}_{\DN(k,\Q)}(\fcal,M) & \simeq \lim_i\mathrm{Hom}_{\DN(k,\Q)}(N_i,M)\\
    &\simeq  \lim_i\mathrm{Hom}_{\DN(k,\Q)^\flat}(\pi(N_i),\pi(M))\\
    &\simeq  \mathrm{Hom}_{\DN(k,\Q)^\flat}(\pi(\fcal),\pi(M))\\
    & \simeq \mathrm{Hom}_{\DN(k,\Q)^\flat}(0,\pi(M))=0.
  \end{align*}
\end{rem}

\section{The abelian category of ordinary Nori motives.}

\subsection{Relative Nori motives as a pullback}

Until further notice, $k$ is an algebraically closed subfield of $\C$. For each finite type $k$-scheme $X$, Ayoub constructed in \cite{MR2602027} a Betti realisation functor 
\[\rho_\mathrm{B}\colon \DM(X,\Lambda)\to \D(X^\mathrm{an},\Z)\] to the derived category of analytic sheaves of abelian groups on $X^\mathrm{an}$. Let $\bscr_k $ be the algebra in $\DM(k,\Z)$ representing Betti cohomology \cite[Notation 1.57]{ayoubAnabelianPresentationMotivic2022} and if $X$ is a finite type $k$-scheme, let $\bscr_X$ be the pullback of $\bscr_k$ to $X$. Ayoub considered the functor on finite type $k$-schemes
\[\Mod_{\bscr_{(-)}}(\DM(-,\Z))\colon \left(\Sch_k^{ft}\right)^\op\to\mathrm{CAlg}(\mathrm{Pr}^{L}_\Z)\] 
of which he gave an explicit description in \cite[Theorem 1.93]{ayoubAnabelianPresentationMotivic2022}: he proved that the tautological functor 
$\DM(-,\Z)\to \Mod_{\bscr_{(-)}}(\DM(-,\Z))$ factors the refined Betti realisation 
\[\widetilde{\rho}_\mathrm{B}\colon\DM(-,\Z)\to \In  \dconsan((-)^\mathrm{an},\Z)\]
where $\dconsan((-)^\mathrm{an},\Z)$ denotes the category of constructible analytic complexes, and that for any finite type $k$-scheme $X$, the induced functor 
$\Mod_{\bscr_{X}}(\DM(X,\Z))\to \In  \dconsan((-)^\mathrm{an},\Z)$ is fully faithful. Its essential image is furthermore compactly generated: it is the indisation of a full subcategory $\dgman(X^\mathrm{an},\Z)$ of $\dconsan(X^\mathrm{an},\Z)$ of complexes of \emph{geometric origin}.

The realisation to Nori motives
\[\rho_\mathrm{N}\colon\DM(-,\Q)\to \DN(-,\Q)\] factors the refined Betti realisation of \'etale motives, hence there exists a map of algebra 
\[\varphi_X\colon\nscr_{X,\Q}\to \bscr_{X,\Q}\] compatible with pullback by finite type morphisms.
\begin{prop}\label{PBN}
Let $X$ be a $k$-scheme of finite-type. 
  There exist a natural map of algebras $\nscr_X\to\bscr_X$ giving back $\varphi_X$ after rationalisation. Moreover, it induces a pullback diagram 
  \[
  \begin{tikzcd}
      \nscr_X \arrow[r] \arrow[d]
          \arrow[dr, phantom, very near start, "{ \lrcorner }"]
        & \bscr_X \arrow[d] \\
      \nscr_{X,\Q} \arrow[r]
        & \bscr_{X,\Q}
  \end{tikzcd}.
  \]
\end{prop}
\begin{proof}
  Let $A$ be the pullback \[
  \begin{tikzcd}
    A \arrow[r] \arrow[d]
        \arrow[dr, phantom, very near start, "{ \lrcorner }"]
      & \bscr_X \arrow[d] \\
    \nscr_{X,\Q} \arrow[r]
      & \bscr_{X,\Q}
\end{tikzcd}.\]
Thanks to Artin's Comparison Theorem 
\cite[Expos\'e XI, Th\'eor\`eme 4.4 \& Expos\'e
XVI, Th\'eor\`eme 4.1]{SGA4} (see also Step 1 of the proof of \cite[Lemma 1.107]{ayoubAnabelianPresentationMotivic2022}), the Hasse arithmetic fracture square for $\bscr_X$ is of the form:
\[
\begin{tikzcd}
  \bscr_X \arrow[r] \arrow[d]
      \arrow[dr, phantom, very near start, "{ \lrcorner }"]
    & \lim_n\un_X/n  \arrow[d] \\
    \bscr_{X,\Q} \arrow[r]
    & (\lim_n\un_X/n)\otimes \Q
\end{tikzcd},\] 
so that by pasting pullbacks, $A$ also fits in the pullback square 
\[
\begin{tikzcd}
  A \arrow[r] \arrow[d]
      \arrow[dr, phantom, very near start, "{ \lrcorner }"]
    & \lim_n\un_X/n \arrow[d] \\
    \nscr_{X,\Q}\arrow[r]
    & (\lim_n \un_X/n)\otimes \Q
\end{tikzcd}.\] where we recognise the Hasse fracture square of $\nscr_\Z$ (the arrow are the correct ones because the equivalence $\nscr_X/n\xrightarrow{\sim} \bscr_X/n$ factors the equivalence $\un_X/n \xrightarrow{\sim} \bscr_X/n$.)
\end{proof}
Therefore we have a canonical functor 
\[\DN(-,\Z)\to \In \dgman((-)^\mathrm{an},\Z).\]
\begin{cor}\label{Betti real 6FF}
  Relative Nori motives with integral coefficients (even relative to a non algebraically closed field) have a Betti realisation 
  \[\DN(-,\Z)\to \D((-)^\mathrm{an},\Z).\]
 which is compatible with the six functors.
\end{cor}
\begin{prop}
  \label{PBBetti}
  The square 
  \[\begin{tikzcd}
      \DN(-,\Z) \arrow[r] \arrow[d] 
        & \In \dgman((-)^\mathrm{an},\Z) \arrow[d] \\
      \DN(-,\Q) \arrow[r]
        & \In \dgman((-)^\mathrm{an},\Q)
  \end{tikzcd}\] is cartesian in $\mathrm{Fun}(\Sch_k^\op,\mathrm{CAlg}(\mathrm{Pr}^\mathrm{L}_\Z))$.
\end{prop}
\begin{proof}
  It suffices to check the proposition after evaluating at a finite type $k$-scheme $X$. If we let $P(X)$ be the pullback of the proposition, there is a functor $\DN(X,\Z)\to P(X)$. This functor preserves compact objects (note that $k$ is algebraically closed thus $\DN(X,\Z)\simeq \In\DNgm(X,\Z)$ by \cref{megathmcons}), thus its right adjoint commutes with colimits, whence commutes with $-\otimes \Q$ so that we may use \cref{LemmeMagique}.
\end{proof}

\begin{cor}
  \label{betticons}
  Let $k$ be any subfield of $\C$ and let $X$ be a finite type $k$-scheme.
  The Betti realisation 
  \[\DNgm(X,\Z)\to \D^b_c(X^\mathrm{an},\Z)\] is conservative.
\end{cor}
\begin{proof}
  Indeed, the functor $\DNgm(X,\Z)\to\DNgm(X\times_k\overline{k},\Z)$ is conservative by continuity and Galois descent (if $M_{\overline{k}}= 0$, there exists a finite extension $L$ of $k$ such that $M_L$ is zero, and then by Galois descent the functor $(-)_L$ is conservative, thus $M= 0$).
  Then by \cref{PBBetti} the functor $\DNgm(X_{\overline{k}},\Z)\to \D^b_c(X^\mathrm{an},\Z)$ is conservative: if $R_\mathrm{B}(M)=0$, then by conservativity of the realisation rationally we have that $M\otimes \Q=0$, thus both $R_\mathrm{B}(M)$ and $M\otimes \Q$ are zero thus $M=0$ in the pullback.
\end{proof}

The pullback square of \cref{PBBetti} allows us to endow the $\DN(X,\Z)$ with a t-structure. Recall first that if $\Lambda$ is regular, the category $\dconsan(X^\mathrm{an},\Lambda)$ is endowed with a t-structure induced by the t-structure of $\D(X^\an,\Lambda)$. Lurie's \cite[Lemma C.2.4.3]{lurieSpectralAlgebraicGeometry} ensures that $\In \dconsan(X^\mathrm{an},\Lambda)$ is endowed with a t-structure whose negative (resp. positive) objects are the filtered colimits of negative (resp. positive) objects, and that therefore, the canonical functor \[\In \dconsan(X^\mathrm{an},\Lambda)\to \D(X^\an,\Lambda)\] is t-exact. One consequence of \cite[Theorem 1.93]{ayoubAnabelianPresentationMotivic2022}, is that $\dgman(X^\mathrm{an},\Lambda)$ has a t-structure induced by that of $\dconsan(X^\mathrm{an},\Lambda)$; we will denote its heart by $\Sh_\mathrm{gm}(X^\mathrm{an},\Lambda)$. Using \cite[Lemma C.2.4.3]{lurieSpectralAlgebraicGeometry} again, the stable $\infty$-category $\In \dgman((-)^\mathrm{an},\Lambda)$ has a t-structure which we call the \emph{ordinary t-structure}. 

On the other hand, we have $\DN(X,\Q)=\In  \D^b(\mcal_\mathrm{perv}(X,\Q))$. In \cite[Definition 4.1]{SwannRealisation}, the second author defined a t-structure on $\D^b(\mcal_\mathrm{perv}(X,\Q))$ that is called the \emph{constructible t-structure} in the quoted paper but that we will call the \emph{ordinary t-structure} in this paper for consistency. Denote its heart by $\mcal_\mathrm{ord}(X,\Q)$, then by \cite[Corollary 4.19]{SwannRealisation}, the natural map \[\D^b(\mcal_\mathrm{ord}(X,\Q))\to \D^b(\mcal_\mathrm{perv}(X,\Q))\] is an equivalence. Furthermore, the functor $\D^b(\mcal_\mathrm{ord}(X,\Q))\to\dgman(X^\mathrm{an},\Q)$ is t-exact with the ordinary t-structure on both sides. Therefore, using \cref{LimTStructures}, we get

\begin{prop}\label{tstrkbarre}
  Assume that $k$ is an algebraically closed subfield of $\C$.
  Let $X$ be a finite type $k$-scheme, there is a t-structure on $\DN(X,\Z)$ such that the rationalisation, endowed with its ordinary t-structure, and the Betti realisation are t-exact. Furthermore, we have a pullback square:
  \[\begin{tikzcd}
    \DN(X,\Z)^\heart \arrow[r] \arrow[d] 
      & \In \Sh_\mathrm{gm}(X^\mathrm{an},\Z)\arrow[d] \\
    \In  \mcal_\mathrm{ord}(X,\Q) \arrow[r]
      & \In \Sh_\mathrm{gm}(X^\mathrm{an},\Q).
  \end{tikzcd}\] 
\end{prop}

\begin{rem}\label{SHNori} 
  Using the ideas of this section, we can define a category of Nori motivic spectra which is akin to the category $\mathrm{SH}^{\et}(X)$ of \'etale motivic spectra when $X$ is a scheme of finite type over an \emph{algebraically closed} field of characteristic $0$. First, let $\bscr_X^{\mathbb{S}}$ be the algebra in $\mathrm{SH}^{\et}(X)$ representing Betti cohomology with $\mathbb{S}$-coefficients. Once again, Ayoub constructed in \cite[Theorem 1.93]{ayoubAnabelianPresentationMotivic2022} an equivalence $\Mod_{\bscr_{X}^{\mathbb{S}}}(\mathrm{SH}^{\et}(X))\to \In  \dgman(X^\mathrm{an},\mathbb{S})$, where $\dgman(X^\mathrm{an},\mathbb{S})$ is the full subcategory of the category $\dconsan(X^\mathrm{an},\mathbb{S})$ of constructible analytic spectral sheaves, made of objects of geometric origin and is endowed with an induced t-structure. 

  By analogy with \cref{PBN}, we can define $\nscr^{\mathbb{S}}_X$ to fit in the pullback   
  \[\begin{tikzcd}
    \nscr_X^{\mathbb{S}} \arrow[r] \arrow[d]
        \arrow[dr, phantom, very near start, "{ \lrcorner }"]
      & \bscr_X^{\mathbb{S}} \arrow[d] \\
    \nscr_{X,\Q} \arrow[r]
      & \bscr_{X,\Q}
\end{tikzcd}
\]
and then define $\mathrm{SH}_\mathrm{Nori}^{\et}(X):={\Mod}_{\nscr_X^{\mathbb{S}}}(\mathrm{SH}^{\et}(X))$. We now claim that the square 
\[\begin{tikzcd}
  \mathrm{SH}_\mathrm{Nori}^{\et}(X) \arrow[r] \arrow[d] 
      & \In \dgman(X^\mathrm{an},\mathbb{S}) \arrow[d] \\
    \DN(-,\Q) \arrow[r]
      & \In \dgman((-)^\mathrm{an},\Q)
\end{tikzcd}\] is cartesian. The full faithfulness of the map $\Psi_X\colon \mathrm{SH}_\mathrm{Nori}^{\et}(X)\to P(X)$ (where $P(X)$ denotes the pullback) follows from the general claim that if $A=B\times_D C$ is a pullback of commutative algebra objects in a stable $\infty$-category $\ccal$, the map $\Mod_A(\ccal) \to \Mod_\mathrm{B}(\ccal) \times_{\Mod_D(\ccal)}\Mod_C(\ccal)$ is fully faithful which can be seen using a direct computation of mapping spectra. To show that the map $\Psi_X$ is essentially surjective, note that $P(X)$ is endowed with an ordinary t-structure by \cref{LimTStructures} so that it suffices to reach the objects of its heart. We claim that it is the case. Indeed the functor \[-\otimes_\mathbb{S} \Z\colon \In \dgman(X^\mathrm{an},\mathbb{S})\to \In \dgman(X^\mathrm{an},\Z)\] yields the equivalence $\tau^{\geqslant 0}(-\otimes_\mathbb{S} \Z)$ between the hearts, therefore we have an equivalence $P(X)^\heart \to \DN(X,\Z)^\heart$.
Using that at the generic point $\eta$ for a projective smooth morphism $f$, the object $f_*\Z$ is the sum in $\DN(\eta,\Z)$ of its cohomology objects (by Deligne's criterion for degeneracy of spectral sequences \cite[Th\'eor\`eme 1.5]{zbMATH03256272}), one can prove the claim at the generic point, and then localisation enables one to prove that $\Psi_X$ is an equivalence in general.
Hence we have an ordinary motivic t-structure on $\mathrm{SH}_\mathrm{Nori}^\et(X)$ whose heart is the same as that of $\DN(X,\Z)$. 
% The diagram 
% \[ 
% \begin{tikzcd}
%     \DN(X,\Z) \arrow[r] \arrow[d] 
%       & \DN(X,\mathbb{S}) \arrow[d] \\
%      \DN(X,\mathbb{S})\arrow[r]
%       & P(X)
% \end{tikzcd}
% \]
% being commutative PAS FINI PAS FINI PAS FINI DIAGRAMME PAS DE SENS.
%Therefore, the natural map $\mathrm{SH}_\mathrm{Nori}^{\et}(X)^\heart\to P(X)^\heart $ is an equivalence and the essential surjectivity follows from \cref{PBBetti}.
%pour tout $N$, on a un pullback [ \begin{tikzcd} N \arrow[r]\arrow[d] & N \otimes_A B \arrow[d] \ N \otimes_A C \arrow[r] & N\otimes_A D \end{tikzcd} ] qui induit un carré cartésien en appliquant $\mathrm{map}_A(M,-)$. Puis, en utilisant les adjonction du produit tensoriel, on a que les $\mathrm{map}_A$ sont les mêmes que les $\mathrm{map}B \times{\mathrm{map}_D} \mathrm{map}_C$ et donc la pleine fidélité.
\end{rem}

\subsection{The ordinary t-structure on geometric Nori motives}

Let $X$ be a qcqs $\Q$-scheme of finite dimension. In this section we construct an ordinary t-structure on $\DNgm(X,\Z)$. Note that by \cref{Corgeomfields} we already have a t-structure on $\DNgm(k,\Z)$ for any subfield $k$ of $\C$.
By continuity and because a filtered colimit of stable $\infty$-categories with $t$-structures and t-exact functors has a natural t-structure, we can assume that $X$ is a finite type $\Q$-scheme. In this case we may use
 the punctual gluing of t-structures due to Vaish \cite[Proposition 3.8]{MR4109490}.
In order to use his result, we have to verify the condition he calls "continuity of the t-structure", which is the purpose of the next proposition. Let us first recall how the t-structure is defined:
\begin{defi}
  \label{punctuT}
  For each $x\in X$, denote by $i_x\colon \Spec(\kappa(x))\to X$ the inclusion of $x$ in $X$.
  An object $M\in\DNgm(X,\Z)$ is called 
  \begin{enumerate}
    \item \emph{punctually connective} if for each $x\in X$ we have $i_x^*M \in \D^b(\mcal(\kappa(x),\Z))^{\leqslant 0}$
    \item \emph{punctually coconnective} if for each $x\in X$ we have $i_x^!M \in \D^b(\mcal(\kappa(x),\Z))^{\geqslant 1}$.
  \end{enumerate}
  We denote by $\DNgm(X,\Z)^{\leqslant 0}$ and $\DNgm(X,\Z)^{\geqslant 1}$ the associated full subcategories.
\end{defi}
\begin{prop}
  Let $X$ be a finite type $\Q$-scheme. Let $x\in X$ and let $M_x\in\D^b(\mcal(\kappa(x),\Z))$.
  \begin{enumerate}
    \item If $M\in\D^b(\mcal(\kappa(x),\Z))^{\leqslant 0}$ there exists a nonempty open subset $U$ of $\overline{\{x\}}$ and an object $M\in\DNgm(U,\Z)^{\leqslant 0}$ such that $i_x^*M = M_x$.
    \item  If $M\in\D^b(\mcal(\kappa(x),\Z))^{\geqslant 1}$ there exists a nonempty open subset $U$ of $\overline{\{x\}}$ and an object $M\in\DNgm(U,\Z)^{\geqslant 1}$ such that $i_x^!M = M_x$.
  \end{enumerate}
\end{prop}
\begin{proof}
  In both cases we can assume that $x$ is a generic point of $X$ and that $X$ is irreducible, and even smooth up to replace $X$ by one of its dense open subsets. In this case both $i_x^*$ and $i_x^!$ agree.  Thus the proof of the two cases are almost the same and we only deal with the first point.
  By continuity, we know that there exist some $U_0\subset X$ nonempty and open, and $M_0\in\DNgm(U_0,\Z)$ such that $i_x^*M_0 = M_x$. Now the Betti realisation being conservative, we know that there exists some integer 
  $C\in\N^*$ such that if $p$ is a prime verifying that the multiplication by $p$ is zero on $M_0$, then $p\mid C$. Moreover, there is some other integer $a>0$ such that the Betti realisation of $M_0$ lives in degrees $[-a,a]$.
  Because $\Z_\ell$ is flat over $\Z$ and the functor $\D^b_c(X_\et,\Z_\ell)\to \D^b_c(X^\mathrm{an},\Z_\ell)$ are conservative and t-exact, we see that for each prime number $\ell$, the $\ell$-adic realisation of $M_0$ is concentrated in degrees $[-a,a]$. Let $U_1$ be a nonempty open subset of $U_0$ on which $M_0$ is dualisable. 
  For each $j\in [1,a]$ and each prime number $\ell$ dividing  $C$, because $i_x^*\HH^j(R_\ell(M_1))$ is zero and $\HH^j(R_\ell(M_1))$ is a local system on $(U_1)_\et$, there exists a nonempty open subset $U_{j,\ell}$ of $U_1$ such that the restriction to $U_{j,\ell}$ of $\HH^j(R_\ell(M_1))$ is zero.
  We let $U$ be the intersection of all of those $U_{j,\ell}$ and $M$ be the restriction of $M_1$ to $U$.
  Then we see that $R_\mathrm{B}(M)$ is connective, thus $M$ is punctually connective, as wanted.
\end{proof}

\begin{thm}\label{tord}
  %\todoeasy{Ajouté la comparaison avec 4.1.5}
  The conditions of \cref{punctuT} define a t-structure on $\DNgm(X,\Z)$ and the Betti realisation is t-exact. If $X$ is of finite type over an algebraically closed field embedded in $\C$, this t-structure is the same as the t-structure constructed in \cref{tstrkbarre}. Moreover, for any qcqs $\Q$-scheme $X$ of finite dimension this induces a t-structure on $\DNgm(X,\Z)$ such that all $\ell$-adic realisations are t-exact. 
\end{thm}

\begin{proof}
  The t-structure on $\DNgm(X,\Z)$ is given by the above proposition together with \cite[Proposition 3.8]{MR4109490}. The fact that the Betti realisation is t-exact can be checked on stalked at closed points, where it is true. Conservativity of the Betti realisation \cref{betticons} then implies the comparison with the t-structure of \cref{tstrkbarre}. Thus the $\ell$-adic realisation are also t-exact, and when we endow  $\DNgm(X,\Z)$, for $X$ a general finite dimentional qcqs $\Q$-scheme by continuity using \cref{coLimTStructures}, the $\ell$-adic realisation stay t-exact because if $X=\lim X_i$ is a cofiltered limit with affine transitions, although the functor \[\colim_i\dconspro(X_{i,\proet},\Z_\ell)\to\dconspro(X_\proet,\Z_\ell)\] is not an equivalence, it is t-exact.
\end{proof}

\begin{defi}
  \label{defi:Mord}
  Let $X$ be a finite dimentional qcqs $\Q$-scheme. The category of \emph{ordinary Nori motives} $\mcal_\mathrm{ord}(X,\Z)$ is the heart of the t-structure given by \cref{tord}. We call the latter the \emph{ordinary t-structure}.
\end{defi}
\begin{rem}
  All the construction above would work if we replace $\Z$ by a localisation of the ring of integers of a number field, finite or not over $\Q$. This is because in this case the category $\mcal(k,\Lambda)$ is well-behaved (because $\Lambda$ is a Dedekind ring), so that \cref{geomfields} and \cref{Corgeomfields} work in that case. For example, we have a t-structure with $\overline{\Z}$-coefficients.
\end{rem}
\begin{rem}
  As the $\ell$-adic realisation is also conservative by \cref{Conservativity}, any t-exactness property of the ordinary t-structure on $\ell$-adic sheaves transfers to the ordinary t-structure on Nori motives.
\end{rem}

\begin{rem}
  Let $k$ be a subfield of $\C$ and let $X$ be a $k$-variety.
  Arapura considered in \cite{MR2995668} an abelian category of motivic sheaves modelled on constructible sheaves. By construction, his category 
  $\mcal_{\mathrm{Arapura}}(X,\Z)$ has an universal property, so that there exists a faithful exact functor $\mcal_{\mathrm{Arapura}}(X,\Z)\to\mcal_\mathrm{ord}(X,\Z)$ compatible with the Betti realisations. We do not know much about this functor, except that it is an equivalence for $X=\Spec(k)$.
\end{rem}

\subsection{The derived category of ordinary Nori motives}
We now prove that $\DNgm$ is the derived category of ordinary Nori motives following Nori \cite{MR1940678}. To that end, we first need:
\begin{lem}\label{EtDescDb}
  The functor $\D^b(\mcal_\mathrm{ord}(-,\Lambda))$ is a finitary \'etale hypersheaf on qcqs schemes of finite dimension.
\end{lem}
\begin{proof}
  The finitarity of the functor follows from \cref{colimDb} and the continuity of $\DNgm$. Because for $f$ an \'etale map, the functors $f^*$ and $f_!$ are t-exact, they exist on $\D^b(\mcal_\mathrm{ord}(-,\Lambda))$ as an adjunction, compatible with the realisation functor \[\D^b(\mcal_\mathrm{ord}(-,\Lambda))\to \DNgm(-,\Lambda)\] which is conservative. Thus, because $\DNgm(-,\Lambda)$ is an \'etale hypersheaf, the functor $\mcal_\mathrm{ord}(-,\Lambda)$ is an \'etale sheaf: it is a sub presheaf of $\DNgm(-,\Lambda)$ and the cohomological amplitude with respect to the ordinary t-structure can be tested \'etale-locally. We can now apply the same proof as \cite[Theorem 5.24]{SwannRealisation} to conclude.
\end{proof}
\begin{thm}\label{DbNat}
  Let $\Lambda$ be a regular ring which is flat over $\Z$. Let $X$ be a qcqs $\Q$-schemes of finite dimension. Then, the canonical functor 
  \[\D^b(\mcal_\mathrm{ord}(X,\Lambda))\to \DNgm(X,\Lambda)\] is an equivalence.
\end{thm}
\begin{proof}
  By \cref{EtDescDb}, continuity and Zariski hyperdescent for $\DNgm(-,\Lambda)$, we can assume that $X$ is of finite type over $\Q$ and separated.
  Note that by \cite[Lemma 1.4]{MR0923133}, to prove the theorem, it is sufficient to prove that for any pair of Nori motives $M,N$ in $\mcal_\mathrm{ord}(X,\Lambda)$ and any integer $i>0$, the map 
  \[\mathrm{Ext}^i_{\mcal_\mathrm{ord}(X,\Lambda)}(M,N)=\Hom_{\D^b(\mcal_\mathrm{ord}(X,\Lambda))}(M,N[i])\to \Hom_{\DNgm(X,\Lambda)}(M,N[i])\] is an isomorphism, which, given the characterisation of the $\mathrm{Ext}^i$, is equivalent to the functor \[\mathrm{E}_M^i:=\Hom_{\DNgm(X,\Lambda)}(M,(-)[i])\colon \mcal_{\mathrm{ord}}(X,\Lambda)\to\mathrm{Ab}\] being effaceable for each $i>0$ and each $M$ in $\mcal_\mathrm{ord}(X,\Lambda)$. More explicitly, this means that for every Nori motive $N$ in $\mcal_\mathrm{ord}(X,\Lambda)$, there exists a monomorphism $f\colon N\to P$ whose image under our functor $\mathrm{E}^i_M$ vanishes. 
  
  The proof of this fact is then very similar to the proof the second author gave, following Nori, in the case of Nori motives with rational coefficients (see \cite{MR1940678,SwannRealisation}). Until the end of this proof, we will call \emph{admissible} an Nori motive $M$ in $\mcal_{\mathrm{ord}}(X,\Lambda)$ such that for all $i>0$ the functor $\mathrm{E}_M^i$ is effaceable. 

  By remarking that all the proofs that lead to \cite[Corollary 4.16]{SwannRealisation} work verbatim with integral coefficients (because we only use the four operations $f^*,f_*,f^!,f_!$, geometric arguments and the vanishing \cite[Lemma 2.6]{SwannRealisation} works with integral coefficients because Nori proved it in this generality in \cite[Proposition 2.2]{MR1940678}), we know that the constant object $\Lambda_X$ is admissible on $X$. To be more precise, we first prove by induction on $n$ the admissibility of the constant object on $\A^n_\Q$ (the case $n=0$ is true because we already know the theorem in that case by \cref{geomfields}, while the induction uses well-chosen projection to $\A^{n-1}$ and the already cited \cite[Lemma 4.6]{SwannRealisation}). Then, using Noether normalisation one can prove that the result is true for any affine scheme. Finally, because our scheme is separated, we can do an induction on the numbers of affine open subsets needed to cover $X$ and deduce that $\Lambda_X$ is indeed admissible on $X$
  
  The next step is to prove that any torsion-free lisse sheaf $L$ on $X$ is admissible. This follows from the fact that because $L$ is dualisable, the functor $-\otimes L$ is t-exact because $L$ is torsion-free (this can be checked after realisation and then on stalks so that this reduces to the same statement about modules) and left adjoint of a t-exact functor; thus by \cite[Corollary 4.13]{SwannRealisation} the object $\Lambda\otimes L = L$ is admissible. Using \cite[Corollary 4.13]{SwannRealisation} again, the functor $f_!$ for $f$ \'etale preserves admissibility: the functors $f^*$ and $f_!$ are t-exact because they are t-exact after taking the Betti realisation. Moreover by \cite[Lemma 4.10]{SwannRealisation}, if one has a short exact sequence \[0\to M'\to M\to M'' \to 0\] with $M'$ and $M''$ admissible, then $M$ is admissible. Now if $L$ is a lisse motive, then we have the short exact sequence 
  \[0\to L_\mathrm{tors} \to L\to L_\mathrm{fr}\to 0\] with $L_\mathrm{tors}$ the biggest torsion subobject and $L_\mathrm{fr}$ torsion-free. Note that $L_\mathrm{tors}$ and $L_\mathrm{fr}$ are also lisse because they are after taking the Betti realisation. Hence $L_\mathrm{fr}$ is admissible. 
  
  We claim that $L_\mathrm{tors}$ is also admissible. As it is torsion, it is in fact of the form $\iota(F)$ for a lisse \'etale sheaf on $X$: it is of the form $\iota(F)$ because of \cref{rigtorscomplet} and $F$ is lisse because the functor $\D(X_\et,\Lambda)\to\D(X^\mathrm{an},\Lambda)$ is conservative and the image of $F$ in $\DN(X,\Lambda)$ is dualisable, and thus its image in $\D(X^\mathrm{an},\Lambda)$ is also dualisable which yields the dualisability of $F$ by conservativity. Choose some $f\colon U\to X$ \'etale such that $f^*F = N$ is a constant $\Lambda$-module. We can choose a surjection $\Lambda^n\to N$, so that by t-exactness of $f_!$, we have a composition of surjections $ f_!\Lambda^n\twoheadrightarrow  f_!f^*F\twoheadrightarrow F$. Let $K$ be the kernel of this composed surjection. It is also a lisse and torsion-free \'etale sheaf: the ring $\Lambda$ is flat over $\Z$ which is a Dedeking domain, so that a subobject of a torsion-free object is again torsion-free. Thus, in $\D^b(\mcal_\mathrm{ord}(X,\Lambda))$ we have an exact triangle 
  \[\iota(K)\to \iota(f_!\Lambda^n)\to L_\mathrm{tors}.\] Because $\iota(K)$ and $\iota(f_!\Lambda^n)$ are admissible, for any motive $P\in\mcal_\mathrm{ord}(X,\Lambda)$ the maps 
  \[\map_{D^b(\mcal(X,\Lambda))}(\iota(K),P)\to \map_{\DNgm(X,\Lambda)}(\iota(K),P)\] and 
  \[\map_{D^b(\mcal(X,\Lambda))}(\iota(f_!\Lambda^n),P)\to \map_{\DNgm(X,\Lambda)}(\iota(f_!\Lambda^n),P)\] are equivalences. Thus the analogous map for $L_\mathrm{tors}$ is an equivalence so that $L_\mathrm{tors}$ is also admissible. This implies that $L$ is admissible. Now, the same proof as \cite[Theorem 4.18]{SwannRealisation} (which is the same proof as \cite[Proposition 3.10]{MR1940678}) gives by d\'evissage that any motive is admissible, finishing the proof.
\end{proof}

\subsection{Applications: motivic Leray spectral sequence; arc descent}
The existence of ordinary integral Nori motives implies that the Leray spectral sequence, with integral coefficients, is motivic; this is a generalisation of \cite[Section 5.4]{ivorraFourOperationsPerverse2022} and \cite{MR2178703}.
\begin{prop}\label{Leray}
  Let $X$ be a finite type $k$-scheme with $k$ a subfield of $\C$.
  Let $\fcal$ be an object of the abelian category $\Sh_\mathrm{cons}(X^\an,\Z)$ of constructible sheaves of abelian groups over $X$. 
  
  Assume that $\fcal$ is motivic, meaning that it is the image of some integral Nori motive under the Betti realisation. Let \[X\xrightarrow{f}Y\xrightarrow{g}Z\] be morphisms of finite type $\C$-schemes. Then the Leray spectral sequence 
  \[\mathrm{E}_2^{p,q}=\mathrm{R}^pg_*\mathrm{R}^qf_*\fcal\Rightarrow \mathrm{R}^{p+q}(g\circ f)_*\fcal\] is the image of a spectral sequence in $\mcal_\mathrm{ord}(Z)$.
\end{prop}
\begin{proof}
  This is a consequence of the existence of the ordinary t-structure and the functoriality of pushforwards. Indeed, let $M\in\mcal_\mathrm{ord}(X)$ and consider the filtered object 
  \[\dots g_*\tau^{\leqslant i+1}f_*M \to g_*\tau^{\leqslant i}f_*M \to g_*\tau^{\leqslant i-1}f_*M\to \dots\] of $\DNgm(Z,\Z)$ (the truncation functors are for the ordinary t-structure). By \cite[Proposition 1.2.2.14]{lurieHigherAlgebra2022} (applied in $\In\DNgm(Z,\Z)$ so that the category considered has all sequential colimits and a t-structure compatible with sequential colimits by \cite[Lemma C.2.4.3]{lurieSpectralAlgebraicGeometry}), we obtain a converging spectral sequence in $\mcal_\mathrm{ord}(Z)$ whose first page reads 
  $$E_1^{p,q} = \HH^{p+q}(\mathrm{cofib}(g_*\tau^{\leqslant p+1}f_*M \to g_*\tau^{\leqslant p}f_*M))$$ and converges to $\HH^{p+q}(\colim_i g_* \tau^{\leqslant i}f_*M)$. Using that the extension of $g_*$ to $\In\DNgm(Y,\Z)$ commutes with colimits by definition, and the fact that $f_*M$ is bounded, we can rewrite the abutement of the spectral sequence as $\HH^{p+q}(g_*f_*M)$. 
  The triangle for the truncation functors gives that $\mathrm{cofib}(g_*\tau^{\leqslant p+1}f_*M \to g_*\tau^{\leqslant p}f_*M) \simeq g_*\HH^p(f_*M)[-p]$. Thus, reindexing $E_1$ to $E_2$ to fit the general convention, we obtain the motivic Leray spectral sequence 
  \[E_2^{p,q}=\HH^{q}g_*\HH^pf_*M\Rightarrow \HH^{p+q}(g\circ f)_*M.\]

Applying the Betti realisation to the filtered object \[\dots g_*\tau^{\leqslant i+1}f_*M \to g_*\tau^{\leqslant i}f_*M \to g_*\tau^{\leqslant i-1}f_*M\to \dots\]
and letting $\fcal$ be the Betti realisation of $M$ yields a filtered object \[\dots g_*\tau^{\leqslant i+1}f_*\fcal \to g_*\tau^{\leqslant i}f_*\fcal \to g_*\tau^{\leqslant i-1}f_*\fcal \to \dots\]
by t-exactness of the Betti realisation (see \cref{tord}) and its compatibility with pushforwards (\cref{Betti real 6FF}). The usual Leray spectral sequence is by definition obtained from this filtered object using the same tools as above. Unwinding the definitions, this proves that it is the image of the motivic Leray spectral sequence through the Betti realisation. 
\end{proof}

As an application of the existence of the ordinary t-structure, we prove that Nori motives afford arc-descent.
\begin{thm}\label{ArcDesc}
The functor 
 $\DNgm(-,\Z)$
 satisfies arc-hyperdescent over qcqs $\Q$-schemes of finite dimension. This means that for any arc-hypercovering $X_\bullet \to X$ (\cite[Definition 1.2]{MR4278670}) where  $X$ and each $X_n$ are $\Q$-schemes of finite dimension, the functor 
 \[\DNgm(X,\Z)\to\lim_{[n]\in\Delta}\DNgm(X_n,\Z)\] is an equivalence.
 \end{thm}
 \begin{proof}
   As \[\DNgm(X,\Z) = \bigcup_n \DNgm(X,\Z)^{[-n,n]}\] we see that to prove that $\DNgm(-,\Z)$ is an arc-hypersheaf, it suffices to show that $\DNgm(-,\Z)^{[-n,n]}$ is an arc-hypersheaf (hence even an arc-sheaf), where the truncation is taken for the ordinary t-structure.

   The presheaf $\DNgm(-,\Z)^{[-n,n]}$ takes values in $\mathrm{Cat}_{2n+1}$ which is compactly generated by cotruncated objects (see \cite[Example 3.6 (3)]{MR4278670}). We consider \[\fcal\colon (\Sch_\Q^\mathrm{qcqs})^\op\to \mathrm{Cat}_{2n+1}\] its left Kan extension to $(\Sch_\Q^\mathrm{qcqs})^\op$ the category of (non necessarily finite-dimensional) qcqs $\Q$-schemes. By \cref{megathmcons} we have that $\fcal(X)= \DNgm(X,\Z)^{[-n,n]}$ for $X$ of finite dimension. We may now apply the main theorem of \cite{MR4278670}: a finitary $v$-sheaf taking values in an $\infty$-category which is compactly generated by cotruncated objects is an arc-sheaf if and only if it is aic-$v$-excisive. By definition, $\fcal$ is finitary.

   First, we want to prove that $\fcal$ is a $v$-sheaf. Given a $v$-covering $f\colon Y\to X$, we can write it as $\lim_i (f_i\colon Y_i\to X_i)$ with $X_i$ and $Y_i$ of finite type over $\Q$. Moreover, as the sheaf condition for $f$ and $\fcal$ is finite because it is a totalisation in an $\infty$-category which is compactly generated by cotruncated objects, it is compatible with the filtered colimit of the continuity. Hence, it suffices to show that $\DNgm(-,\Z)^{[-n,n]}$ is a $v$-sheaf when restricted to finite type $\Q$-schemes, but for Noetherian schemes the $v$-topology and the $h$-topology agree, so that this is follows from \cref{hdescGro} and \cref{megathmcons}\footnote{In fact in our setting there is a simpler proof thanks to the t-structure: one can do as in \cite[Proposition 5.11]{MR4278670}, in which they reduce $v$-descent to descent along proper surjective maps thanks to \'etale hyperdescent, and then the proper base change theorem implies the result, with a combination of Barr-Beck-Lurie theorem and the finitary property. This proof essentially dates back to Deligne in \cite{MR1106898}.}.

   Next, we prove that $\fcal$ satisfies aic-$v$-excision. Recall that this means that absolulety integrally  closed valuation ring $V$ which is a $\Q$-algebra, together with a prime ideal $\mathfrak{p}\subset V$, the square 
   % https://q.uiver.app/#q=WzAsNCxbMCwwLCJcXGZjYWwoXFxTcGVjKFYpKSJdLFswLDEsIlxcZmNhbChcXFNwZWMoVl9cXG1hdGhmcmFre3B9KSkiXSxbMSwwLCJcXGZjYWwoXFxTcGVjKFYvXFxtYXRoZnJha3twfSkpIl0sWzEsMSwiXFxmY2FsKFxcU3BlYyhcXGthcHBhKFxcbWF0aGZyYWt7cH0pKSkiXSxbMCwxXSxbMSwzXSxbMCwyXSxbMiwzXV0=
\[\begin{tikzcd}
	{\fcal(\Spec(V))} & {\fcal(\Spec(V/\mathfrak{p}))} \\
	{\fcal(\Spec(V_\mathfrak{p}))} & {\fcal(\Spec(\kappa(\mathfrak{p})))}
	\arrow[from=1-1, to=1-2]
	\arrow[from=1-1, to=2-1]
	\arrow[from=1-2, to=2-2]
	\arrow[from=2-1, to=2-2]
\end{tikzcd}\] is cartesian. Let $(V,\mathfrak{p})$ be such datum. By \cite[Lemma 2.22]{MR4278670}, we can write $V = \colim_i V_i$ as a filtered colimits of absolulety integrally closed valuation subrings $V_i$ such that each $V_i$ has finite rank (this is equivalent to being finite-dimensional by \cite[Lemma 2.3.1]{zbMATH02208665}) and each map $V_i\to V_j$ is faithfully flat. As filtered colimits commute with finite limits and $\fcal$ is finitary, we may assume that $V$ is finite-dimensional (by replacing $V$ by $V_i$ and $\mathfrak{p}$ by $\mathfrak{p}\cap V_i$). In the setting of $\Q$-schemes of finite dimension, we can prove more generally Minor excision: assume that 
 \[\begin{tikzcd}W \ar[d,"g"]\ar[r,"k"] & Y \ar[d,"f"] \\
     Z \ar[r,"i"] & X 
     \end{tikzcd}\] is a Milnor square (\cite[Definition 1.11]{MR4278670}) of qcqs $\Q$-schemes of finite dimension, we claim that applying $\DNgm(-,\Z)^{[-n,n]}$ to this square yields a cartesian square. We first prove it for $\DNgm(-,\Z)$. Firstly, the functor
     \[F\colon \DNgm(X,\Z)\to \DNgm(Z,\Z)\times_{\DNgm(W,\Z)}\DNgm(Y,\Z)\] is fully faithful. This can be checked after tensoring the mapping spectra with $\Q$ and $\Z/p$ for all primes $p$. The mod $p$ part follows from rigidity (\cite[Theorem 3.1]{bachmannrigidity}) and \cite[Theorem 5.13]{MR4278670}. Rationally, the $-\otimes\Q$ goes inside the mapping spectra thanks to \cref{megathmcons} and we can use \cite[Theorem 5.6]{MR4319065}. Indeed, as in \cite[Remark 4.7]{MR4319065}, the functor $\DN(-,\Q) = \mathrm{Mod}_{\mathscr{N}_\Q}(\DM(-,\Q))$ satisfies Milnor excision because $\DM(-,\Q)$ satisfies it by \cite[Theorem 5.6]{MR4319065}. Now, the compact objects of a finite limit of compactly generated $\infty$-categories are the finite limit of the compact objects, thus we even have that 
     \[\DNgm(X,\Q)\simeq \DNgm(Z,\Q)\times_{\DNgm(W,\Q)}\DNgm(Y,\Q).\] 
     
    Secondly, we prove essential surjectivity. As we already proved full faithfullness, we are allowed to use d\'evissage and retracts of objects of the essential image also are in the essential image.
    We first claim that the rationalisation of the $\infty$-category $\DNgm(Z,\Z)\times_{\DNgm(W,\Z)}\DNgm(Y,\Z)$ is $\DNgm(Z,\Q)\times_{\DNgm(W,\Q)}\DNgm(Y,\Q)$. Indeed there is a fully faithful functor 
     \[(\DNgm(Z,\Z)\times_{\DNgm(W,\Z)}\DNgm(Y,\Z))\otimes_{\Perf_\Z}\Perf_\Q\to \DNgm(Z,\Q)\times_{\DNgm(W,\Q)}\DNgm(Y,\Q),\] as the rationalisation is the idempotent completion of the $\infty$-category where we replace all mapping spectra by their rationalisation. This functor is essentially surjective because the target is idempotent complete. In particular, using \cite[Recolletion 2.13]{balmer-gallauer}, we see that given $N\in \DNgm(Z,\Z)\times_{\DNgm(W,\Z)}\DNgm(Y,\Z)$, there exists $M\in \DNgm(X)$ and a map $N\oplus N[1]\to F(M)$ which is an equivalence after rationalisation. In particular, the cofiber $C$ of this map is a (fiber product of) geometric motives that satisfy $C\otimes\Q=0$. By \cref{megathmcons} and the fact that we are looking at a finite limit, this implies that there is some $n\in\N^*$ such that the multiplication by $n$ on $C$ is the zero map. In particular, we have that $C/n\simeq C\oplus C[1]$, and it suffices to show that $C/n$ is in the image of $F$. By doing a d\'evissage, we may assume that $n=p$ is a prime number. In that case, we have that $C/p$ comes from $$\DNgm(Z,\F_p)\times_{\DNgm(W,\F_p)}\DNgm(Y,\F_p)\simeq \DNc(Z,\F_p)\times_{\DNc(W,\F_p)}\DNc(Y,\F_p),$$ by \cref{megathmcons}. Using rigidity \cite[Theorem 3.1]{bachmannrigidity} and Milnor excision for constructible \'etale sheaves \cite[Theorem 5.13]{MR4278670}, this object $C/p$ comes from $\DNc(X,\F_p)\simeq \DNgm(X,\F_p)$. This means that we have some $M\in\DNgm(X,\F_p)$ such that the image of $M$ in $\DNgm(Z,\F_p)\times_{\DNgm(W,\F_p)}\DNgm(Y,\F_p)$ is $C/p$. Because $\DNgm(X,\F_p)$ is generated as a thick subcategory of itself by $M'/p$ with $M'\in\DNgm(X,\Z)$, this implies that $C/p$, as an object of $\DNgm(Z,\Z)\times_{\DNgm(W,\Z)}\DNgm(Y,\Z)$, lies in the thick subcategory generated by the image of $F$, which is the image of $F$ as the latter is fully faithful. This proves that $\DNgm(-,\Z)$ satisfies Milnor excision for this square of finite dimension $\Q$-schemes.
    
     To prove that $\DNgm(-,\Z)^{[-n,n]}$ satisfies excision as well, it suffices to show that if $M$ belongs to $\DNgm(X,\Z)$ and is such that $f^*M$ and $i^*M$ are in cohomological degrees $[-n,n]$ then in fact $M$ is also in degrees $[-n,n]$. This is true because the map $Z\sqcup Y\to X$ is surjective hence the hypothesis imply that for any point $x$ of $X$, the motive $x^*M$ in $\DNgm(x,\Z)$ is concentrated in degrees $[-n,n]$, so that $M$ is also concentrated in degrees $[-n,n]$.
 \end{proof}

\begin{rem}
  Conjecturally, the category $\DMgm(X,\Z)$ has an ordinary t-structure after rationalisation. In that case the functor $\DMgm(X,\Z)\to\DNgm(X,\Z)$ is an equivalence because it is an equivalence rationalised and mod $n$ for any nonzero integer $n$, so that $\DMgm(-,\Z)$ is a finitary arc-hypersheaf and the category $\DMgm(X,\Z)$ has an ordinary t-structure without needing to rationalise. In fact, it would suffice to prove that $\DMgm(\Spec(\Q),\Q)$ has a motivic t-structure (see \cref{rem:Pridham}).
\end{rem}

\section{Perverse Nori motives.}
\subsection{The perverse t-structure}
In this subsection, we define the perverse t-structure on Nori motives. The definition is very similar to that of the perverse t-structure on constructible complexes of $\Z/\ell^n\Z$-modules given in \cite[Section 2.2]{MR0751966}. We also prove that the derived category of perverse Nori motives is $\DNgm$, which is in the spirit of \cite{MR0923133}.
\begin{prop}\label{ExistPerv} Let $X$ be an excellent $\Q$-schemes of finite dimension endowed with a dimension function $\delta$\footnote{See \cite[Expos\'e XIV, D\'efinition 2.1.8]{travauxgabber} for the definition of a dimension function.}.
  We let \[{}^p\DNgm(X,\Z)^{\leqslant 0} :=\{M \in \DNgm(X,\Z) \mid \forall x \in X, i_x^*(M)\leqslant -\delta(x)\},\]
  \[{}^p\DNgm(X,\Z)^{\geqslant 0} :=\{M \in \DNgm(X,\Z) \mid \forall x \in X, i_x^!(M)\geqslant -\delta(x)\}.\]

  Then, the pair $({}^p\DNgm(X,\Z)^{\leqslant 0}, {}^p\DNgm(X,\Z)^{\geqslant 0})$ defines a t-structure on $\DNgm(X,\Z)$ such that the $\ell$-adic realisation functor \[R_\ell\colon \DNgm(X,\Z)\to \dconspro(X_\proet,\Z_\ell)\] is t-exact when the right-hand side is endowed with its perverse t-structure. 
\end{prop}
\begin{proof}
  The proof is the same as in \cite[Section 2.2]{MR0751966} but we include a sketch for the reader's convenience.
  Consider couples $(\scal,\lcal)$ where $\scal$ is a finite stratification of $X$ by locally closed connected regular subschemes, and $\lcal$ is the data, for each stratum $Z$ of $\scal$ of a finite set $\lcal(Z)$ of dualisable Nori motives on $Z$. We denote by $\DN_{(\scal,\lcal)}(X,\Z)$ the full subcategory of $\DNgm(X,\Z)$ made of $(\scal,\lcal)$-constructible Nori motives, that is the motives $M$ such that, for each stratum $Z$ of $\scal$ and each integer $n$, $\HH^n(M|_{Z})$ is isomorphic to an iterated extension of elements of $\lcal(Z)$. 
  
  Any such couple can be refined into a couple $(\scal,\lcal)$ such that if $Z$ is a stratum in $\scal$ and $N$ belongs to $\lcal(Z)$, letting $j\colon Z\to X$ be the inclusion, the Nori motive $j_*(N)$ is $(\scal,\lcal)$-constructible. 

  This assumption ensures that we can define a t-structure on $\DN_{(\scal,\lcal)}(X,\Z)$ by gluing (in the sense of\cite[Section 1.4]{MR0751966})the t-structure on each stratum $Z$ by $-\delta(Z)$. If $(\scal,\lcal)$ refines $(\scal',\lcal')$, the t-structure on $\DN_{(\scal,\lcal)}(X,\Z)$ induces the t-structure on $\DN_{(\scal',\lcal')}(X,\Z)$; this is a consequence of the absolute purity property for Nori motives: 
  \begin{lem}
    Let $i\colon Z\to X$ be a closed immersion everywhere of codimension $c$ between regular schemes and let $M$ be a dualisable object of $\DN(Z,\Z)$, then the natural map \[i^*(M)(-c)[-2c] \to i^!(M)\] is an equivalence. 
  \end{lem}
  \begin{proof}
    This follows from the same statement for $\ell$-adic sheaves \cite[Expos\'e XVI, Th\'eor\`eme 3.1.1]{travauxgabber} and from the conservativity of the $\ell$-adic realisation \cref{Conservativity} and its compatibility with the six functors over quasi-excellent schemes \cref{LAdicSmall}. 
  \end{proof}
  
  To finish the proof of the proposition, note that \[\DNgm(X,\Z)=\colim_{(\scal,\lcal)}\ \DN_{(\scal,\lcal)}(X,\Z)\] and therefore we get a t-structure on $\DNgm(X,\Z)$ by \cref{coLimTStructures}. The t-exactness of the realisation follows from the fact that the perverse t-structure is defined in the same way by gluing along stratifications in \cite[Section 2.2]{MR0751966}. The characterisation of the positive and negative parts of the t-structure then follows from the t-exactness of the realisation combined with its conservativity and the analogous characterisation of the positive and negative parts of the perverse t-structure on $\dconspro(X_\proet,\Z_\ell)$ which is \cite[2.2.12]{MR0751966}. 
\end{proof}
\begin{rem}
  Along the way we proved that the perverse t-structure on the subcategory $\DN_\mathrm{lisse}(X,\Z)$ of $\DNgm(X,\Z)$ made of dualisable objects is the ordinary t-structure shifted by $-\delta(X)$. 

  This result also implies that any t-exactness property of the perverse t-structure on $\ell$-adic sheaves is true for the perverse t-structure on Nori motives. In particular, the Lefschetz Hyperplane Theorem (also known as Artin Vanishing) is true for Nori motives. The Hard Lefschetz theorem for Nori motives is also true using similar arguments.
\end{rem}

\begin{defi}
  \label{defi:Mperv}
  Let $X$ be an excellent $\Q$-schemes of finite dimension endowed with a dimension function. The \emph{abelian category of perverse Nori motives} is the heart of the perverse t-structure on $\DNgm(X,\Z)$. We denote it by $\mcal_\mathrm{perv}(X,\Z)$. 
\end{defi}
\begin{rem}\label{perv+}
  One of the main feats of the perverse t-structure with rational coefficients is that it is preserved by Verdier duality. This is famously not the case with integral coefficients (note for instance that $\map_{\D(\Z)}(\Z/n\Z,\Z)=\Z/n\Z[-1]$). However, if $\ell$ is a prime number there is a t-structure on $\dconspro(X_\proet,\Z_\ell)$ that is exchanged with the perverse t-structure by duality (see \cite[4.0 (a)]{MR0751966}). 
  
  We can mimic its definition to get a t-structure on $\DNgm(X,\Z)$ as follows; assume that $\Z$ is a Dedekind domain and define \[{}^{\mathrm{ord}^+}\DNgm(X,\Z)^{\leqslant 0}:=\{ M \in \DNgm(X,\Z) \mid M\leqslant_\mathrm{ord} 1 \text{ and } \HH^1(M)\otimes_\Z \Q=0\}\]
  \[{}^{\mathrm{ord}^+}\DNgm(X,\Z)^{\geqslant 0}:=\{ M \in \DNgm(X,\Z) \mid M\geqslant_\mathrm{ord} 0 \text{ and } \HH^0(M) \text{ is torsion-free} \}\]
  This is a t-structure on $\DNgm(X,\Z)$; then by gluing as before this t-structure along stratifications (using the same shifts) we get a t-structure $({}^{p^+}\DNgm(X,\Z)^{\leqslant 0},{}^{p^+}\DNgm(X,\Z)^{\geqslant 0})$ whose heart we denote by $\mcal_\mathrm{perv}^+(X,\Z)$. 

  Because the six functors are compatible with the $\ell$-adic realisation and because same statement is true after realisation by \cite[3.3.4]{MR0751966}, the Verdier duality functor $\mathbb{D}_X$ of \cref{Verdier} exchanges the subcategory ${}^{p^+}\DNgm(X,\Z)^{\leqslant 0}$ (resp. ${}^{p^+}\DNgm(X,\Z)^{\geqslant 0}$) with ${}^{p}\DNgm(X,\Z)^{\geqslant 0}$ (resp. ${}^{p^+}\DNgm(X,\Z)^{\leqslant 0}$) and therefore it exchanges $\mcal_\mathrm{perv}(X,\Z)$ with $\mcal_\mathrm{perv}^+(X,\Z)$. 
\end{rem}

% \begin{rem}\label{j!*}
%   Using \cite[D\'efinition 1.4.22]{MR0751966}, we can also define the intermediate extension functor for Nori motives: more precisely letting $j\colon U\to X$ be an open immersion where $X$ is as above, we define $j_{!*}=\im(\HHp^0j_!\to \HHp^0j_*)$. When defined, its images through the $\ell$-adic and Betti realisations are the usual intermediate extension functors because both realisation are t-exact and compatible with the six operations. 
% \end{rem}

\subsection{Universal motives and the derived category of perverse Nori motives}
By the universal property of the bounded derived category, there is a canonical realisation functor 

\begin{equation}
  \label{realMperv}
  \D^b(\mcal_\mathrm{perv}(X,\Z))\to \DNgm(X,\Z).
\end{equation}

This functor is an equivalence after tensorisation by $\Perf_\Q$. We have the following rigidity result to compute its torsion:

\begin{prop}
  \label{torsMp}Let $X$ be an excellent $\Q$-schemes of finite dimension endowed with a dimension function $\delta$.
  Let $n$ be a prime number.
  Consider the composition  
\[\iota_n\colon \mathrm{Perv}(X_\et,\Z/n)\subseteq\dconset(X_\et,\Z/n) \xrightarrow{f}\dconset(X_\et,\Z)_{tors}\xrightarrow{\iota} \DNgm(X,\Z)_{tors}\] where $\mathrm{Perv}(X_\et,\Z/n)$ is the category of \'etale perverse sheaves and $f$ is the forgetful functor. It lands in $\mcal_\mathrm{perv}(X,\Z)[n]$ and induces an equivalence \[\mathrm{Perv}(X_\et,\Z/n)\xrightarrow{\sim} \mcal_\mathrm{perv}(X,\Z)[n].\]
\end{prop}
\begin{proof}
  By \cref{rigtorscomplet} the last functor is an equivalence. Moreover as $\DNgm(X,\Z)_{tors}$ is the kernel of the perverse t-exact functor $-\otimes_\Z\Q$, it is endowed with a perverse t-structure, whose heart is the abelian category of torsion perverse motives.

  First note that the forgetful functor $f$ is t-exact: indeed, the heart of the perverse t-structure on $\dconset(X_\et,\Z)$ consists of those complexes $K$ such there is a stratification of $X$ by nil-regular subschemes such that on each stratum $S$ we have $K|_S \simeq L[-\delta(X)]$ for a lisse sheaf $L$; the image through the forgetful functor of a lisse sheaf of $\Z/n$-modules is a lisse sheaf of $\Z$-modules, whence the functor $f$ is indeed t-exact. 

  This implies that the functor $\iota_n$ lands in $\mcal_\mathrm{perv}(\Z)[n]$. Furthermore, the restriction of the forgetful functor to the heart affords a left adjoint, namely $\HHp^0(-\otimes_\Z \Z/n)$; it is in fact the underived tensor product with $\Z/n$ that sends a perverse sheaf $K$ to the cokernel of the map $K\xrightarrow{\times n} K$. With this, it is clear that  if $M$ belongs to $\mathrm{Perv}(X_\et,\Z/n)$, we have \[\HH^0(f(M)\otimes_\Z\Z /n)=M\] and for any object of $\dconset(X_\et,\Z)_{tors}^\heartsuit$ which is of $n$-torsion, we have \[f(\HH^0(N\otimes_\Z\Z/n))=N.\] Whence, the functor $\iota_n$ indeed induces an equivalence. 
\end{proof}
 Unfortunately, we do not know if the abelian category $\mcal_\mathrm{perv}$ has enough torsion-free sheaves. If it was true, we would use \cref{torsderive} to conclude that the functor \cref{realMperv} is fully faithful thanks to \cref{torsMp} and Beilinson's Theorem on the derived category of perverse sheaves \cite{MR0923133}; this would then imply that this functor is an equivalence. We have a partial answer in that direction, that requires to introduce the universal category of perverse motives. Let $k$ be a subfield of the complex numbers (a more general field can be considered, but the following results for those fields can be reduced to subfields of $\C$ by the continuity properties of perverse categories considered along fields).

 Recall first the theory of universal abelian categories in the form used in \cite[Section 2.1]{ivorraFourOperationsPerverse2022}:

\begin{constr}\label{univCatdef} Given a triangulated category $\mathrm{T}$, one first consider the category of finite presentation modules over $\mathrm{T}$: it is the full subcategory $\mathrm{fp}(\mathrm{T})\subset \Sh_{\mathrm{add}}(\mathrm{T},\mathrm{Ab})$ of additive functors $\fcal$ which fit in a short exact sequence 
\[y_M\to y_N\to \fcal\to 0\] where $y_M$ is the image of $M$ under the Yoneda embedding. Any cohomological functor on $\mathrm{T}$ will factor through $\mathrm{fp}(\mathrm{T})$. Thus, starting with a cohomological functor 
\[\HH \colon\mathrm{T}\to\acal\] with values in an abelian category $\acal$,
we have a factorisation of $\HH$ of the form 
\[\mathrm{T}\to \mathrm{fp}(\mathrm{T})\xrightarrow{\bar{\HH}}\acal.\]
Denote by $I$ the kernel of $\bar{\HH}$. It is the full subcategory of $\mathrm{fp}(\mathrm{T})$ consisting of objects $M$ such that $\bar{\HH}(M)=0$. The universal factorisation of $\HH$ is the factorisation 
\[\mathrm{T} \to \mcal(\mathrm{T},\HH):=\mathrm{fp}(\mathrm{T})/I \xrightarrow{\HH_\mathrm{univ}} \acal.\]
Is has clearly an universal property, and note that composing $\HH$ with any faithful functor will give the same universal category $\mcal(\mathrm{T},\HH)$. And that any functor comparing with $\HH$ after composing with an
exact functor $F:\acal\to\bcal$ will also factor uniquely through $\mcal(\mathrm{T},\HH)$ (because one kernel will be included in the other). 
For example, this is the case for $\acal=\Lambda\text{-}\mathrm{mod}\xrightarrow{\otimes_\Lambda \Lambda'} \Lambda'\text{-}\mathrm{mod}=\bcal$ with $\Lambda'$ a flat $\Lambda$-algebra.
\end{constr}

We apply this to the perverse realisation functor.
\begin{constr}
  \label{constru:MunivPerv}
  Let $X$ be a finite type $k$-scheme. Consider the cohomological functor 
  \[\HHp^0\circ\rho_\mathrm{B}\colon\DMgm(X,\Z)\to \mathrm{Perv}(X^\mathrm{an},\Z)\] where $\mathrm{Perv}(X^\mathrm{an},\Z)$ is the category of analytic perverse sheaves and $\HHp^0$ is the $0$-th perverse cohomology functor.
  As in \cref{univCatdef}, we can consider the universal abelian factorisation $\mcal_\mathrm{univ}(X,\Z)$ of this functor through a faithful exact functor $r_\mathrm{B}\colon\mcal_\mathrm{univ}(X,\Z)\to\mathrm{Perv}(X^\mathrm{an},\Z)$. We thus have a universal homological functor 
  \[\HH_\mathrm{univ}\colon \DMgm(X,\Z)\to\mcal_\mathrm{univ}(X,\Z).\]
  Because the Betti realisation of perverse Nori motives is faithful exact (as the realisation from $\DNgm(X,\Z)$ is conservative), we actually have a finer factorisation:
  \[\begin{tikzcd}
    {\DMgm(X,\Z)} & {\mcal_\mathrm{univ}(X,\Z)} & {\mcal_\mathrm{perv}(X,\Z)}
    \arrow["{\mathrm{comp}}", from=1-2, to=1-3]
    \arrow["{\HH_\mathrm{univ}}", from=1-1, to=1-2]
    \arrow["{\HHp^0\circ\rho_\mathrm{N}}"', from=1-1, to=1-3,bend right=20]
  \end{tikzcd}\]
with $\mathrm{comp}$ a faithful exact functor.
\end{constr}

\begin{lem}
  \label{compQ}
  Let $X$ be a finite type $k$-scheme. The functor 
  \[\mathrm{comp}\otimes_\Z \Q\colon \mcal_\mathrm{univ}(X,\Z)\otimes_\Z \Q\to\mcal_\mathrm{perv}(X,\Z)\otimes_\Z \Q\] is an equivalence.
\end{lem}
\begin{proof}
  The definition of $\mcal_\mathrm{perv}(X,\Q)$ is $\mcal_\mathrm{univ}(X,\Q)$. Thus we have to show that $\mcal_\mathrm{univ}(X,\Z)\otimes_\Z \Q$ is the universal category for $\HHp^0\circ\rho_\mathrm{B}\otimes_\Z \Q$. This follows easily from the construction of the universal category together with the fact that the rationalisation of the abelian hull of an additive category if the abelian hull of the rationalisation.
\end{proof}

For each nonzero integer $n$, we have an Artin motive functor 
\[\iota\colon\dconset(X_\et,\Z/n)\to \DMgm(X,\Z)\] such that the composition with the realisation is perverse t-exact and conservative. Thus it induces a faithful exact functor 
\[i\colon \mathrm{Perv}(X_\et,\Z/n)\to \mcal_\mathrm{univ}(X,\Z)[n].\]

The following result is an adaptation of a proof due to Nori (\cite{fakhruddinNotesNoriLectures2000}) in the case of fields.
\begin{prop}
  \label{compTors}
  Let $X$ be a finite type $k$-scheme. Let $n$ be an integer. The functor 
  \[i\colon\mathrm{Perv}(X_\et,\Z/n)\to\mcal_{\mathrm{univ}}(X,\Z)[n]\] is an equivalence of quasi-inverse the functor 
  \[\mcal_{\mathrm{univ}}(X,\Z)[n] \xrightarrow{\mathrm{comp}} \mcal_{\mathrm{perv}}(X,\Z)[n] \simeq \mathrm{Perv}(X_\et,\Z/n).\]
\end{prop}
\begin{proof}
  We begin by noting three things: 
  
\begin{enumerate}
  \item The composite functor $\mathrm{comp}\circ i$ lands in $\mcal_\mathrm{perv}(X,\Z)[n] =\mathrm{Perv}(X_\et,\Z/n)$ and induces the identity of the latter, because we have that $\mathrm{comp}\circ i = \HHp^0\circ\rho_\mathrm{N}\circ \iota$. \item  The functors $\mathrm{comp}$ and $i$ are faithful; thus it suffices to prove that $i$ is essentially surjective. Indeed, if $i$ is essentially surjective, this implies that $\mathrm{comp}$ is fully faithful, thus $i$ is also fully faithful.
  \item The image of $i$ is stable under subquotients. Indeed, if $M$ belongs to $\mcal_\mathrm{perv}(X,\Z)[n]$ and is a subobject of something of the form $i(N)$, then $\mathrm{comp}(M)$ is a subobject of $\mathrm{comp}\circ i (N)=N$. This implies that $i\circ\mathrm{comp}(M)$ and $M$ are both subobjects of $i(N)$ that are isomorphic after applying the faithful functor $\mathrm{comp}$ thus they are isomorphic. For the quotients, one uses that if there is a surjection $i(N)\to M$ then the kernel is in the image of $i$ thus the cokernel also.
\end{enumerate}

  Note also that if $N$ belongs to $\DMgm(X,\Z)$, we can write $N/n = \iota(K)$ for some $K$ in $\dconset(X_\et,\Z/n)$, thus \[\HH_\mathrm{univ}(N/n) =\HH_\mathrm{univ}(\iota(K)) \overset{(*)}{=} \HH_\mathrm{univ}(\iota(\HHp^0(K)))=i(\HHp^0(K))\]
 is in the image of $i$, where $(*)$ follows from the perverse t-exactness of the functor 
  \[\dconset(X_\et,\Z/n)\to\DNgm(X,\Z).\] 

  It suffices to show that any object of the form $M/nM$ is in the image of $i$, because any $n$-torsion object is of this form.  Thus, let $M$ be in $\mcal_\mathrm{univ}(X,\Z)$ and let $\HH_\mathrm{univ}(N)\to M$ be a surjective map (this exists by definition of the universal category see \cref{univCatdef}). As there is an exact triangle 
  \[N\xrightarrow{\times n} N\to N/n\] in $\DMgm(X,\Z)$, we have a short exact sequence 
  \[\HH_\mathrm{univ}(N)\xrightarrow{\times n}\HH_\mathrm{univ}(N)\to \HH_\mathrm{univ}(N/n).\] Denote by t the image of the map $\HH_\mathrm{univ}(N)\to\HH_\mathrm{univ}(N/n)$. Because $\HH_\mathrm{univ}(N/n)$ is in the image of $i$ we see that t is also in the image of $i$ as the latter is closed under subquotients. Moreover we have the following commutative diagram with exact rows:
  \[\begin{tikzcd}
    {\HH_\mathrm{univ}(N)} & {\HH_\mathrm{univ}(N)} & T & 0 \\
    M & M & {M/nM} & 0
    \arrow["{\times n}", from=1-1, to=1-2]
    \arrow[from=1-1, to=2-1]
    \arrow[from=1-2, to=1-3]
    \arrow[from=1-2, to=2-2]
    \arrow[from=1-3, to=1-4]
    \arrow[dashed, from=1-3, to=2-3]
    \arrow["{\times n}"', from=2-1, to=2-2]
    \arrow[from=2-2, to=2-3]
    \arrow[from=2-3, to=2-4]
  \end{tikzcd}\] where the map $T\to M/nM$ exists because the composite of the map $M\to M/nM$ with the map $\HH_\mathrm{univ}(N)\to M$ vanishes on the kernel of the quotient map $\HH_\mathrm{univ}(N)\to T$. As $\HH_\mathrm{univ}(N)\to M$ is surjective, this is also the case of the map $T\to M/nM$ thus $M/nM$ is in the image, finishing the proof.
\end{proof}

Thus the functor 
\[\mathrm{comp}\colon \mcal_{\mathrm{univ}}(X,\Z)\to \mcal_{\mathrm{perv}}(X,\Z)\] is an equivalence after being tensored by $\Q$ and induces an equivalence on $n$-torsion objects for all $n$. To conclude that it is an equivalence, one could compute some $\mathrm{Ext}^1$ in $\mcal_{\mathrm{perv}}(X,\Z)$, but this seems to be beyond the scope of our current technology. However, the following conjecture seems more reachable:

\begin{conject}
  \label{conjTors}
Let $X$ be an algebraic variety over $\Q$. Then the abelian category $\mcal_\mathrm{univ}(X,\Z)$ has enough torsion-free object: for every $M$ in $\mcal_\mathrm{univ}(X,\Z)$ there exists an epimorphism $N\twoheadrightarrow M$ in $\mcal_\mathrm{univ}(X,\Z)$ with $N$ a torsion-free object.
\end{conject}
This conjecture would follow from the equivalence \[\mcal_\mathrm{pairs}(X,\Z)\xrightarrow{\sim} \mcal_\mathrm{univ}(X,\Z),\] where 
$\mcal_\mathrm{pairs}(X,\Z)$ is the universal category associated to a diagram of very good pairs as in Nori's original construction, as considered by F. Ivorra in \cite{MR3723805} (he was considering rational motives but his construction remains valid for integral sheaves). This equivalence would follow from the fact that Ivorra's realisation functor, the main construction in \cite{MR3518311} is the $\HHp^0$ of our functor $\rho_\mathrm{N}$. Note that in \cite{MR3518311} the author was only considering smooth varieties but it is not hard to see that a slight modification of his functor (namely, replace the constant sheaf $\Q[-\dim X]$ on $X$ by $\pi_X^!\Q_{\Spec(\Q)}$) would give a functor for any variety. Such a conjecture is related to \cite[Conjecture 2.12]{ivorraFourOperationsPerverse2022}, and might be doable with the current technology. 

We now explain how \cref{conjTors} implies that both $\mcal_\mathrm{univ}(X,\Z)\simeq \mcal_\mathrm{perv}(X,\Z)$ for all $X$ but also that $\D^b(\mcal_\mathrm{perv}(X,\Z))\simeq\DNgm(X,\Z)$.

\begin{prop}
  \label{DbPerv}
 Assume that \cref{conjTors} holds. Then for every field $k$ of characteristic zero and $X$ a finite type $k$-scheme, the functor $\mathrm{comp}\colon\mcal_\mathrm{univ}(X,\Z)\to\mcal_\mathrm{perv}(X,\Z)$ is an equivalence, and the functor  
  \[\D^b(\mcal_\mathrm{univ}(X,\Z))\to \DNgm(X,\Z)\] is an equivalence.
\end{prop}
\begin{proof} 
  By continuity along change of fields, it suffices to prove the theorem when $X$ is of finite type over $\Q$.
  It suffices to prove that the functor 
  \[\D^b(\mcal_\mathrm{univ}(X,\Z))\to\DNgm(X,\Z)\] induced by universal property from the functor 
  \[\mcal_\mathrm{univ}(X,\Z)\to \mcal_\mathrm{perv}(X,\Z)\to \DNgm(X,\Z)\] is fully faithful. Indeed, once it is fully faithful, it will be an equivalence by d\'evissage because the essential image is then stable under extension. Indeed if $M\in\mcal_\mathrm{perv}(X,\Z)$, then there is an short exact sequence \[0\to M_\mathrm{tors}\to M\to M_\mathrm{fr}\to 0\] in $\mcal_\mathrm{perv}(X,\Z)$ with $M_\mathrm{tors}$ of torsion and $M_\mathrm{fr}$ torsion-free. By \cref{compTors} and \cref{torsMp} the object $M_\mathrm{tors}$ is in the image. Moreover by \cref{compQ} there exists $P\in\mcal_\mathrm{univ}(X,\Z)$ without torsion and a map $P\to M_\mathrm{fr}$ which becomes an equivalence after tensorisation by $\Q$. In particular, this map has to be injective, and the cokernel is torsion, thus in the image. By fully faithfulness, the object $M_\mathrm{fr}$ is in the image, whence $M$ is in the image.

  The functor is an equivalence after tensorisation by $\Q$, and if we tensor the mapping spectra by $\otimes_\Z\Z/p\Z$, using \cref{torsderive} (we have enough torsion-free objects by hypothesis!) together with \cite{MR0923133} and rigidity, we conclude. Note that this also proves that $\mathrm{comp}$ is an equivalence.
\end{proof}

\subsection{Applications to Artin motives}
\label{0mot}
Recall that given a scheme $X$ and a ring of coefficients $\Z$, the category of Artin motives (resp. constructible Artin motives) is the localising (resp. thick) subcategory of $\DM(X,\Z)$ generated by the $f_*\un_Y$ for $f\colon Y\to X$ finite (see for instance \cite[Definition 2.1.1.1]{ruimyAbelianCategoriesArtin2023}); we denote it by $\mathrm{DM}^{\et,0}(X,\Z)$ (resp. $\mathrm{DM}^{\et,0}_c(X,\Z)$). We likewise define $\DN^0(X,\Z)$ and $\DN^0_c(X,\Z)$. 
\begin{prop}\label{0motGros}
  Let $X$ be a qcqs $\Q$-scheme and let $\Lambda$ be a commutative ring. Then, the Nori realisation functor $\rho_N$ induces an equivalence $\mathrm{DM}^{\et,0}(X,\Lambda)\to \DN^0(X,\Lambda)$. Moreover, the functor $\mathrm{DM}^{\et,0}_c(X,\Z)\to \DNgm(X,\Z)$ is t-exact when $\mathrm{DM}^{\et,0}_c(X,\Z)$ is endowed with the ordinary homotopy t-structure (see \cite[Definition 3.1.1.1 and Theorem 3.1.2.7]{ruimyAbelianCategoriesArtin2023}).
\end{prop}
\begin{proof}
  The functors $\Mod_\Lambda(\DN^0(X,\Z))\to \DN^0(X,\Lambda)$ and $\Mod_\Lambda(\mathrm{DM}^{\et,0}(X,\Z))\to \mathrm{DM}^{\et,0}(X,\Lambda)$ are equivalences. Hence, we can assume that $\Lambda=\Z$. Using \cref{LemmeMagique}, it suffices to tackle the case where $\Lambda=\Z/n\Z$ and the case where $\Lambda=\Q$. The first case follows from rigidity while the second case follows from the main Theorem of \cite{Swann0mot}. The t-exactness property follows from the t-exactness of the $\ell$-adic realisations functors on Artin motives (\cite[Theorem 3.1.2.7]{ruimyAbelianCategoriesArtin2023}) and their conservativity and t-exactness on $\DNgm(X,\Z)$.
\end{proof}

We will now use Nori motives to give a characterisation of dualisable objects of $\mathrm{DM}^{\et,0}_c(X,\Z)$. To that end, recall that the category $\mathrm{DM}^{\et,\mathrm{sm}0}(X,\Z)$ (resp. $\mathrm{DM}^{\et,\mathrm{sm}0}_c(X,\Z)$) of smooth Artin motives (resp. constructible smooth Artin motives) is the localising (resp. thick) subcategory of $\DM(X,\Z)$ generated by the $f_*\un_Y$ for $f\colon Y\to X$ finite \emph{\'etale}. Using \cref{0motGros}, they are equivalent to full subcategories of $\DN(X,\Z)$ defined similarly.

\begin{lem}\label{Lemmamachin}
  Let $\Lambda$ be a commutative ring. For $X$ a qcqs $\Q$-scheme denote by  $\mathrm{DM}^{\et,0}_\mathrm{rig}(X,\Lambda)$ the full subcategory of $\mathrm{DM}^{\et,0}(X,\Lambda)$ made of those objects which are dualisable in $\DM(X,\Lambda)$. The functor $\mathrm{DM}^{\et,0}_\mathrm{rig}(-,\Lambda)$ is a finitary \'etale sheaf.
\end{lem}
\begin{proof}
  As $\DMc(-,\Lambda)$ is finitary by \cite[Corollary 3.10]{DMpdf}, the functor $\mathrm{DM}^{\et,0}_\mathrm{rig}(-,\Lambda)$ is also finitary: being Artin and rigid are both finitary conditions. We have to show that $\mathrm{DM}^{\et,0}_\mathrm{rig}(-,\Lambda)$ is a subsheaf of $\DMc(-,\Lambda)$, the latter being a sheaf by \cite[Lemma 2.5]{DMpdf}, meaning that if $f\colon Y\to X$ is an \'etale covering and $M$ is an object of $\DMc(X,\Lambda)$ such that $f^*M$ is dualisable and Artin, then so is $M$. But dualisability is \'etale local and so is being a $0$-motive.
\end{proof}

\begin{prop}
  Let $X$ be a normal $\Q$-scheme. The inclusion \[\mathrm{DM}^{\et,\mathrm{sm}0}_c(X,\Z)\subseteq \mathrm{DM}^{\et,0}_\mathrm{rig}(X,\Z)\] is an equivalence. The t-structure on $\DNgm(X,\Z)$ restricts to the full subcategory $\mathrm{DM}^{\et,\mathrm{sm}0}_c(X,\Z)$ and the heart of the restricted t-structure is equivalent to $\mathrm{Rep}^A(\pi_1^\et(X),\Z)$.
\end{prop}
\begin{proof}
  Using \'etale descent, continuity and \cref{Lemmamachin}, it suffices to prove the statements if $X$ is of finite type over an algebraically closed field. 
  
   First note that the inclusion $\mathrm{DM}^{\et,0}_\mathrm{rig}(X,\Z) \to \DN^0_\mathrm{rig}(X,\Z)$ is an equivalence, where the right-hand side consists of the objects of $\DN^0(X,\Z)$ which are dualisable in $\DN(X,\Z)$. Indeed, if $M$ belongs to $\DN^0_\mathrm{rig}(X,\Z)$ then its dual is also a $0$-motive, because being a constructible $0$-motive can be tested on points and all constructible $0$-motives over a point are dualisable, with $0$-motivic duals. But now, the t-structure of $\DNgm(X,\Z)$ restricts to $\mathrm{DM}^{\et,0}_\mathrm{rig}(X,\Z)$ because it can be restricted to both dualisable objects and $0$-motives.
   
   By \cref{PBBetti} we have a pullback square  
   \[ \begin{tikzcd}
    \mathrm{DM}^{\et,0}_\mathrm{rig}(X,\Z)^\heart \arrow[r] \arrow[d]
          \arrow[dr, phantom, very near start, "{ \lrcorner }"]
        & \mathrm{Loc}(X^\an,\Z) \arrow[d] \\
        \mathrm{DM}^{\et,0}_\mathrm{rig}(X, \Q)^\heart \arrow[r,swap]
        & \mathrm{Loc}(X^\an,\Q)
  \end{tikzcd}\] where $\mathrm{Loc}(X^\an,\Z)$ is the abelian category of local systems on $X^\an$ with $\Z$ coefficients. By \cite{Swann0mot} the category $\mathrm{DM}^{\et,0}_\mathrm{rig}(X,\Q)^\heart$ is equivalent (through $\rho_!$) to  $\mathrm{Rep}^A(\pi_1^\et(X),\Q)$. As the category $\mathrm{Rep}^A(\pi_1^\et(X),\Z)$ also fits in the corner of this pullback diagram, this proves that the functor 
  $\rho_!$ induces an equivalence $\mathrm{Rep}^A(\pi_1^\et(X),\Z)\to \mathrm{DM}^{\et,0}_\mathrm{rig}(X,\Z)^\heart$. Finally, by d\'evissage, because any object of $\mathrm{Rep}^A(\pi_1^\et(X),\Z)$ defines an object of $\mathrm{DM}^{\et,\mathrm{sm}0}_c(X,\Z)$, any rigid $0$-motive is smooth Artin.
\end{proof}

Using the framework of Nori motives, we can prove new cases of \cite[Theorem 3.2.3.7]{ruimyAbelianCategoriesArtin2023} which is one of the main results of \textit{loc. cit.} and is an analogue of the Artin Vanishing Theorem in the setting of Artin motives. We begin by the following lemma:

\begin{lem}\label{cohcoh}
  Let $X$ be a qcqs $\Q$-scheme. The ordinary and perverse t-structures on $\DNgm(X,\Z)$ restrict to the category $\DN^\mathrm{coh}_c(X,\Z)$ defined as the thick subcategory of $\DN(X,\Z)$ generated by the $f_*\Z_Y$ where $f\colon Y\to X$ is proper.
\end{lem}
\begin{proof}
  By \cite{Swann0mot}, the result holds for $\Q$ coefficients. Moreover it suffices to prove the result over the spectrum of a field $k$. Because the map \[\DN^\mathrm{coh}_c(k,\Z)\otimes_{\Perf_\Z}\Perf_{\Q}\to \DN^\mathrm{coh}_c(k,\Q)\] is an equivalence and if $M$ is in $\DN^\mathrm{coh}_c(k,\Z)$ then $\HH^i(M\otimes_\Z \Q)\simeq \HH^i(M)\otimes_\Z \Q$ is cohomological, there is a cohomological motive $P$ in $\DN^\mathrm{coh}_c(k,\Z)$ and a map 
  $P\to \HH^i(M)$ whose cone is torsion and dualisable, thus a constructible $0$-motive, which are cohomological. The claim about the perverse t-structure is true by gluing. 
\end{proof}

Recall that a smooth Artin motive $M$ is $\Q$-constructible if $M\otimes_\Z \Q$ is constructible as a smooth Artin motive. We denote by $\mathrm{DM}^{\et,sm0}_{\Q-c}(X,\Z)$ the category of $\Q$-constructible smooth Artin motives. We can now prove new cases of \cite[Theorem 3.2.3.7]{ruimyAbelianCategoriesArtin2023}. 
\begin{prop}\label{ArtinVanishing} Let $S$ be an excellent scheme $\Q$-scheme, let $f\colon X\to S$ be an affine quasi-finite morphism. Assume that $X$ is regular. Then, the functor $$f_!\colon \mathrm{DM}^{\et,sm0}_{\Q-c}(X,\Z) \to \mathrm{DM}^{\et,0}(S,\Z)$$ is t-exact with respect to the perverse homotopy t-structure (see \cite[Definition 3.2.1.1]{ruimyAbelianCategoriesArtin2023}).
\end{prop}
\begin{proof}
  The same proof as in \cite[Theorem 3.2.3.7]{ruimyAbelianCategoriesArtin2023} allows to reduce to the case of rational coefficients. We now have to prove that  $f_!\colon \DN^{sm0}_{c}(X,\Q) \to \DN^0(S,\Q)$ is t-exact with respect to the perverse homotopy t-structure (defined in the same fashion as in \cite[Definition 3.2.1.1]{ruimyAbelianCategoriesArtin2023}). 

  The rest of the proof is an adaptation of the arguments of \cite[Section 4.8]{ruimyArtinPerverseSheaves2023}.
  We can write the functor  $$f_!\colon \mathrm{DM}^{\et,sm0}_{c}(X,\Q) \to \mathrm{DM}^{\et,0}(S,\Q)$$ as the composition 
  $$\mathrm{DM}^{\et,sm0}_{c}(X,\Q)\simeq \mathrm{DN}^{sm0}_{c}(X,\Q) \xrightarrow{\iota_0}\mathrm{DN}_c^\mathrm{coh}(X,\Q)\xrightarrow{f_!}\mathrm{DN}^\mathrm{coh}_c(S,\Q)\xrightarrow{\omega^0}\mathrm{DN}^0_c(S,\Q)\simeq  \mathrm{DM}^{\et,0}_c(S,\Q),$$
  where the Artin truncation functor $$\omega^0\colon \DN^\mathrm{coh}(X,\Q)\to \DN^0(X,\Q)$$ defined as the right adjoint to the inclusion $\iota_0$ of Artin motives into cohomological motives, lands in $\DN^0_c(X,\Q)$ by \cite{Swann0mot}.
  Thus it suffices to check that all functors are t-exact.
   The functor $\omega^0$ is t-exact with respect to the perverse t-structure on the left-hand side (which exists by the above \cref{cohcoh}) and the perverse homotopy t-structure on the right-hand side: this is because an object of $\DN^0(X,\Q)$ is perverse homotopy t-negative if and only if it is perverse t-negative (this can be shown as in \cite[Proposition 4.5.1]{ruimyArtinPerverseSheaves2023}) and because the functor $\omega^0$ is perverse right t-exact as per \cite[Lemma 3.2.6]{MR3293216}. But as $\iota_0$ is t-exact when restricted to smooth Artin objects because it is t-exact for the ordinary t-structure, the result reduces to the usual Artin Vanishing Theorem.
  %The idea is then to adapt the arguments of \cite[Section 4.8]{ruimyArtinPerverseSheaves2023} {\color{red} ah tu as mis les refs à la fin jpp}: as in [VAISH-NAIR], we can, knowing that Nori motives are endowed with a good theory of weights [REF] introduce a t-structure $(w_{\leqslant \id},w_{>\id})$ on $\DNgm(X,\Lambda)$. Then, $\DN^0_c(X,\Lambda)$ is a subcategory of the category $\DN^{w\leqslant 0}(X,\Lambda)$ of t-non-positive objects of the latter t-structure. 
\end{proof}
\subsection{Motivic intersection complex.}
\label{j!*}
Consider an excellent scheme $B$ of dimension $2$ or less as well as a regular separated $B$-scheme of finite type $S$, endowed with a geometric and $\otimes$-invertible object $\omega_S$ in $\DN(S, \Z)$ (note that $\omega_S = \Z_S$ works). We fix $X$ a finite type $S$-scheme. 

From now on, we assume that $X$ is endowed with a bounded below dimension function $\delta$; by shifting, we can assume $0$ to be its lower bound. 
Any non-increasing function $\overline{q}\colon \N\to\Z$, defines a perversity function $q=\overline{q}\circ \delta$ on $X$ along any stratification of $X$ by connected regular and locally closed strata (in the sense of \cite[Section 2.1.1]{MR0751966}). 
%If $X$ is endowed with a perversity function $q$ 
%For example if $X$ is endowed with a dimension function $\delta$ as in \cref{ExistPerv}, 
%For example $q=-\delta$ defines a perversity function, 
The construction given in \cref{ExistPerv} for $\overline{q}=-\id$ works whenever the condition 
\[(*)\colon \forall n\in \N,\ q(n)-q(n+1)\in \{0,1,2\}\]
is satisfied. It then yields a $q$-perverse t-structure ${}^q\mathrm{t}$ on $\DNgm(X,\Z)$, whose heart we will denote by $\mcal_{q\mathrm{-perv}}(X,\Z)$ and whose cohomology functors are will be denoted by ${}^q\HH^n(-)$.
\begin{defi} Let $\overline{q}\colon \N\to\Z$ be a non-increasing function satisfying condition $(*)$, let $q=\overline{q}\circ \delta$ and let $j\colon U\to X$ be an open immersion.
  The \emph{$q$-intermediate extension functor} is the functor 
  \[{}^q j_{!*}\colon \mcal_{q\mathrm{-perv}}(U,\Z)\to \mcal_{q\mathrm{-perv}}(X,\Z)\] defined by 
  \[{}^q j_{!*}M = \mathrm{Im}(^q\HH^0(j_!M)\to {}^q\HH^0(j_*M)).\]
The \emph{$q$-motivic intersection complex} is the $q$-perverse Nori motive $${}^q\mathrm{IC}_{X,\Z}:={}^qj_{!*}(\Z[-q(U)])[q(U)].$$ When $X$ is endowed with a dimension function $\delta$ and $q=-\delta$ is the self-dual perversity, we simply refer to those as the intermediate extension functor $j_{!*}$ and the motivic intersection complex $\mathrm{IC}_{X,\Z}$. 
\end{defi}

\begin{rem}
  In \cite{MR2953419,MR4050012} J\"org Wildeshaus constructed a motivic intersection complex in $\mathrm{DM}(X,\Q)$ using weights, in several cases (mainy when $X$ is a Shimura variety). In those cases, we expect the complex $\mathrm{IC}_{X,\Q}$ to be the Nori realisation of Wildeshaus' motivic intersection complex up to a shift of $-\dim(X)$. Indeed assuming the existence of the motivic t-structure, this would be true by uniqueness of Wildeshaus' complex \cite[Theorem 3.1]{MR2953419}.
\end{rem}

We fix a stratification $(X_i)$ of $X$ by regular and connected locally closed subschemes. Given an integer $n$, we set $U_n$ to be the union of those strata $X_i$ such that $\delta(X_i)\geqslant n$ and $j_n\colon U_{n}\to U_{n-1}$ to be the open immersion.
%Assume that $q$ verifies the condition $(*)$: if $S\subset \overline{T}$ is an inclusion of a locally closed subset $S$ of $X$ in the adherence of another locally closed subset $T$, then $q(S)\geqslant q(T)$. 
\begin{prop}
  \label{formuleSale}
  Let $\overline{q}\colon \N\to\Z$ be a non-increasing function satisfying condition $(*)$ and let $q=\overline{q}\circ \delta$. %If $n\geqslant 0$, let $U_n$ be the union of strata such that $q(U_n)\leqslant n$ and let $j_n\colon U_{n-1}\to U_n$ be the open immersion. 
  %Assume that $q\leqslant a$ for some $a\in\Z$. 
  Let $k\geqslant 0$ be an integer and let $j\colon U_k\to X=U_0$ be the open immersion. Then there is a canonical isomorphism of functors 
  \[{}^qj_{!*}\simeq \left(\tau^{\leqslant \overline{q}(0)-1}(j_{1})
_*\right)\circ\cdots\circ \left(\tau^{\leqslant \overline{q}(k-2)-1}(j_{k-1})_*\right)\circ\left(\tau^{\leqslant \overline{q}(k-1)-1}(j_{k})_*\right)\] 
% \[{}^qj_{!*}\simeq \left(\tau^{\geqslant \overline{q}(0)+1}(j_{1})
% _!\right)\circ\cdots\circ \left(\tau^{\geqslant \overline{q}(k-2)+1}(j_{k-1})_!\right)\circ\left(\tau^{\geqslant \overline{q}(k-1)+1}(j_{k})_!\right)\]
on $ \mcal_{q\mathrm{-perv}}(U_k)$, with $\tau^{\leqslant i}$ the truncations functors for the ordinary t-structure.
\end{prop}
\begin{proof}
  This is a consequence of \cite[Proposition 1.4.23]{MR0751966} (see the proof of \cite[Proposition 2.1.11]{MR0751966}).
\end{proof}

\begin{cor}
  Let $\overline{q},\overline{r}\colon \N\to\Z$ be non-increasing functions and $q=\overline{q}\circ \delta$ and $r=\overline{r}\circ \delta$.
  Let $k\geqslant 0$ be an integer and let $j\colon U_k\to X=U_0$ be the open immersion. 
  There exists a canonical pairing in $\DN_c(X,\Z)$
 \[ {}^qj_{!*}(M[-q(U_k)])\otimes{}^rj_{!*}(N[-r[U_k]])\to{}^{q+r}j_{!*}(M\otimes N[-q(U_k)-r(U_k)])\]
 for $M,N\in\mcal_\mathrm{ord}^\mathrm{rig}(U_k)$ motivic local systems (that is, objects in the ordinary heart which are dualisable). %{\color{red} ???????? même pas définis sur le même ouvert ? ou alors c'est sur le premier truc de la filration en fait, le lieu liss,e mais alors il faut des hyp sur q et r.}
\end{cor}
\begin{proof}
This is a direct consequence of the formula for the intermediate extension given in \cref{formuleSale}, knowing that the functors of the form $\gamma_*$ are lax-monoidal as right adjoints of monoidal functors and that for any $M$ and $N$, we have a map $$\tau^{\leqslant n}(M)\otimes \tau^{\leqslant m}(N)\to \tau^{\leqslant n+m}(M\otimes N)$$ given by the right t-exactness of the tensor product and adjunction.
\end{proof}
This yields a pairing 
\[ \mathrm{IC}_{X,\Z}^q\otimes\mathrm{IC}_{X,\Z}^r\to \mathrm{IC}_{X,\Z}^{q+r}\]
in $\mathrm{DN}(X,\Z)$.
% \begin{defi}
%   The motivic intersection complex $\mathrm{IC}_{X,\lambda}^q$ is the object ${}^qj_{!*}(\Lambda_U[d])[-d]$ where $j\colon U\to X$ is a regular dense open subset of dimension $d$. 
% \end{defi}
Now consider the \emph{top perversity} $\overline{t}=-2\mathrm{Id}_\Z$. Gluing on the stratification as in \cite[Section 5.1]{IH2}, whenever $\overline{q}\leqslant \overline{t}$, there exists a map $\mathrm{IC}_{X,\Z}^{q}\to \omega_X[\delta(X)]$ where $\omega_X = \mathbb{D}_X(\Z)$ is the dualising complex. Whence a map 
\begin{equation}
  \label{ICparing}
  \mathrm{IC}_{X,\Z}\otimes\mathrm{IC}_{X,\Z}\to \omega_X[\delta(X)].
\end{equation}
With rational caoefficients, this pairing induces an equivalence 
\begin{equation}\label{pairingICQ}\mathrm{IC}_{X,\Q} \xrightarrow{\sim} \mathbb{D}_X(\mathrm{IC}_{X,\Q})[-\delta(X)]\end{equation} by \cite[Section 5.3]{IH2}.
Let $k$ be a field of characteristic zero.
\begin{lem}
  \label{coeffuniv}
  Let $K\in\DNc(k,\Z)$ and $M\in\mcal(k,\Z)$. For each $i\in\Z$ there exists an short exact sequence 
  \[0\to\HH^1(\underline{\mathrm{Hom}}(\HH^{-i+1}(K),M))\to \HH^i(\underline{\mathrm{Hom}}(K,M))\to \HH^0(\underline{\Hom}(\HH^{-i}(K),M))\to 0\]
  in $\mcal(k,\Z)$.
\end{lem}
\begin{proof}
  This realises to the classical universal coefficient theorem of abelian groups (or $\Z_\ell$-modules if one takes $\ell$-adic realisations).
\end{proof}
\begin{defi}
  Assume that $X$ is a finite type $k$-scheme. The intersection homology complex of $X$ is the object \[\mathrm{IH}(X,\Z):= p_*\mathrm{IC}_{X,\Z}\] in $\DN_c(k,\Z)$ where $p\colon X\to\Spec(k)$ is the structural morphism. Denote by 
  $\mathrm{IH}_i(X,\Z)=\HH^{-i}(\mathrm{IH}(X,\Z))$ the $i$-th intersection homology motive. There is also a version with compact support 
  \[\mathrm{IH}_c(X,\Z):= p_!\mathrm{IC}_{X,\Z}\] and its homology 
  motive $\mathrm{IH}_{i,c}(X,\Z)=\HH^{-i}(\mathrm{IH}_c(X,\Z))$
\end{defi}
With rational coefficients, we obtain from \cref{pairingICQ} a perfect pairing 
\[\mathrm{IH}_{i,c}(X,\Q)\otimes_\Q \mathrm{IH}_{d-i}{(X,\Q)}\to \Q_k(-d),\] where $d$ is the dimension of $X$. 

% \begin{rem}
%   In \cite{MR2953419} and \cite{MR4050012} J\"org Wildeshaus constructed intersection complex objects in $\mathrm{DM}(k,\Q)$ using weights, in several cases (mainy Shimura varieties). He proves unicity of those objects (when they exists). This readily implies that his construction, when defined, realises to the intersection complex we defined though the Nori realisation functor $\rho_\mathrm{N}$.
% \end{rem}

We finish with a motivic lift of the results in \cite[Section 3]{MR699009} providing a motivic Seifert linking pairing. The proof of the next proposition is the same as in \emph{loc. cit.}
\begin{prop}
  For $i\in\Z$ denote by $T_i(X)$ (\emph{resp.} by $T_{i,c}(X)$) the maximal torsion subobject of $\mathrm{IH}_i(X,\Z)$ (\emph{resp.} of $\mathrm{IH}_{i,c}(X,\Z))$. There is a pairing 
  \[T_{i,c}(X)\times T_{d-i-1}(X)\to \Q/\Z(-d)\] in $\mcal(k,\Z)$.
\end{prop}
\begin{proof}
  We apply \cref{coeffuniv} and obtain a short exact sequence
  \[ 0\to \HH^1(\underline{\Hom}(\mathrm{IH}_{d-i-1}(X,\Z),\Z))\to \HH^{-i}(\underline{\Hom}(\mathrm{IH}(X,\Z)[d](d),\Z))\to \HH^0(\underline{\Hom}(\mathrm{IH}_{d-i}(X,\Z),\Z))\to 0.\]
  Using the exact triangle $\Z\to\Q\to\Q/\Z$ and the fact that for any $M\in\mcal(k,\Z)$ the object $\HH^1(\underline{\Hom}(M,\Q))$ vanishes we obtain
  that  \[\HH^0(\underline{\Hom}(\mathrm{IH}_{d-i-1}(X,\Z)(d),\Q/\Z))\xrightarrow{\sim}\HH^1(\underline{\Hom}(\mathrm{IH}_{d-i-1}(X,\Z)(d),\Z))\] is an equivalence. Moreover because $\Q/\Z$ is torsion, the map 
  \[\HH^0(\underline{\Hom}(\mathrm{T}_{d-i-1}(X)(d),\Q/\Z))\to\HH^0(\underline{\Hom}(\mathrm{IH}_{d-i-1}(X,\Z)(d),\Q/\Z))\] is also an equivalence (this is true after realisation). Also any map $M\to \HH^0(\underline{\Hom}(\mathrm{IH}_{d-i}(X,\Z)(d),\Z))$ with $M$ torsion vanishes because it can be checked after realisation. Thus the composed map
  %  (here we use that $X$ is proper to have the intersection homology complex on both sides and no compact support objects appearing)
    $T_{i,c}(X)\to \mathrm{IH}_{i,c}(X,\Z)\to \HH^{-i}(\underline{\Hom}(\mathrm{IH}(X,\Z)[d](d),\Z))$ obtained from applying $\HH^{-i}p_!$ to \cref{ICparing} factors through the inclusion  \[\HH^0(\underline{\Hom}(\mathrm{T}_{d-i-1}(X)(d),\Q/\Z))\to \HH^{-i}(\underline{\Hom}(\mathrm{IH}(X,\Z)[d](d),\Z)),\] yielding
  \[ T_{i,c}(X)\to \HH^0(\underline{\Hom}(\mathrm{T}_{d-i-1}(X)(d),\Q/\Z)).\]
  By adjunction (here we use that $\mcal(k,\Z)$ has enough flat objects) this provides the desired pairing.
\end{proof}
\begin{rem}
  The above pairing is non-degenerate when $X$ is smooth. It is not non-degenerate in general, there is a topological obstruction introduced in \cite{MR699009} when $X$ is singular and proper.
\end{rem}
\section{Integral mixed Hodge modules.} 

\subsection{Definition of mixed Hodge modules with integral coefficients}
Let $X$ be a separated finite type $\C$-scheme. In \cite{MR1047415}, M. Saito constructed an abelian category $\mathrm{MHM}(X)$ of mixed Hodge modules over $X$, which is a relative version of mixed Hodge structures, modelled on perverse sheaves, in the sense that if $X= \Spec(\C)$, we have that $\mathrm{MHM}(X)\simeq \mathrm{MHS}^P_\Q(\C)$ is the category of polarisable mixed Hodge structure over $\C$ with rational coefficients (which consists of mixed Hodge structure such that the graded pieces of the weight filtration are polarisable \cite[D\'efinition 2.1.15]{MR0498551}, it is a full subcategory of Deligne's mixed Hodge structures), and there is a faithful exact functor \[\mathrm{rat}\colon \mathrm{MHM}(X)\to\mathrm{Perv}(X^\mathrm{an},\Q).\] 
Moreover, M. Saito constructed the six operations on the bounded derived category $\D^b(\mathrm{MHM}(-))$, in a way compatible with the functor $\mathrm{rat}$. 
In \cite{SwannRealisation} the second author proved that the construction of these six operations by M. Saito could be canonically lifted to the world of $\infty$-categories, with all possible coherences (this includes an extension to non separated schemes and morphisms). As a consequence, using the universal property of $\DM(-,\Q)$, the second author obtained a realisation 
\[\rho_\mathrm{H}\colon\DM(-,\Q)\to\D_\mathrm{H}(-)\] compatible with the six operations, where $\D_\mathrm{H}(-)$ is the indization of $\D^b(\mathrm{MHM}(-))$. 
% By \cite{MR0862628} the abelian category of polarisable mixed Hodge structure with rational coefficients $\mathrm{MHS}_\Q^P(\C)$ is of cohomological dimension 1. If $X$ is a complex variety of dimension $d$ with structural morphism $\pi_X$, as the cohomological amplitude of $(\pi_X)_*\underline{\Hom}(-,-)$ can be measured after applying the functor $\mathrm{rat}$, it is at most $2d+1$; therefore the abelian category of $\mathrm{MHM}(X)$ is of cohomological dimension less that $2d + 2$. Thus we have in fact that $\D_\mathrm{H}(X)$ is the unbounded derived category of ind-mixed Hodge modules $\D(\In\mathrm{MHM}(X))$ (because, for example by \cite[Lemma 1.1.7]{MR3477640} the functor $\map_{\D}(M,-)$ commutes with colimits when $M$ an object of $\mathrm{MHM}(X)$).

In this section, we show that the methods we used for Nori motives apply for mixed Hodge modules, and that proofs are easier in this case.

\begin{defi}
  \label{defDH}
  The presentable $\infty$-category of mixed Hodge modules with integral coefficients is the pullback 
  \[ \begin{tikzcd}
      \D_\mathrm{H}(-,\Z) \arrow[r] \arrow[d]
          \arrow[dr, phantom, very near start, "{ \lrcorner }"]
        & \D_\mathrm{H}(-) \arrow[d,"\mathrm{rat}"] \\
      \In\dconsan((-)^\mathrm{an},\Z) \arrow[r,"-\otimes \Q",swap]
        & \In \dconsan((-)^\mathrm{an},\Q)
  \end{tikzcd}\] where the pullback is taken in the $\infty$-category of functors $(\Sch_\C^\mathrm{ft})^\op\to\mathrm{Pr}^{\mathrm{L},\omega}_\St$.
\end{defi}

Because all categories have a perverse t-structure and that the functors are t-exact, the $\infty$-category $\D_\mathrm{H}(X,\Z)$ affords a perverse t-structure, that restricts to compact objects, over any finite type $\C$-scheme $X$. Moreover, $\D_\mathrm{H}(-,\Z)$ affords a $6$-functor formalism.

Recall that in \cite[D\'efinition 2.3.1]{MR0498551} Deligne defines mixed Hodge structures that have an integral lattice. In fact, we have a pullback for the abelian categories of polarisable mixed Hodge structures
\[ \begin{tikzcd}
    \mathrm{MHS}^P_\Z(\C) \arrow[r] \arrow[d,"\mathrm{int}"]
        \arrow[dr, phantom, very near start, "{ \lrcorner }"]
      & \mathrm{MHS}^P_\Q(\C) \arrow[d] \\
    \mathrm{Ab}^\mathrm{ft} \arrow[r]
      & \mathrm{Vect}_\Q^\mathrm{fd}
\end{tikzcd}\] with $\mathrm{Vect}_\Q^\mathrm{fd}$ the abelian category of finite-dimensional $\Q$-vector spaces 
(note that the weight filtration only exists with rational coefficients by definition). By definition, we then have that $\mathrm{MHS}^P_\Z(\C)$ is heart of the full subcategory of compact objects in $\D_\mathrm{H}(\Spec\C,\Z)$.
By the universal property of the fiber product, we obtain a functor 
 \[\D^b(\mathrm{MHS}_\Z^P(\C))\to\D_\mathrm{H,c}^b(\Spec(\C))\] with target bounded objects of $\D_{\mathrm{H}}(\Spec(\C))$ with cohomology in $\mathrm{MHS}_\Z^P(\C)$.
\begin{prop}
   \label{MHMpt}
   The above functor is an equivalence.
 \end{prop}
 \begin{proof} 
   This follows from the fact that the cohomological dimension of $\mathrm{MHS}_\Z^P$ is $1$ by \cite[Corollary 1.10]{MR0862628}. 
\end{proof}

There is a canonical realisation 
\[\rho_\mathrm{H}\colon \DM(-,\Z)\to\D_\mathrm{H}(-,\Z)\] that commutes with the $6$ operations. Indeed, this functor is obtained from the universal property of the pullback, and as $\D_\mathrm{H}(-,\Z)$ is compactly generated all operations commute with rationalisation so that the proof consist in proving that the exchange maps are equivalence when rationalised (this follows from \cite[Theorem 4.2.29]{MR3971240}) and when taken mod $p$ for all primes $p$.

\subsection{Mixed Hodge modules over \texorpdfstring{$\R$}{R}-schemes.}
In this section we extend the definition of mixed Hodge modules (with integral coefficients) to $\R$-schemes, using the same method as for Nori motives.

Let $X$ be a finite type $\R$-scheme. Then $X_\C := X\times_{\Spec \R}\Spec(\C)$ has a canonical $\Z/2\Z$ action given by the complex conjugation, and thus the $\infty$-category $\D_\mathrm{H}(X_\C,\Z)$ is endowed with a $\Z/2\Z$ action. In fact this construction is functorial and we obtain a functor 
\[(\Sch^\mathrm{ft}_\R)^\op\to\mathrm{Fun}(\mathrm{B}\Z/2\Z,\CAlg(\PrL))\] that sends a finite type $\R$-scheme $X$ to $\D_\mathrm{H}(X_\C,\Z)$. 

Now consider the functor \[(-)^{\mathrm{h}\Z/2\Z}\colon\mathrm{Fun}(\mathrm{B}\Z/2\Z,\CAlg(\PrL))\to \CAlg(\PrL)\] that sends a presentably symmetric monoidal $\infty$-category $\ccal$ with a $\Z/2\Z$ action to the homotopy fixed points $\ccal^{\mathrm{h}\Z/2\Z}$. Note that this functor take a diagram $G\colon \mathrm{B}\Z/2\Z\to \CAlg(\PrL)$ to its limit, thus has a left adjoint $F$ that takes an object $\ccal\in\CAlg(\PrL)$ to the constant diagram $\underline{\ccal}$.
\begin{lem}
  If $X$ is a finite type $\C$-scheme, then the natural functor \[p_1^*\colon\D_\mathrm{H}(X,\Z)\to\D_\mathrm{H}(X_\C,\Z)\] induced by the first projection $p_1\colon X\times_{\Spec(\R)}\Spec(\C)\to X$ identifies naturally the left-hand side to the invariants of the right-hand side. More precisely there is a canonical isomorphism between the functor $p_1^*$ and the counit 
  \[\eta\colon FG(\D_\mathrm{H}(X_\C,\Z))\to \D_\mathrm{H}(X_\C,\Z).\]
\end{lem}
\begin{proof}
  Indeed, one remarks that $X_\C \simeq X\times_\R X$, and that the projection map $X_\C\to X$ is an \'etale cover, 
  so that the Bousfield-Kan formula for the limit of the $\Z/2\Z$-invariants is exactly the limit for \'etale descent.
\end{proof}

In particular the following definition does not introduce clash of notations:
\begin{defi}
  The functoriality of the presentable $\infty$-category of mixed Hodge modules on finite type $\R$-schemes is the composition 
  \[\D_\mathrm{H}(-,\Z)\colon (\Sch_\R^\mathrm{ft})^\op\to \mathrm{Fun}(\mathrm{B}\Z/2\Z,\CAlg(\PrL)) \xrightarrow{(-)^{\mathrm{h}\Z/2\Z}} \CAlg(\PrL).\]
\end{defi}
By the previous lemma, the resriction of this functor to finite type $\C$-schemes corresponds to ind-mixed Hodge modules as constructed by M. Saito.
We may denote by $\D_\mathrm{H}(-,\Q):=\D_\mathrm{H}(-,\Z)\otimes\Mod_\Q$. 
\begin{prop}
  \label{MHMRrat}
  Let $X$ be a finite type $\R$-scheme. Then $\D_\mathrm{H}(X,\Q)$ is compactly generated and its full subcategory of compact objects is canonically isomorphic to $\D^b(\mathrm{MHM}(X,\Q))$ where $\mathrm{MHM}(X,\Q)$ consists of the $\Z/2\Z$-invariants of $\mathrm{MHM}(X_\C,\Q)$.
\end{prop}
\begin{proof}
  Because $\D_\mathrm{H}(X_\C,\Q)$ has a t-structure that restricts to compact objects, it is true that $\D_\mathrm{H}(X,\Q)$ has a t-structure, that restricts to the invariants $\dcal(X)$ of the compact objects, seen as a full subcategory of $\D_\mathrm{H}(X,\Q)$. Now if $D$ belongs to $\dcal(X)$, and $M=\colim_i M_i$ is a filtered colimit of objects of $\D_\mathrm{H}(X,\Q)$, we have that 
  \begin{align*}
    \map_{D_\mathrm{H}(X,\Q)}(D,\colim M_i) &\simeq \map_{D_\mathrm{H}(X_\C,\Q)}(D_{\mid X_\C},\colim_i (M_i)_{\mid X_\C})^{\mathrm{h}\Z/2\Z} \\
    & \simeq \mathrm{R}\Gamma(\R_\et,\colim\map_{D_\mathrm{H}(X_\C,\Q)}(D_{\mid X_\C}, (M_i)_{\mid X_\C}) ) \\
    & \overset{(*)}{\simeq} \colim_i\mathrm{R}\Gamma(\R_\et,\map_{D_\mathrm{H}(X_\C,\Q)}(D_{\mid X_\C}, (M_i)_{\mid X_\C}) )\\
   & \simeq\colim_i \map_{\D_\mathrm{H}(X,\Q)}(D,M_i)
  \end{align*} where $(*)$ follows from \cite[Lemma 1.1.10]{MR3477640}. Thus any object of $\dcal(X)$ is compact. Moreover, because the functor 
  \[\D_\mathrm{H}(X,\Q)\to\D_\mathrm{H}(X_\C,\Q)\] is conservative, the $\infty$-category $\dcal(X)$ generates $\D_\mathrm{H}(X,\Q)$ under colimits. Thus, the canonical functor $\In\dcal(X)\to\D_\mathrm{H}(X,\Z)$ is an equivalence.

  Because $p_1^*$ is perverse t-exact, the functor $\mathrm{MHM}(-,\Q)$ has descent along $p_1$. Moreover, because $\D^b(\mathrm{MHM}(-,\Q))\to\dcal(-)$ is conservative and commutes with $p_1^*$ as well as $(p_1)_*$ (which is also perverse t-exact), the same proof as \cite[Theorem 5.23]{SwannRealisation} gives that $\D^b(\mathrm{MHM}(-,\Q))$ has descent for the morphism $p_1$, but this implies that the functor $\D^b(\mathrm{MHM}(X,\Q))\to\dcal(X)$ is an equivalence, finishing the proof.
\end{proof}
\begin{cor}\label{MHMRZ}
  The functor $\D_\mathrm{H}(-,\Z)$ is an \'etale hypersheaf on finite type $\R$-schemes. For any finite type $\R$-scheme $X$, the $\infty$-category $\D_\mathrm{H}(X,\Z)$ affords a t-structure that restricts to the subcategory $\D_\mathrm{H,c}^b(X,\Z)$of \emph{constructible objects}, which consists of those objects whose underlying ind-constructible sheaf is constructible. We denote by $\mathrm{MHM}(X,\Z)$ the heart of the latter $\infty$-category.
\end{cor}
\begin{rem}
  Let $X$ be a smooth $\R$-scheme.
  Let $M\in\mathrm{MHM}(X,\Z)$, then $M_{\mid X_\C}$ admits an underlying $\dcal$-module on $X_\C$ with an action of $\Z/2\Z$, so that in fact we have a functor \[\mathrm{MHM}(X,\Z)\to \dcal\text{-}\mathrm{mod}(X)\] by Galois descent of $\dcal$-modules.
\end{rem}

\begin{prop}
  \label{MHMptR}
  The heart of $\D^b_\mathrm{H,c}(\Spec(\R),\Z)$ is the category $\mathrm{MHS}_\Z^p(\R)$ of polarisable mixed Hodge modules over $\R$ with integral coefficients. Moreover, the canonical functor 
  \[\D^b(\mathrm{MHS}_\Z^p(\R))\to\D^b_\mathrm{H,c}(\Spec(\R),\Z)\]
  is an equivalence.
\end{prop}
\begin{proof}
  We can identify the heart with the abelian category of polarisable mixed Hodge structures over $\C$ with integral coefficients, together with an action of $\Z/2\Z$, and this is exactly $\mathrm{MHS}_\Z^p(\R)$. Thus we have a functor 
  \[\D^b(\mathrm{MHS}_\Z^p(\R))\to\D^b_\mathrm{H,c}(\Spec(\R),\Z).\]
  As in the proof of \cref{MHMRrat} the left-hand side has descent along the morphism $\R\to\C$, so that the result follows from \cref{MHMpt}.
\end{proof}

\begin{cor}
  Let $f\colon X\to \Spec \R$ be a finite type map, then we recover the absolute Hodge cohomology groups of Beilinson (\cite{MR0862628}) as 
  \[\mathrm{R}\Gamma_\hcal(X,\Z(p))\simeq\map_{\D_\mathrm{H}(\Spec(\C),\Z)}(\Z,f_*\Z(p))\simeq \map_{\D_\mathrm{H}(X,\Z)}(\Z,\Z(p)).\]
\end{cor}
\begin{proof}
  Indeed Saito proved the rational result for $X_\C$ in \cite[Section 1.15]{MR1047415}, thus it suffices to check that Beilinson definition is indeed the $\Z/2\Z$-invariants of a pullback of the rational cohomology and the $\Z$-structure, but this is the case by \cite[Section 4.1, Theorem 3.4, Section 7]{MR0862628}.
\end{proof}

We finish with a proposition computing the torsion of $\D_\mathrm{H}(X,\Z)$.
\begin{prop}
  Let $X$ be a finite type $\R$-scheme. Then the functor 
  \[\D(X_\et,\Z)\to\D_\mathrm{H}(X,\Z)\] 
  obtained by taking Galois invariant of the "Artin object functor"
  \[\D((X_\C)_\et,\Z)\to\DM(X_\C,\Z)\to\D_\mathrm{H}(X_\C,\Z)\]
  induces an equivalence 
  \[\D(X_\et,\Z/n\Z)\simeq\D_\mathrm{H}(X,\Z)\otimes_\Z\Mod_{\Z/n\Z}.\]
\end{prop}
\begin{proof}
  Because everything is obtained by taking Galois invariants, it suffices to prove the claim when $X$ is a finite type $\C$-scheme. In that case, the pullback definition provides us with a proof.
\end{proof}

\subsection{Motivic Hodge modules}

As for Nori motives, this realisation functor factors through modules over the algebra $\hcal$ representing Hodge cohomology. Using the same method as \cite[Section 6.2]{SwannRealisation}, we proved in \cite[Section 2]{SwannMHM} that the canonical functor 
\[\Mod_{\hcal\otimes_\Z\Q}(\DM(-,\Q))\to \D_\mathrm{H}(-)\] is fully faithful over finite type $\C$-schemes, and that the t-structure on $\D_\mathrm{H}(-)$ induces a t-structure on the image that we describe following Ayoub's ideas in \cite[Theorem 1.98]{ayoubAnabelianPresentationMotivic2022}: it is the localising subcategory generated by the $f_*\Q(n)$ with $f$ a proper morphism and $n$ an integer, so that we may call those \emph{mixed Hodge modules of geometric origin} (this means that it is the smallest full subcategory stable under all colimits and finite limits that contains the $f_*\Q(n)$). This category of modules was first introduced by Drew in \cite{drewMotivicHodgeModules2018}. 

There is also an enriched version of this result, where instead of obtaining the localising subcategory generated by the $f_*\Q(n)$, we obtain the localising subcategory generated by the $f_*H(n)$ for any mixed Hodge structure $H$ (that we see as a constant mixed Hodge module), that we may call mixed Hodge modules of Hodge origin. Because this theory of enriched modules is more subtle than the geometric origin one, we will explain how to construct an integral version of the enriched category, and leave the case of geometric origin to the careful reader.

\begin{constr}
  The functor $\D_\mathrm{H}(-,\Z)$ naturally takes values in $\D_{H}(\Spec(\R),\Z)$-linear presentable $\infty$-categories. Thus, the functor $\rho_\mathrm{H}$ induces a $\D_{H}(\Spec(\R),\Z)$-linear functor 
  \[\bm{\rho_\mathrm{H}}\colon \DM(-,\Z)\otimes_{\Mod_\Z}\D_{H}(\Spec(\R),\Z) \to \D_\mathrm{H}(-,\Z).\]
  By the good properties of Lurie tensor product, it turns out that the left-hand side is 
  \[\bm{\DM}(-,\Z):=\DM(-,\D_{H}(\Spec(\R),\Z))\] 
  the presentable $\infty$-category of \'etale motives with coefficients in mixed Hodge structures which has the same definition as \'etale motives with coefficient $\Lambda$, except that one replaces $\Mod_\Lambda$ by $\D_{H}(\Spec(\R),\Z)$. The right adjoint $\bm{\rho^\mathrm{H}_*}$ of $\bm{\rho_\mathrm{H}}$ is lax monoidal thus we get an object $\bm{\hcal}_X=\bm{\rho^\mathrm{H}_*}(\Z)$ of $\mathrm{CAlg}(\bm{\DM}(X,\Z))$ for all finite type $\R$-schemes $X$. 
  There is a natural map
  \[\pi_X^*\bm{\hcal}_{\Spec(\C)}\to\bm{\hcal}_X\] with $\pi_X\colon X\to\Spec(\R)$ the structural morphism, it will be shown to be an isomorphism.
  The presentable $\infty$-categories of integral motivic Hodge modules is the category 
  \[\bm{\mathrm{DH}}(X,\Z):=\mathrm{Mod}_{\pi_X^*\bm{\hcal}}(\bm{\DM}(X,\Z)).\]
  As for Nori motives, it is also the base change of $\bm{\DM}(-,\Z)$ from $\bm{\DM}(\Spec(\R),\Z)$-linear categories to $\bm{\mathrm{DH}}(\Spec(\R),\Z)$-linear categories:
  \[\bm{\mathrm{DH}}(-,\Z)\simeq \bm{\DM}(-,\Z)\otimes_{\bm{\DM}(\Spec(\R),\Z)}\bm{\mathrm{DH}}(\Spec(\R),\Z).\] Thus as in \cref{compatibilite sif}, the functor 
  $\bm{\mathrm{DH}}(-,\Z)$ affords all the six operations in such a way that the functor from $\bm{\DM}(-,\Z)$ and the canonical functor to $\D_\mathrm{H}(-,\Z)$ commute with them, the composition of the two being $\bm{\rho_\mathrm{H}}$.
\end{constr}

\begin{lem}
  \label{algfields}
  Let $\iota\colon\Spec(\C)\to\Spec(\R)$ be the spectrum of the inclusion of $\R$ into $\C$. Then the canonical map 
  $\iota^*\bm{\hcal}_{\Spec(\R)}\to\bm{\hcal}_{\Spec(\C)}$ is an equivalence.
\end{lem}
\begin{proof}
  Becuase $\iota$ is finite \'etale, the functor $\iota^*=i^!$, being a right adjoint, commutes with $\bm{\rho_\mathrm{H}}$. Thus $\iota^*\bm{\hcal}_{\Spec(\R)} \simeq \bm{\rho_\mathrm{H}}(\iota^*\un)\simeq\bm{\hcal}_{\Spec(\C)}$.
\end{proof}

\begin{thm}
  \label{theoMHMmod}
  The canonical functor \[\bm{\underline{\rho_\mathrm{H}}}\colon\bm{\mathrm{DH}}(-,\Z)\to\D_\mathrm{H}(-,\Z)\] factoring $\bm{\rho_\mathrm{H}}$ is fully faithful and its image is stable under truncations.
\end{thm}
\begin{proof}
  By \'etale descent and \cref{algfields}, it suffices to prove the claim over finite type $\C$-schemes.
  Because $\D_\mathrm{H}(-,\Z)$ is compactly generated, the right adjoint $\bm{\rho_*^\mathrm{H}}$ commutes with rationalisation, so that the algebra $\pi_X^*\bm{\hcal}_{\Spec(\C)}$, tensored with $\Q$ is the algebra considered by Drew in \cite{drewMotivicHodgeModules2018} and the second author in \cite{SwannMHM}. This implies that $\bm{\underline{\rho_\mathrm{H}}}\otimes\mathrm{Mod}_\Q$ is fully faithful with image stable under truncations by the main theorem of \cite[Section 2]{SwannMHM} (see also\cite[Theorem 1.98]{ayoubAnabelianPresentationMotivic2022}). When tensoring the square of \cref{defDH} by $\mathrm{Mod}_{\Z/n\Z}$, we obtain an equivalence 
  \[\D_\mathrm{H}(-,\Z)\otimes_{\Mod_\Z}\Mod_{\Z/n\Z} \xrightarrow{\sim} \In(\dconsan(X^\mathrm{an},\Z)\otimes_{\Perf_\Z}\Perf_{\Z/n\Z}),\] but the right-hand side coincides with the derived category of \'etale sheaves of $\Z/n\Z$-modules on $X$, thanks to Artin comparison theorem. By rigidity, this implies that the natural map \[\Z/n\Z\to \pi_X^*\bm{\hcal}_{\Spec(\C)}/n\] is an equivalence, and then that $\bm{\underline{\rho_\mathrm{H}}}\otimes_{\Mod_\Z}\Mod_{\Z/n\Z}$ is an equivalence. This finishes the proof of the full faithfulness.

  Now let $M$ be in the image of $\bm{\underline{\rho_\mathrm{H}}}$. Because the t-structure is compatible with filtered colimits (this is the case for the t-structure of all categories in the pullback defining $\D_\mathrm{H}$), we have an exact triangle 
  \[\tau^{\leqslant 0}M\to \tau^{\leqslant 0}(M\otimes_\Z \Q) \to \colim_n \tau^{\leqslant 0}(M)\otimes_\Z \Z/n\Z.\]
  The middle term is in the image because we proved the theorem with rational coefficients in \cite{SwannMHM}. Thus it suffices to show that for all $n\in\N^*$, the object $\tau^{\leqslant 0}(M)\otimes_\Z\Z/n\Z$ is in the image. This is the case because $\bm{\underline{\rho_\mathrm{H}}}\otimes_{\Mod_\Z}\Mod_{\Z/n\Z}$ is an equivalence. This proves that the image is stable under truncations, finishing the proof.
\end{proof}

\begin{rem}
  The fact that $\bm{\underline{\rho}_\mathrm{H}}$ is fully faithful implies readily that the map $\pi_X^*\bm{\hcal}_{\Spec(\R)}\to\bm{\hcal}_X$ is an equivalence for all $X$.
\end{rem}

\begin{rem}
  As explained at the beginning of this subsection, the above theorem also works for the non enriched version of the Hodge realisation of \'etale motives. We will call those Hodge modules of geometric origin, and denote them by $\mathrm{DH}(-,\Z)$. 
\end{rem}

\begin{rem}
  \label{remGen}
  For $X$ a finite type $\R$-scheme, the image of $\bm{\rho_\mathrm{H}}$ can be 
  described as the localising subcategory 
  spanned by objects of the form $p_*H(n)$ for $p\colon Y\to X$ proper, and $H$ a polarisable integral mixed Hodge structure over $\R$ seen as a constant integral mixed Hodge module over $Y$.
\end{rem}

\subsection{Abelian categories of integral Hodge modules}

By definition and by \cref{theoMHMmod} for each finite type $\R$-scheme $X$ the three categories $\D_\mathrm{H}(X,\Z)$, $\mathrm{DH}(X,\Z)$ and $\bm{\mathrm{DH}}(X,\Z)$ admit a perverse t-structure, that restrict to constructible object  $\D_\mathrm{H,c}^b(X,\Z)$, $\mathrm{DH}_c(X,\Z)$ and $\bm{\mathrm{DH}}_c(X,\Z)$.
As above, the two categories $\mathrm{DH}_c(X,\Z)$ and $\bm{\mathrm{DH}}_c(X,\Z)$ are the thick categories spanned by the generators of \cref{remGen}.

We will denote by $\mathrm{MHM}(X,\Z)$, $\mathrm{MHM}_\mathrm{geo}(X,\Z)$ and $\mathrm{MHM}_\mathrm{Hdg}(X,\Z)$ the three hearts of the constructible objects. 

In more concrete terms, a  mixed Hodge module with integral coefficients is the data of three objects $(M,\fcal,S,\varphi)$ where $M\in\mathrm{MHM}(X_\C,\Q)$ is a mixed Hodge module as constructed by M. Saito, $\fcal\in\mathrm{Perv}(X_\C^\mathrm{an},\Z)$ is a perverse sheaf, $S\in\operatorname{End}(M,\fcal)$ is an involution and $\varphi\colon \fcal\otimes_\Z \Q\simeq \mathrm{rat}(M)$ is an isomorphism of perverse sheaves that commutes with $S$.  

Because we have the $6$ operations, we can also construct ordinary t-structures on the constructible objects by gluing along stratifications, doing the same construction as in the proof of \cref{ExistPerv} in reverse. We will denote their hearts by $\mathrm{MHM}^\mathrm{ord}(X,\Z)$, $\mathrm{MHM}_\mathrm{geo}^\mathrm{ord}(X,\Z)$ and $\mathrm{MHM}_\mathrm{Hdg}^\mathrm{ord}(X,\Z)$ respectively. Using the same proof as \cref{DbNat}, we obtain:

\begin{prop}
  \label{Dbmhm}
  The natural functors 
  \[\D^b(\mathrm{MHM}^\mathrm{ord}(X,\Z))\to \D_\mathrm{H,c}^b(X,\Z),\] 
  \[\D^b(\mathrm{MHM}_\mathrm{geo}^\mathrm{ord}(X,\Z))\to\mathrm{DH}_c(X,\Z)\] and \[\D^b(\mathrm{MHM}_\mathrm{Hdg}^\mathrm{ord}(X,\Z))\to\bm{\mathrm{DH}}_c(X,\Z)\] are equivalences.
\end{prop}
% \begin{cor} 
%   Over finite type $\C$-schemes, the natural functors 
%   \[\D(\In\mathrm{MHM}^\mathrm{ord}(X,\Z))\to \D_\mathrm{H}(X,\Z),\] 
%   \[\D(\In\mathrm{MHM}_\mathrm{geo}^\mathrm{ord}(X,\Z))\to\mathrm{DH}(X,\Z)\] and \[\D(\In\mathrm{MHM}_\mathrm{Hdg}^\mathrm{ord}(X,\Z))\to\bm{\mathrm{DH}}(X,\Z)\] are equivalences.
% \end{cor}
% \begin{proof}
%   Indeed the hearts of the compact objects have finite cohomological dimension (by the same argument as for the perverse hearts), so that the corollary follows by taking indization of the theorem.
% \end{proof}

% \begin{rem}
%   The above result is false for ordinary motives over $\C$-schemes. Indeed, assuming that the functor $\DM(X,\Lambda)\to\DN(X,\Lambda)$ is an equivalence, Ayoub's counterexample \cite[Lemma 2.4]{MR3751289} would imply that the functor $\DN(X,\Lambda)\to\D(X^\mathrm{an},\Lambda)$ is not conservative, but the functor $\D(\In\mcal_\mathrm{ord}(X,\Lambda))\to\D(X^\mathrm{an},\Lambda)$ \emph{is} conservative. This is the only obstruction, see \cref{DIndQ}.
% \end{rem}
\begin{rem}
  As for Nori motives, we suspect that the functors 
  \[\D^b(\mathrm{MHM}_\mathrm{geo}(X,\Z))\to\mathrm{DH}_c(X,\Z)\] 
  and 
  \[\D^b(\mathrm{MHM}_\mathrm{Hdg}(X,\Z))\to\bm{\mathrm{DH}}_c(X,\Z)\] are equivalences, because in some sense, objects of geometric origin have a tendency to be flat, or at least, resolved by flat objects. 
\end{rem}

We finish the paper by linking our notion of mixed Hodge modules to the classical notion of \emph{variation of mixed Hodge structure}. We say that an integral mixed Hodge module $M \in \mathrm{MHM}(X,\Z)$ is lisse if its underlying perverse sheaf $\mathrm{int}(M)$ is locally constant. We denote by $\mathrm{MHM}_\mathrm{lis}(X,\Z)$ the category of lisse Mixed Hodge modules on a finite-type $\C$-scheme. Note that the pullback defining $\D_\mathrm{H}(X,\Z)$ implies that the perverse t-structure induces a t-structure on the subcategory of dualisable objects of $\D_\mathrm{H}(X,\Z)$; the heart of this t-structure is precisely $\mathrm{MHM}_\mathrm{lis}(X,\Z)$. 

\begin{prop}\label{VMHS}
  Let $X$ be a finite-type $\C$-scheme and let $\mathrm{VMHS}_\mathrm{ad}(X,\Z)$ be the category of admissible variations of mixed Hodge structures, where a variation of mixed Hodge structure in the sense of \cite[1.8.14]{MR0601520} is admissible if its rationalisation is admissible in the sense of \cite[2.1]{zbMATH00015320}. Then, we have an equivalence of categories \[\mathrm{MHM}_\mathrm{lis}(X,\Z)\to \mathrm{VMHS}_\mathrm{ad}(X,\Z).\]
\end{prop}
\begin{proof}
  We have a pullback square:
\[\begin{tikzcd}
  \mathrm{MHM}_\mathrm{lis}(X,\Z)\ar[r]\ar[d] & \mathrm{Loc}(X^\an,\Z) \ar[d]\\
  \mathrm{MHM}_\mathrm{lis}(X,\Q) \ar[r]& \mathrm{Loc}(X^\an,\Q)
\end{tikzcd}\] where $\mathrm{Loc}(X,\Lambda)$ is the category of locally constant sheaves of $\Lambda$-modules with values a in finitely generated $\Lambda$-modules. 

By \cite[Theorem 2.2]{zbMATH00015320}, we have an equivalence of categories \[\mathrm{MHM}_\mathrm{lis}(X,\Q)\to \mathrm{VMHS}_\mathrm{ad}(X,\Q).\]
But by definition, a VMHS with integral coefficients is exactly a VMHS with rational coefficients admitting an integral lattice. Hence, we also have a pullback diagram:
\[\begin{tikzcd}
  \mathrm{VMHS}_\mathrm{ad}(X,\Z)\ar[r]\ar[d] & \mathrm{Loc}(X^\an,\Z)\ar[d] \\
  \mathrm{VMHS}_\mathrm{ad}(X,\Q) \ar[r]& \mathrm{Loc}(X^\an,\Q),
\end{tikzcd}\]
and the result follows. 
\end{proof}

\begin{rem}
  One could go further and define the usual operations of mixed Hodge modules, but with integral coefficients. For example if $f\colon X\to\A^1_\C$ is a map of finite type $\C$-schemes, one can define nearby cycles $\Psi_f$ associated to $f$ as the functor 
  \[ \Psi_f\colon \D^b_\mathrm{H,c}(X_\eta,\Z)\to\D^b_\mathrm{H,c}(X_s,\Z),\]
  with $X_\eta = f^{-1}(\Gm)$ and $X_s = f^{-1}(\{0\})$, defined by the pullback property:
  \[\begin{tikzcd}
	{\D_\mathrm{H,c}(X_\eta,\Z)} \\
	& {\D_\mathrm{H,c}(X_s,\Z)} & {\D^b_c(X_s,\Z)} \\
	& {\D_\mathrm{H,c}(X_s,\Q)} & {\D^b_c(X_s,\Z)}
	\arrow["{\Psi_f}"{description}, dashed, from=1-1, to=2-2]
	\arrow["{\Psi_f(\mathrm{int}(-))}", from=1-1, to=2-3,bend left=20]
	\arrow["{\Psi_f(-\otimes_\Z \Q)}"', from=1-1, to=3-2,bend right=20]
	\arrow[from=2-2, to=2-3]
	\arrow[from=2-2, to=3-2]
	\arrow["\lrcorner"{anchor=center, pos=0.125}, draw=none, from=2-2, to=3-3]
	\arrow[from=2-3, to=3-3]
	\arrow[from=3-2, to=3-3]
\end{tikzcd},\] so that all known properties about $\Psi_f$ on $\D^b_c(-,\Z)$ remain true for the one for $\D^b_{\mathrm{H,c}}(-,\Z)$, for example t-exactness or compatibility with duality. Similarly, one can define vanishing cycles, monodromic mixed Hodge modules and have a Thom-Sebastiani formula by pullback. Taking $\Z/2\Z$-invariants, we also have a version over finite type $\R$-schemes.
\end{rem}

\bibliographystyle{alpha}
\bibliography{BibSCNet}

\vfill

\begin{tabular}{ll}
    (Rapha\"el Ruimy) &Universit\'e Grenoble Alpes, Institut Fourier \\
    &F-38610 Grenoble, France. \\
    &Email adress: \url{raphael.ruimy@univ-grenoble-alpes.fr} \\ \\

    (Swann Tubach) &Sorbonne Universit\'e and Universit\'e Paris Cit\'e, CNRS, IMJ-PRG, \\
    &F-75005 Paris, France. \\
    &Email adress: \url{tubach@imj-prg.fr}
\end{tabular}

\end{document}